\documentclass[reqno,11pt]{amsart}
 
\usepackage{amsmath}
\usepackage{amsthm}  
\usepackage{verbatim}
\usepackage{enumerate} 
\usepackage{mathtools}  
\usepackage{amssymb}
\usepackage{mathrsfs}
\usepackage{mathabx}
\usepackage[top=1.5in, bottom=1.5in, left=0.7in, right=0.7in]{geometry}  
\usepackage[colorlinks]{hyperref}
\hypersetup{
	colorlinks,
	citecolor=blue,
	linkcolor=red
}   
\numberwithin{equation}{section}
\usepackage{xypic}
\usepackage{bm}
\usepackage{tikz}
\usetikzlibrary{decorations.pathreplacing}
\usepackage{float}

\newtheorem{theorem}{\textbf{Theorem}}[section]
\newtheorem{theorem*}{\textbf{Theorem}}
\newtheorem{thmx}{Theorem}
 
\newtheorem{corollaryx}[thmx]{\textbf{Corollary}}
\newtheorem{definition}[theorem]{\textbf{Definition}}
\newtheorem{proposition}[theorem]{\textbf{Proposition}}
\newtheorem{lemma}[theorem]{\textbf{Lemma}}
\newtheorem{question}[theorem]{\textbf{Question}}

\newtheorem{corollary}[theorem]{\textbf{Corollary}}
\newtheorem{remark}[theorem]{\textbf{Remark}}

\newtheorem{example}[theorem]{\textbf{Example}}
\newtheorem{conjecture}[theorem]{\textbf{Conjecture}}
\newtheorem{definition/proposition}[theorem]{\textbf{Definition/Proposition}}
\newtheorem{innercustomgeneric}{\customgenericname}
\providecommand{\customgenericname}{}
\newcommand{\newcustomtheorem}[2]{%
	\newenvironment{#1}[1]
	{%
		\renewcommand\customgenericname{#2}%
		\renewcommand\theinnercustomgeneric{##1}%
		\innercustomgeneric
	}
	{\endinnercustomgeneric}
}
\newcustomtheorem{customthm}{Theorem}
\newcustomtheorem{customprop}{Proposition}
\newcustomtheorem{customcoro}{Corollary}

\def\A{{\mathbb A}}
\def\N{{\mathbb N}}
\def\R{\mathbb{R}}
\def\Z{{\mathbb Z}}

\def\C{{\mathbb C}}
\def\D{{\mathbb D}}

\def\Q{{\mathbb Q}}

\newcommand{\CP}{\mathbb{C}\mathbb{P}}

\def\i{\iota}
 
\def\x{\xi}

\def\z{\zeta}

\def\cA{{\mathcal A}}

\def\cC{{\mathcal C}}

\def\cH{{\mathcal H}}

\def\cJ{{\mathcal J}}

\def\cM{{\mathcal M}}
\def\cN{{\mathcal N}}

\def\cP{{\mathcal P}}

\def\cR{{\mathcal R}}
\def\cS{{\mathcal S}}

\def\bB{{\bm B}}

\def\bH{{\bm H}}

\def\bJ{{\bm J}}

\def\rD{{\rm D}}
\def\rSD{{\rm SD}}

\def\rd{{\rm d}}

\def\la{\langle\,}
\def\ra{\,\rangle}
\def\st{\: \big| \:}

\DeclareMathOperator{\Ima}{im}
\DeclareMathOperator{\ind}{ind}

\DeclareMathOperator{\Id}{id}

\DeclareMathOperator{\coker}{coker}

\DeclareMathOperator{\rank}{rank}

\DeclareMathOperator{\End}{End}

\title{Symplectic fillings of asymptotically dynamically convex manifolds II--$k$-dilations}
\author{Zhengyi Zhou}

\begin{document}
	\maketitle
\begin{abstract}
	We introduce the concept of $k$-(semi)-dilation for Liouville domains, which is a generalization of symplectic dilation defined by Seidel-Solomon. We prove that the existence of $k$-(semi)-dilation is a property independent of fillings for asymptotically dynamically convex (ADC) manifolds. We construct examples with $k$-dilations, but not $k-1$-dilations for all $k\ge 0$. We extract invariants taking value in $\N \cup \{\infty\}$ for Liouville domains and ADC contact manifolds, which are called the order of (semi)-dilation. The order of (semi)-dilation serves as embedding and cobordism obstructions. We determine the order of (semi)-dilation for many Brieskorn varieties and use them to study cobordisms between Brieskorn manifolds.
	
\end{abstract}
\tableofcontents
\section{Introduction}\label{s1}
One fundamental question in symplectic topology is understanding the symplectic cobordism category. As a field theory for the symplectic cobordism category, symplectic field theory (SFT) outlined by Eliashberg-Givental-Hofer \cite{eliashberg2000introduction} is a powerful tool. However, the full SFT is an enormous algebra with rich structures and often far too complicated to fully compute. Hence the task is identifying smaller pieces of structures of SFT that enjoy the field theory property and are relatively easy to compute in certain non-trivial examples. Then one can derive applications in symplectic topology from them. One success along this line is the algebraic torsion introduced by Latschev-Wendl \cite{latschev2011algebraic}, which puts a hierarchy on contact manifolds without fillings. Latschev-Wendl then found a sequence of contact $3$-manifolds with increasing algebraic torsion, see also \cite{moreno2019algebraic} for higher dimensional examples. Our goal here is to introduce two numerical invariants, which we call the order of dilation and the order of semi-dilation, to put a hierarchy on Liouville domains. Moreover, following \cite{filling}, those invariants are independent of certain fillings, if the contact boundary satisfies an index condition called asymptotically dynamically convex (ADC). Hence they can be used to give obstructions to cobordisms. 

The analytical foundations of SFT are currently under development by Hofer-Fish-Wysocki-Zehnder \cite{SFT,hofer2017polyfold}. Unlike the algebraic torsion, which requires the full construction of SFT, we will use symplectic cohomology and its $S^1$-equivariant version developed in \cite{bourgeois2016s,seidel2006biased}. It has the benefits of being well-defined for all Liouville domains and its relation to SFT was studied in \cite{bourgeois2009exact} when transversality holds. Following the work of Ganatra \cite{ganatra2019cyclic}, the $S^1$-equivariant symplectic cohomology has an isomorphic open string description by the cyclic homology of the wrapped Fukaya category for non-degenerate Liouville domains including all Weinstein domains. Then the $k$-dilation in this paper can also be extracted from the wrapped Fukaya category. In fact, the existence of $k$ dilation for some $k$ is equivalent to that the canonical smooth Calabi-Yau structure defined in \cite[Theorem 3]{ganatra2019cyclic} is exact by the recent work of Li \cite{li2019exact}, see Remark \ref{rmk:li} for more discussion.

\subsection{$k$-dilations}
For a Liouville domain $W$, by considering the Hamiltonian-Floer theory of a quadratic Hamiltonian, one can construct a cochain complex $(C,\delta^0)$. If we include the $S^1$-action on the free loop space into consideration, the theory can be enhanced into an equivariant theory as outlined in \cite{seidel2006biased}. One such construction endows an $S^1$-structure $(\delta^0,\delta^1,\ldots)$ on $(C,\delta^0)$. Then the equivariant symplectic cohomology is the cohomology of $(C[u,u^{-1}]/u,\delta^{S^1}:=\sum u^i\delta^i)$, where $u$ is a formal variable of degree $2$. By considering the associated $u$-adic filtration, we say $W$ admits a $k$-dilation iff the unit $1$ is killed in the $k+1$th page of the spectral sequence. Therefore the vanishing of symplectic cohomology is equivalent to the existence of $0$-dilation and the symplectic dilation introduced by Seidel-Solomon \cite{seidel2012symplectic} is equivalent to $1$-dilation. Since the existence of $k$-dilation implies the existence of $k+1$-dilation by definition,  we can introduce the order of dilation as the minimal number $k$ for which $W$ admits a $k$-dilation. Since the Viterbo transfer map preserves such structure, i.e.\ if we have a subdomain $V\subset W$, then $V$ admits a $k$-dilation if $W$ admits a $k$-dilation, the order of dilation provides obstructions to embeddings of Liouville domains just like the case for $0$ and $1$-dilations, e.g.\ any flexible Weinstein domain does not contain an exact Lagrangian and any Liouville domain with symplectic dilation does not contain exact Lagrangian torus \cite{seidel2012symplectic}. To illustrate the concept of $k$--dilation is a non-trivial extension of symplectic dilation, we provide basic examples in the following theorem.
\begin{thmx}\label{thm:A}
	The Milnor fiber of the singularity $x_0^k+x_1^k+\ldots+x_m^k=0$ admits a $k-1$-dilation but not a $k-2$ dilation for $m\ge k$ over $\Q$.
\end{thmx}
The $k=m=1$ case is $\C$ and $k=m=2$ case is $T^*S^2$, which are known to have vanishing symplectic cohomology and symplectic dilation respectively. The existence of $k$-dilation is preserved under product, Lefschetz fibration and flexible surgery under mild conditions. We also define $k$-semi-dilation as a generalization of $k$-dilation. $k$-semi-dilation enjoys the same property as $k$-dilation in the study of symplectic embeddings and $k$-semi-dilation is preserved in  Lefschetz fibration without any condition. Moreover, $k$-semi-dilation is more natural from a SFT perspective and can be generalized to contact manifolds coupled with general algebraic augmentations \cite{moreno2020landscape}. Starting from Theorem \ref{thm:A}, along with the embedding property and Lefschetz fibration,  we have many examples with $k$-dilations. In particular, we determine the order of (semi)-dilation for many Brieskorn varieties and find many new examples of Liouville domains with symplectic dilations, which are not Lefschetz fibration built from $T^*S^2$ \cite{seidel2012symplectic} and plumbings of $T^*S^2$ according to $A_m$ or $D_m$ diagrams \cite{etgu2017koszul}. 

\begin{remark}\label{rmk:li}
	Li \cite{li2019exact} defined the concept of cyclic dilation independently, which is closely related to the notion of $k$-dilation. A cyclic dilation is an element $x\in SH^{1}_{S^1}(W)$, such that $\bB(x)=h$ where $\bB:SH^{1}_{S^1}(X)\to SH^0(W)$ is the map in the Gysin exact sequence \cite{bourgeois2016s} and $h$ is an invertible element in $SH^0(W)$. When $h=1$, existence of cyclic dilation is equivalent to existence of $k$-dilation for some $k$, see Remark \ref{rmk:cyclic}. A similar structure called higher dilation was introduced in \cite{zhao2016periodic}, which implies cyclic dilation for $h=1$ but not the other way around. \cite{li2019exact} studied the open string implication when the Liouville domain supports a cyclic dilation and the $k=m=3$ case of Theorem \ref{thm:A} was also obtained there. Theorem \ref{thm:A} provides more non-trivial examples with cyclic dilation. $k$-dilation can be viewed as a refinement of cyclic dilation for $h=1$ and the quantity $k$ plays an important role for the purpose of this paper. 
\end{remark}

\begin{remark}
	The Gutt-Hutchings capacity defined in \cite{MR3868228} is also closely related. The first capacity $c^{GH}_1(W)$ is defined to the spectral invariant of the primitive of $1\in SH^*_{S^1}(W)$,  in other words, the spectral invariant associated to the element $x$ in the positive $S^1$-equivariant symplectic cohomology $SH^*_{S^1,+}(W)$ that is mapped to $1\in H^*_{S^1}(W)$ in the tautological long exact sequence. In particular, a Liouville domain has finite $c^{GH}_1$ iff $W$ admits a cyclic dilation for $h=1$ or equivalently a $k$-dilation for some $k$. The higher Gutt-Hutchings capacity $c^{GH}_i$ is the minimum of spectral invariants of elements in $SH^*_{S^1,+}(W)$ that are mapped to $u^{-i+1}+\sum_{j=0}^{i-2}u^{-j}a_j\in H^*_{S^1}(W)$, or schematically the spectral invariant of ``$u^{-i+1}x$" for $x\in SH^*_{S^1,+}(W)$ mapped to $1\in H^*_{S^1}(W)$.	
	
	By Proposition \ref{prop:order}, having a $k$-dilation but not a $k-1$ dilation implies that $u^ix\ne 0$ for $i\le k$, where $x\in SH^*_{S^1,+}(W)$ is mapped to $1\in H^*_{S^1}(W)$. Then we can define $c^{-i+1}_{GH}(W)$ as the spectral invariant of $u^ix$ for $i\le k$ and the increasing property of \cite[Theorem 1.1.]{MR3868228} still holds. Therefore the order of dilation $k$ is to what extent Gutt-Hutchings capacity can be defined in the negative direction. On the other hand, the maximal extent such that the Gutt-Hutchings capacity can be defined in the positive direction defines another numerical invariant by the recent work of Ganatra-Siegel \cite[Definition 3.1, Theorem 3.2]{ganatra2020embedding}, which can also be used to obstruct exact embeddings.
\end{remark}

The existence of $k$-(semi)-dilation is a uniruled condition, as already shown in \cite{filling} for the vanishing and dilation case. Therefore for algebraic varieties with non-negative log-Kodaira dimension, we also have the non-existence result.

\subsection{Persistence of $k$-dilations}
Asymptotically dynamically convex (ADC) manifolds were introduced by Lazarev \cite{lazarev2016contact} as a generalization of index positive contact manifold in \cite{cieliebak2018symplectic}. The key property is that the SFT algebra of an ADC manifold is supported in positive gradings (asymptotically). Examples of ADC contact manifolds are boundaries of flexible domains \cite{lazarev2016contact}, cotangent bundles $T^*M$ for $\dim M \ge 4$ and links of terminal singularities \cite{mclean2016reeb}.  We proved in \cite{filling} that vanishing of symplectic cohomology and existence of symplectic dilation are properties independent of fillings for ADC manifolds. Since the existence of $k$-dilation is the structural generalization of symplectic dilation, we prove the similar independence property in this paper.

More precisely, we construct, from the $u$-adic spectral sequence, a sequence of structural maps $\Delta^k_+:\ker \Delta^{k-1}_+\to \coker \Delta^{k-1}_{+}$ and $\Delta^{k}_{\partial}:\ker \Delta_+^{k} \to \coker \Delta^{k-1}_{\partial}$, with $\Delta^0_+=0:SH_+^*(W)\to SH_+^*(W)$ and $\Delta^{0}_{\partial}$ is the composition $SH_+^*(W)\to H^{*+1}(W)\to H^{*+1}(Y)$.  We call a filling $W$ of $Y$ topologically simple iff $c_1(W)=0$ and $\pi_1(Y)\to \pi_1(W)$ is injective.
\begin{thmx}\label{thm:B}
	If $Y$ is ADC contact manifold, then $\Delta^k_+,\Delta^k_{\partial}$ are independent of topologically simple Liouville fillings.
\end{thmx}

The precise statement of the independence can be found in Theorem \ref{thm:ind}. The existence of $k$-semi-dilation is equivalent to $1\in \Ima \Delta^k_{\partial}$ when the first Chern class vanishes, and the existence of $k$-dilation is closely related. Therefore we obtain the persistence of $k$-dilations as follows.
\begin{corollaryx}\label{cor:C}
	If $Y$ is an ADC contact manifold, then
	\begin{itemize}
		\item the existence of $k$-semi-dilation is a property independent of topologically simple Liouville fillings $W$.
		\item the existence of $k$-dilation is independent of Weinstein fillings $W$ if $\dim Y \ge 4k+1$ or $\frac{\dim W}{2}$ is odd.
	\end{itemize}
\end{corollaryx}

Philosophically, the contact homology algebra of an ADC contact manifold admits no non-trivial $\Z$-graded augmentation, since it is supported in positive grading (asymptotically). Since augmentations can be understood as the algebraic analogue of fillings, the theme of Theorem \ref{thm:B} can be summarized as fillings of contact manifolds that admit no non-trivial augmentation have some rigidity. In particular, it is possible to reformulate Theorem \ref{thm:B} using (possibly a variant of) SFT, and the result can be improved to contact manifolds with only trivial augmentation. The exactness condition in Theorem \ref{thm:B} is technical, while the $c_1(W)=0$ is essential since ADC is an index property. However, with the current setup using Hamiltonians, the strong filling version of Theorem \ref{thm:B} only holds for tamed ADC contact manifolds due to the shrinking issue explained in \cite[\S 8]{filling}.

The structural maps $\Delta^k_{\partial}$ factor through the cohomology of the filling $W$, as a consequence, if $Y$ is ADC and $\Ima \Delta^k_{\partial}$ contains an element of degree $>\frac{\dim W}{2}$, then there is no Weinstein filling of $Y$. Cases of $k=0,1$ were considered in \cite{filling} and proven to be symplectic in nature. Similarly, $\Delta^k_{\partial}$ can be used to develop obstructions to cobordisms. A special case of such results is the following.
\begin{corollaryx}\label{cor:D}
	Let $Y_1,Y_2$ be two simply connected contact manifolds with topologically simple exact fillings $W_1,W_2$. We define  the order of semi-dilation  $\rSD(W):=\min(\{k|\text{$W$ admits a $k$-semi-dilation}\}\cup \{\infty\})$. If $Y_1$ is ADC and  $\rSD(W_2)> \rSD(W_1)$, then there is no exact cobordism from $Y_2$ to $Y_1$ with vanishing first Chern class.
\end{corollaryx}

\subsection*{Organization of the paper}
In \S \ref{s2}, we explain the algebraic preliminaries of $S^1$-structures and define $k$-(semi)-dilation in the algebraic setting. In \S \ref{s3}, we build the $S^1$-equivariant cochain complexes following \cite{zhao2014periodic}. We define $k$-(semi)-dilations for Liouville domains and study basic applications of them in \S \ref{s3}. We prove Theorem \ref{thm:B} and corollaries of it in \S \ref{s4}, the argument is completely analogous to \cite{filling} after appropriate setup. Theorem \ref{thm:A} is proven in \S \ref{s6} and we study the cobordism problem between Brieskorn manifolds.

\subsection*{Acknowledgments}
The author is supported  by  the National Science Foundation under Grant No. DMS-1638352.  It is a great pleasure to acknowledge the Institute for Advanced Study for its warm hospitality. The author is grateful to Yaim Cooper and Samuel Lisi for answering various questions, and to the referee for numerous helpful comments and suggestions. This paper is dedicated to the memory of Chenxue.
\section{Algebra preliminary}\label{s2}
In this section, we discuss the algebraic properties of $S^1$-cochain complexes using the explicit model as in \cite{zhao2014periodic} and the abstract definition of $k$-dilation. For simplicity, we will use a field coefficient $\bm{k}$ with a default setting of $\Q$ throughout this paper. Most of the paper works for any ring coefficient, with exception of the (non-conical) reconstruction from spectral sequence in Remark \ref{rmk:SS},\ref{rmk:SSphi}, the K{\"u}nneth formula in Proposition \ref{prop:prod} and the computation on Brieskorn manifolds in \S  \ref{s52}.
\subsection{$S^1$-complexes}
\begin{definition}{\cite[Definition 2.1]{zhao2014periodic}}
	Let $(C,\delta_0)$ be a $\Z$ (or $\Z_2$) graded cochain complex over $\bm{k}$. An $S^1$-structure on $(C,\delta_0)$ is given by a sequence of maps 
	$$\delta:=(\delta^0,\delta^1,\ldots), \quad \delta^i: C^* \to C^{*+1-2i}$$
	such that the relation 
	$$\sum_{i+j=k} \delta^i\circ \delta^j=0$$
	holds for every $k\ge 0$.
\end{definition}

Given an $S^1$-complex $(C,\delta)$, we have the following three cochain complexes (viewed as $\bm{k}[u]$ modules) for the equivariant theory 
\begin{eqnarray*}
	C^- & = & \varprojlim C[u]/u^pC[u]=C[[u]];\\
	C^{\infty} & = & u^{-1}(C^-)=C((u));\\
	C^+ & = & C^{\infty}/uC^{-},
\end{eqnarray*}
with differential $\delta^{S^1}:=\sum_{i=0}^\infty u^i\delta^i$. More explicitly, $C^-$ consists of power series in $u$ with coefficient in $C$  and $C^{\infty}$ consists of Laurent series in $u$ with coefficient in $C$. Then $C^+$ as a $\bm{k}$ vector space is isomorphic to $C[u^{-1}]$. In other words, $C^+$ is the space of polynomials of $u^{-1}$ with coefficients in $C$ equipped with a $\bm{k}[u]$-module structure given by $u\cdot (c_0+c_1u^{-1}+\ldots+c_nu^{-n})=c_1+c_2u^{-1}+\ldots+c_nu^{-n+1}$. We say $\delta=(\delta^0,\ldots, \delta^N)$ is a $N$-truncated $S^1$-structures iff 
$$\sum_{i+j=k}\delta^i\circ \delta^j=0, \quad \forall k \le N.$$ 
Then for a $N$-truncated $S^1$-structure, $\delta^{S^1}$ also defines a differential on $F^NC^+:=C\otimes \la 1,u^{-1},\ldots,u^{-N}\ra $. Given an $S^1$-stricture,  we have a filtration of cochain complexes $C=F^0C^+\subset F^1C^+\subset \ldots \subset F^NC^+ \subset \ldots \subset C^+$ with $C^+=\varinjlim F^NC^+$. Obviously, every $S^1$-structure induces a $N$-truncated structure by truncation.

\begin{definition}[{\cite[Definition 2.2]{zhao2014periodic}}]
	Let $(C,\delta),(D,\partial)$ be two $S^1$-complexes, an $S^1$-morphism between them is a sequence of maps
	$$(\phi^0,\phi^1,\ldots), \quad \phi^i:C^*\to D^{*-2i},$$
	such that 
	$$\sum_{i+j=k}\phi^i\circ \delta^j - \partial^j\circ \phi^i = 0.$$
	Similarly, given two $S^1$-morphisms $\phi,\tilde{\phi}$, an $S^1$-homotopy between them is a sequence of maps
	$$(h^0,h^1,\ldots), \quad h^i:C^* \to D^{*-2i-1},$$
	such that
	$$\phi^k-\tilde{\phi}^k=\sum_{i+j=k}h^i\circ \delta^j+\partial^j\circ h^i.$$
	Similarly, we have the $N$-truncated version if we only have the first $N+1$ maps and the first $N+1$ relations.
\end{definition}
Similar to the definition of $\delta^{S^1}$, we have $\phi^{S^1}$ and $h^{S^1}$, which are morphism and homotopy defined on cochain complex $(C^+,\delta^{S^1})$. The composition $\tilde{\phi}\circ \phi$ of two $S^1$-morphisms is defined as $(\tilde{\phi}\circ \phi)^k:=\sum_{i+j=k}\tilde{\phi}^i\circ \phi^j$.

Next we assume that $C$ is a free $\bm{k}$-module and splits into direct sums of free $\bm{k}$-modules $C_0$ and $C_+$, such that we have a short exact sequence of cochain complexes $0\to C_0\stackrel{\iota}{\to} C \stackrel{\pi}{\to} C_+ \to 0$ for the first differential $\delta^0$.
\begin{definition}
	We say an ($N$-truncated) $S^1$-structure is compatible with the splitting iff $\delta^i|_{C_0}=0$ for $i\ge 1$.
\end{definition}
If the $S^1$-structure is compatible with the splitting, then $\pi\circ \delta^r|_{C_0}=0$. Hence $\delta^r$ has a decomposition $\delta^r_0+\delta^r_++\delta^r_{+,0}$ with $\delta^r_0:C_0\to C_0,\delta^r_+:C_+\to C_+$ and $\delta^r_{+,0}:C_+\to C_0$. Then $C_0$  is equipped with an $S^1$-structure $(\delta^0_0,0,\ldots)$ and $C_+$ is equipped with an $S^1$-structure $(\delta^0_{+},\delta^1_+,\ldots)$. Moreover, $(\iota,0,\ldots), (\pi,0,\ldots)$ define $S^1$-morphisms from $C_0$ to $C$ and from $C$ to $C_+$ respectively. Then we have the following short exact sequence
\begin{equation}\label{eqn:short}
0\to C^+_0 \stackrel{\iota^{S^1}}{\longrightarrow} C^+ \stackrel{\pi^{S^1}}{\longrightarrow} C^+_+ \to 0.
\end{equation}
The connecting map $(\delta^{0}_{+,0},\delta^1_{+,0},\ldots)$ defines an $S^1$ morphism from $C_+$ to $C_0[1]$, where $C_0[1]$ is equipped with the $S^1$-structure $(-\delta_0^0,\ldots)$.

If $(C,\delta),(D,\widetilde{\delta})$ are both splitted so that the $S^1$-structures are compatible, an $S^1$-morphism $\phi$ is compatible with splitting iff $\phi^i = \phi^i_+ +\phi^i_{+,0}+\phi^i_0$,  where $\phi^i_0=0$ for $i > 0$. In particular, $\phi^{S^1}_0$ preserves the decomposition w.r.t. $\{u^{-i}\}_{i\ge 0}$ on $H^*(C_0^+,\delta^{S^1}_0), H^*(D_0^+,\widetilde{\delta}^{S^1}_0)$. Two compatible  $S^1$-morphisms are homotopic if there is an $S^1$-homotopy $h$ with decomposition $h^i=h^i_++h^i_{+,0}+h^i_0$ (no need for $h^i_0=0$ when $i\ge 1$). Given an $S^1$-morphism $\phi$ compatible with the splitting, $\phi_0:=(\phi^i_0)_{i\ge 0}, \phi_+:=(\phi^i_+)_{i\ge 0}$ are $S^1$-morphisms from $C_0$ to $D_0$ and $C_+$ to $D_+$ respectively, and we have the following commutative diagram of short exact sequences,
$$
\xymatrix{
0 \ar[r] & C^+_0 \ar[r]\ar[d]^{\phi^{S^1}_0} & C^+\ar[r]\ar[d]^{\phi^{S^1}} & C^+_+ \ar[r]\ar[d]^{\phi^{S^1}_+} & 0\\
0 \ar[r] & D^+_0 \ar[r] & D^+ \ar[r] & D^{+}_+ \ar[r] & 0
}
$$
When two $S^1$-morphisms $\phi,\tilde{\phi}$ compatible with splittings are homotopic, then $\phi^{S^1}_0,\phi^{S^1},\phi^{S^1}_+$ are homotopic to $\tilde{\phi}^{S^1}_0,\tilde{\phi}^{S^1},\tilde{\phi}^{S^1}_+$ respectively.

\subsection{$k$-dilations}
In this part, we state the algebraic definition of $k$-dilations and $k$-semi-dilations. Assume $(C,\delta)$ is a ($N$-truncated) $S^1$-structure compatible with a splitting. We assume $H^*(C_0,\delta^0_0)$ contains an distinguished element $[e]\in H^0(C_0,\delta^0_0)$ such that we have a preferred projection $\pi_0:H^*(C,\delta^{0}_0)\to \la [e] \ra$. In the context of symplectic cohomology, $(C_0,\delta_0^0)$ will be the Morse cochain complex of a connected filling $W$ and $[e]\in H^*(W)$ will be the unit $1$ and the preferred projection is induced by the decomposition by degrees. Note that since $\delta^i_0=0$ for $i\ge 1$, we have $H^*(C^+_0,\delta^{S^1}_0)=H^*(C,\delta^0_0)\otimes \bm{k}[u,u^{-1}]/u$. In particular, $[e]$ can be viewed as a class in $H^0(C^+_0,\delta_0^{S^1})$ as well as in any truncated version $H^0(F^NC^+_0,\delta_0^{S^1})$. We also denote the induced projection $H^*(F^NC^+_0,\delta_0^{S^1})\to \la[e]\ra$ by $\pi_0$.

\begin{definition}\label{def:dilation}
	Under the assumptions above, we say the $N$-truncated $S^1$ structure $(C,\delta)$ carries a $k$-dilation for $k\le N$, if $\iota^{S^1}[e]=0\in H^*(F^kC^+,\delta^{S^1})$. We say it carries a $k$-semi-dilation iff $[e]$ is in the image of $H^*(F^kC^+_+,\delta^{S^1}_+)\stackrel{[\delta^{S^1}_{+,0}]}{\to} H^{*+1}(F^kC^+_0,\delta^{S^1}_0)\stackrel{\pi_0}{\to} \la [e]\ra$.
\end{definition}
 An instant corollary of the definition is the following.
\begin{proposition}\label{prop:hier}
	If $(C,\delta)$ carries a $k$-dilation, then it carries a $k$-semi-dilation. If $(C,\delta)$ carries a $k$-(semi)-dilation then it also carries a $k+1$-(semi)-dilation (If $C$ only has a $N$-truncated $S^1$ structure, we assume $k+1\le N$).
\end{proposition}
\begin{proof}
	By the tautological long exact sequence induced by \eqref{eqn:short}, $(C,\delta)$ carries a $k$-dilation iff $[e]$ is in the image of  $H^*(F^kC^+_+,\delta^{S^1}_+)\to H^{*+1}(F^kC^+_0,\delta^{S^1}_0)$. In particular, it carries a $k$-semi-dilation. The remaining of the claim follows from the following commutative diagram of short exact sequences, 
	$$
	\xymatrix{
		0 \ar[r] & F^{k}C^{+}_0 \ar[r]\ar[d] & F^kC^{+}\ar[r]\ar[d] & F^{k}C^+_+ \ar[r]\ar[d] & 0\\
		0 \ar[r] & F^{k+1}C^+_0 \ar[r] & F^{k+1}C^+ \ar[r] & F^{k+1}C^{+}_+ \ar[r] & 0
	}
	$$
	with $[e]$ preserved in $H^{*}(F^kC_0^{+})\to H^{*}(F^{k+1}C_0^+)$.
\end{proof}

\begin{proposition}\label{prop:order}
	Let $(C,\delta)$ be an $S^1$-complex, then $C$ carries a $k$-dilation iff there exists $x\in H^*(C^+_+,\delta^{S^1}_+)$  such that $[\delta^{S^1}_{+,0}](x)=[e]$ and $u^{k+1}x=0$. $C$ carries a $k$-semi-dilation iff there exists  $x\in H^*(C^+_+,\delta^{S^1}_+)$ such that $\pi_0\circ [\delta^{S^1}_{+,0}](x)=[e]$ and $u^{k+1}x=0$. 
\end{proposition}
\begin{proof}
	If $C$ carries a $k$-dilation, then there exists $x'\in H^*(F^kC^+_+,\delta^{S^1}_+)$ that is mapped to $[e]$. Then the same holds for $x$, which is the image of $x'$ under $H^*(F^kC^+_+,\delta^{S^1}_+)\to H^*(C^+_+,\delta^{S^1}_+)$. It is clear that $u^{k+1}x=0$. On the other hand, assume closed $x=\sum_{i=0}^Nc_i u^{-i}\in C^+_+$ is mapped to $[e]$ and $u^{k+1}x=\sum_{i=0}^{N-k-1}c_{i+k+1}u^{-i}$ is $\delta^{S^1}_+ y$ for some $y\in C_+^+$. Note that the commutator $[u^{-k-1},\delta^{S^1}_+]$ takes value in $F^kC^+_+$. Then $x'=x-\delta^{S^1}_+(u^{-k-1}y)=\sum_{i=0}^k c_i u^{-i}+[u^{-k-1},\delta^{S^1}_+](y)$, hence $x'$ can be viewed as a closed class in $F^kC^+_+$, which mapped to $C^+_+$ is cohomologous to $x$. Therefore $x'$ is also mapped to $[e]$, since $H^*(F^kC^+_0,\delta^{S^1}_0)\to H^{*}(C^+_0,\delta^{S^1}_0)$ is injective and sends $[e]$ to $[e]$. The proof for $k$-semi-dilation is similar.
\end{proof}

\begin{remark}\label{rmk:cyclic}
	In the previous version of this paper, the $k$-dilation was equivalently defined using the spectral sequence for the $u$-adic filtration as in Proposition \ref{prop:equi} below. If $\iota^{S^1}[e]=0$ in $H^*(C^+,\delta^{S^1})$, then $C$ carries a cyclic dilation (for $h=1$) introduced by Li \cite{li2019exact}. Then the existence of cyclic dilation for $h=1$ is equivalent to the existence of $k$-dilation for some $k$.
\end{remark}

\subsection{$u$-adic filtration and spectral sequences}
In the following, we introduce some structural maps from the spectral sequence induced from the $u$-adic filtration on $C^+$, some of which enjoy the invariance property for ADC contact manifolds in the context of $S^1$-equivariant symplectic cohomology and are closely related to $k$-(semi)-dilations.

\begin{definition}\label{def:Z}
	Given a $N$-truncated $S^1$-structure, for every $0\le k \le N$, we can define a subspace $Z_k\subset C$ consisting of elements $\alpha_0$, such that there exist $\alpha_1,\ldots, \alpha_k\in C$ so that the following
	$$
		\sum_{i=0}^{m} \delta^{i}(\alpha_{m-i}) = 0
	$$
	holds for every $0\le m \le k$, i.e.\ $\delta^{S^1}(\sum_{i=0}^k u^{-i}\alpha_{k-i})=0$.   Hence $Z_k$ is the space of leading terms of closed forms in $F^kC^+$. Similarly, we define $B_k \subset C$ consisting of elements $\alpha$, such that there exists $\alpha_0,\ldots \alpha_k\in C$ so that 
	$$\alpha=\sum_{i} \delta^i(\alpha_{k-i}), \quad \sum_{i} \delta^i(\alpha_{k-j-i})=0, \forall j > 0.$$
	In other words, $\delta^{S^1}(\sum_{i=0}^k u^{-i}\alpha_{k-i})=\alpha$, i.e.\ $B_k$ is the subset of $C$ that is exact in $F^kC^+$.
\end{definition}
In particular, $Z_0$ is the set of $\delta^0$-closed classes and $B_0$ is the set of $\delta^0$-exact classes.

\begin{proposition}\label{prop:ss}
	Given an $S^1$-complex, we have the following.
	\begin{enumerate}
		\item\label{ss1} $B_0\subset \ldots \subset B_k \subset \ldots \subset Z_k\subset \ldots \subset Z_0$.
		\item\label{ss2} For $k\ge 1$, $\Delta^k:Z_{k-1}/B_0\to Z_0/B_{k-1}$ given by $\Delta^k(\alpha_0)=\sum_{i=1}^k \delta^{i}(\alpha_{k-i})$, is well-defined of degree $1-2k$, where $\alpha_1,\ldots,\alpha_k$ are those in the definition of $Z_{k-1}$.
		\item\label{ss3} $\Delta^1$ is $[\delta^1]$ on $H^*(C,\delta^0)=Z_0/B_0$. Moreover,  we have $\ker\Delta^{k}=Z_{k}/B_0$, $\Ima \Delta^k=B_k/B_{k-1}$ and $\coker \Delta^k= Z_0/B_{k}$.
	\end{enumerate}
\end{proposition}
\begin{proof}
	That $B_i\subset B_j,Z_i\supset Z_j$ for $i<j$ follows from definition. Therefore to prove \eqref{ss1}, it suffices to show that $B_k\subset Z_k$. If $\alpha \in B_k$, then we have $A=\sum_{i=0}^k u^{-i}\alpha_{k-i}\in F^kC^+$, such that $\delta^{S^1}(A)=\alpha$. Since the commutator $[u^{-k},\delta^{S^1}]$ takes value in $F^{k-1}C^{+}$, we have $\delta^{S^1}(u^{-k}A) = u^{-k}\alpha+ [\delta^{S^1},u^{-k}]A\in F^kC^{+}$ and is $\delta^{S^1} $exact. Hence $\alpha \in Z_k$ by definition.
	
	Note that $\Delta^k(\alpha_0)=\delta^{S^1} (\sum_{i=1}^ku^{-i}\alpha_{k-i})$, hence $\delta^0(\Delta^k(\alpha_0))=(\delta^{S^1})^2(\sum_{i=1}^ku^{-i}\alpha_{k-i})=0$, i.e.\ $\Delta^k(\alpha_0)\in Z_0$. Now assume $\alpha'_1,\ldots,\alpha'_{k-1}$ is another family such that $\delta^{S^1}(\sum_{i=0}^{k-2} u^{-i}\alpha'_{k-1-i}+u^{-k+1}\alpha_0)=0$. Then $\delta^{S^1}(\sum_{i=0}^{k-2}u^{-i}(\alpha_{k-1-i}-\alpha'_{k-1-i}))=0$ and the difference in $\Delta^k(\alpha_0)$ is given by $\delta^{S^1}(\sum_{i=1}^{k-1}u^{-i}(\alpha_{k-i}-\alpha'_{k-i}))\in C$ since $[u^{-1},\delta^{S^1}]\in C$. Hence the difference is in $B_{k-1}$ by definition. Lastly if $\alpha_0=\delta^0\beta \in B_0$, then we can pick $\sum_{i=0}^{k-2}u^{-i}\alpha_{k-1-i}=[\delta^{S^1},u^{-k+1}]\beta$ which satisfies the conditions for $Z_{k-1}$. Then $\Delta^k(\alpha_0)$ is given by 
	\begin{eqnarray*}
	\delta^{S^1}(u^{-1}[\delta^{S^1},u^{-k+1}]\beta + u^{-k}\delta^{S^1}\beta) & = & \delta^{S^1}( u^{-1}\delta^{S^1}u^{-k+1}\beta) \\
	& = & (\delta^{S^1})^2(u^{-k}\beta)+\delta^{S^1}([u^{-1},\delta^{S^1}]u^{-k+1}\beta)\\
	& \in & \Ima \delta^{S^1}(C)=B_0. 
	\end{eqnarray*}
	Hence $\Delta^k(x)\in B_{0}\subset B_{k-1}$.
	
	It is straightforward to check that $\Delta^1$ is $[\delta^1]$ on $H^*(C,\delta^0)$. To prove $\ker \Delta^k = Z_k/B_0$, if $\alpha_0\in \ker \Delta^k$, that is there exist $\alpha_1,\ldots,\alpha_{k-1}$ and $\beta_0,\ldots,\beta_{k-1}$, such that $\delta^{S^1}(u^{-k}\alpha_0+\sum_{i=1}^{k-1}u^{-i} \alpha_{k-i})=\delta^{S^1}(\sum_{i=0}^{k-1}u^{-i}\beta_{k-1-i})\in C$. Hence $\delta^{S^1}( u^{-k}\alpha_0+\sum_{i=1}^{k-1} u^{-i}(\alpha_{k-i}-\beta_{k-1-i})-\beta_{k-1})=0$ and then $ \alpha_0\in Z_k$ by definition.  On the other hand, if $\alpha_0\in Z_k$, i.e.\ there exists $\alpha_1,\ldots,\alpha_k$ such that $\delta^{S^1}(\sum_{i=0}^ku^{-i}\alpha_{k-i})=0$. Then $\delta^{S^1}(\sum_{i=1}^{k}u^{-i}\alpha_{k-i})=-\delta^{S^1}(\alpha_k)\in B_0$. That is $\Delta^k(\alpha_0)=0$ in $Z_0/B_{k-1}$. 
	
	That $\Ima \Delta^k\subset B_k/B_{k-1}$ follows from definition. Now assume $\alpha \in B_k$, i.e. $\alpha = \delta^{S^1}(\sum_{i=0}^k u^{-i}\alpha_{k-i})$, then $\alpha-\delta^{0}(\alpha_k)$ is in the image of $\Delta^k$ by definition. Since $\delta^{0}(\alpha_k)=0\in Z_0/B_{k-1}$, then $\alpha \in \Ima \Delta^k$. Then $\coker \Delta^k=Z_0/B_k$ follows from $\Ima \Delta^k=B_{k}/B_{k-1}$. 
\end{proof}

\begin{remark}
	The key property used in Proposition \ref{prop:ss} is that $[u^{-n},\delta^{S^1}]|_{F^mC^+}$ is in $F^{n-1}C^+$. The same holds on $(F^NC^{+},\delta^{S^1})$ if $n+m\le N$, where on $F^{N}C^+$ we define $u^{-n}\cdot u^{-k}=u^{-n-k}$ for $k+n\le N$ and $u^{-n}\cdot u^{-k}=0$ for $k+n>N$. Therefore Proposition \ref{prop:ss} holds for $N$-truncated $S^1$-structure if $2k\le N$. 
\end{remark}

\begin{remark}\label{rmk:SS}
	$\{\Delta^k\}_{k\ge 1}$ carries all the information about spectral sequences from the $u$-adic filtration, as we will explain in the following.
	\begin{enumerate}
		\item\label{sss1} By \eqref{ss1} and \eqref{ss3} of Proposition \ref{prop:ss}, we have $\Delta^k$ is zero on $B_{k-1}/Z_0$ with image in $B_{k}/B_{k-1}\subset Z_{k-1}/B_{k-1}$. In particular, $\Delta^k$ induces $\overline{\Delta}^k:Z_{k-1}/B_{k-1}\to Z_{k-1}/B_{k-1}$ such that $(\overline{\Delta}^k)^2=0$ by $B_{k}\subset Z_k$. Moreover \eqref{ss3} also implies that $Z_{k}/B_k = H^*(Z_{k-1}/B_{k-1},\overline{\Delta}^k)$. That is we have an abstract spectral sequence $\{E_i=(Z_{i-1}/B_{i-1},\overline{\Delta}^i)\}_{i\ge 1}$.
		\item For the filtration $C=F^0C^+\subset \ldots F^NC^+$, we can consider the the associated Leray spectral sequence \cite{mccleary2001user}. The $0$th page is given by $C\oplus u^{-1}C\oplus \ldots \oplus u^{-N}C$ with differential $\delta^0\oplus \ldots \oplus \delta^0$.  For $0\le k\le N$, the $k+1$-page is given by $\oplus_{i=1}^N u^{-i} Z_{\min\{i,k\}}/B_{\min\{k,N-i\}}$. The differential on $k+1$th page is  from $u^{-i-k-1} Z_k/B_{\min\{k,N-i-k-1\}}\to u^{-i}Z_{\min\{i,k\}}/B_{\min\{k,N-i\}}$, which is induced by $\Delta^{k+1}$ similar to \eqref{sss1} above. Moreover, $\Delta^N$ is exactly the only differential on the $N$th page (the last page before it converges).
		\item We can similarly consider the $u$-adic filtration on $C^+$, $C^{-}$ and $C^{\infty}$, each page of the spectral sequences has a similar description induced from $\{\Delta^k\}_{k\ge 1}$. Therefore $\{\Delta^k\}_{k\ge 1}$ contain all the information about all such spectral sequences. The abstract spectral sequence in \eqref{sss1} tensor with $\bm{k}((u))$(also tensor $u^k$ to $\overline{\Delta}^k$) is the the Leray spectral sequence of $C^{\infty}$. The $u$-adic spectral sequences on $C^-,C^\infty$ typically do not converge.
	\end{enumerate}
\end{remark}

\begin{proposition}\label{prop:func1}
	Given an $S^1$-morphism $\phi: (C,\delta)\to (D,\widetilde{\delta})$, then $\phi^0$ maps $Z_k(C),B_k(C)$ to $Z_k(D),B_k(D)$ respectively, and $\phi^0\circ \Delta^k=\widetilde{\Delta}^k\circ \phi_0$. If there is another $S^1$-morphism $\tilde{\phi}$ that is homotopic to $\phi$, then the induced commutative diagrams with $\Delta^k$ are the same. If everything is $N$-truncated, then the claim holds for $2k \le N$ (the same applies to every proposition below).
\end{proposition}	
\begin{proof}
	If $\alpha_0\in Z_k(C)$, i.e.\ $\delta^{S^1}(\sum_{i=0}^ku^{-i}\alpha_{k-i})=0$, then $\widetilde{\delta}^{S^1}\circ \phi^{S^1}(\sum_{i=0}^ku^{-i}\alpha_{k-i})=0$. Since $\phi^{S^1}(\sum_{i=0}^ku^{-i}\alpha_{k-i})=u^{-k}\phi^0(\alpha_0)+\sum_{i=0}^{k-1}u^{-i}\beta_{k-i}$, we have $\phi^0(\alpha_0)\in Z_k(D)$. The situation for $B_k(C)$ is similar. To prove $\phi^0$ commutes with $\Delta^k$, let $\alpha_0\in Z_{k-1}$, then we have
	\begin{eqnarray*}
	 \phi_0\circ \Delta^k(\alpha_0) & = & \phi^0 \circ \delta^{S^1} (\sum_{i=1}^k u^{-i}\alpha_{k-i})\\
	 & = & \phi^{S^1}\circ  \delta^{S^1} (\sum_{i=1}^k u^{-i}\alpha_{k-i})\quad \text{  (since  $\delta^{S^1} (\sum_{i=1}^k u^{-i}\alpha_{k-i}) \in C$)}\\
	 & = & \text{$\widetilde{\delta}^{S^1}$}\circ \phi^{S^1}(\sum_{i=1}^k u^{-i}\alpha_{k-i}) \\
	 & = &  \text{$\widetilde{\Delta}^{k}$}\circ \phi_0(\alpha_0)+ \text{$\widetilde{\delta}^{0}$}(\sum_{i=1}^k \phi^i(\alpha_{k-i})) =  \text{$\widetilde{\Delta}^{k}$}\circ \phi_0(\alpha_0) \in Z_0/B_{k-1}
	\end{eqnarray*}
	 When $\phi,\tilde{\phi}$ are homotopic, $\phi^0=\tilde{\phi}^0$ on $H^*(C,\delta^0)=Z_0/B_0$. Hence the induced commutative diagrams with $\Delta^k$ from $\phi,\tilde{\phi}$ are the same since $Z_{k-1}/B_0\subset Z_0/B_0$.
\end{proof}

The higher order terms of $\phi$ can also be seen in some cases. For simplicity, let $\phi:(C,\delta)\to (D,\partial)$ be an $S^1$-morphism and $\partial^r=0$ for $r\ge 1$, then we define $\Phi^0:H^*(C,\delta^0)\to H^*(D,\partial^0)$ and define $\Phi^k:\ker \Delta^k\to \coker \Phi^{k-1}$ by
$$\Phi^k(\alpha_0)=\sum_{i=0}^k\phi^i(\alpha_{k-i})=\Pi_0\circ \phi^{S^1}(\sum_{i=0}^k u^{-i}\alpha_{k-i}),$$
where $\alpha_1,\ldots,\alpha_k$ are those in the definition of $Z_k$ and $\Pi_0:D^{+}\to D$ is the projection. 

\begin{proposition}
	$\Phi^k:\ker \Delta^k \to \coker \Phi^{k-1}$ is well-defined. 
\end{proposition}	
\begin{proof}
	First note that 
	\begin{eqnarray*}
		\partial^0\circ \Pi_0\circ \phi^{S^1}(\sum_{i=0}^k u^{-i}\alpha_{k-i}) & =  & \partial^{S^1}\circ \Pi_0\circ \phi^{S^1}(\sum_{i=0}^k u^{-i}\alpha_{k-i})\\
		& = & \Pi_0\circ \partial^{S^1}\circ \phi^{S^1}(\sum_{i=0}^k u^{-i}\alpha_{k-i}) \quad \text{(by $\partial^r=0$ for $r\ge 1$)}\\
		& = & \Pi_0 \circ \phi^{S^1}\circ \delta^{S^1}(\sum_{i=0}^k u^{-i}\alpha_{k-i}) = 0,
	\end{eqnarray*}
	where the last equality follows from $\alpha_0\in Z_k$. Therefore $\Phi^k(\alpha_0)$ is closed. Now assume there is another family $\alpha'_1,\ldots,\alpha'_k$ satisfies the conditions in $Z_k$ for $\alpha_0$. Then we have $\delta^{S^1}(\sum_{i=0}^{k-1}u^{-i}(\alpha_{k-i}-\alpha'_{k-i}))=0$, hence $\alpha_1-\alpha'_1\in Z_{k-1}$. Then the difference resulted in $\Phi^k(\alpha_0)$ is given by $\Pi_0\circ \phi^{S^1}(\sum_{i=0}^{k-1}u^{-i}(\alpha_{k-i}-\alpha'_{k-i}))$, which is $\Phi^{k-1}(\alpha_1-\alpha'_1)$. Hence $\Phi^k$ is well-defined as a map from $\ker \Delta^k\to \coker \Phi^{k-1}$ inductively. 
\end{proof}

\begin{remark}\label{rmk:SSphi}
	Since $\phi^{S^1}$ defines a cochain map $F^NC^+ \to F^ND^+$ preserving the $u$-adic filtration, it induces a morphism between the spectral sequences. In particular, on the last page $E_{N+1}=E_\infty=\oplus_{i=0}^N u^{-i}Z_i/B_{N-i}$, we have a morphism 
	\begin{equation}\label{eqn:iso}
	\bigoplus_{i=0}^N u^{-i}Z_i/B_{N-i}\simeq H^*(F^NC^+,\delta^{S^1}) \to \bigoplus_{i=0}^{N} u^{-i} H^*(D,\partial^0),
	\end{equation}
	where the first isomorphism is not canonical (assume we use field coefficient) and both maps preserve filtrations. Then the morphism  $\phi^{S^1}$ on cohomology can be written as a matrix 
	\begin{equation}\label{eqn:matrix}
	\left[
	\begin{array}{cccc}
	\Phi^0_0, & \Phi^1_0, & \ldots, & \Phi^N_0\\
	0 & \Phi^0_1 & \ldots, & \quad\Phi^{N-1}_1\\
	\vdots & \vdots & \ddots & \vdots \\
	0 & 0 & \ldots & \Phi^0_N
	\end{array}
	\right]
	\end{equation}
	for $\Phi^i_j:u^{-i-j}Z_{i+j}/B_{N-i-j}\to u^{-j}H^*(D)$. However, since the first isomorphism in \eqref{eqn:iso} is not canonical, $\Phi^N_0$ is not canonically defined but with a ambiguity from the image of $[\Phi^0_0,  \Phi^1_0,  \ldots \Phi^{N-1}_0]$. One can show that $\Phi^i(B_{j})=0$ and $\Phi^{i}_0 = \Phi^i| \mod \Ima \Phi^{i-1}$. In particular, $\Phi^N$ is well-defined portion of $\Phi^N_0$. Under another non-canonical isomorphism $H^*(D,\partial^0)\simeq  \coker \Phi^N\oplus_{i=1}^N \Ima \Phi^i$, the map $H^*(F^NC^+,\delta^{S^1}) \to u^{-0}H^*(D,\partial^0)$ can be reconstructed. Similarly, we have $\Phi^i_j= \Phi^i|_{Z_{i+j}} \mod  \Ima \Phi^{i-1}(Z_{i+j-1})$. In other words, the map $H^*(F^NC^+,\delta^{S^1}) \to H^{*}(F^ND^+,\partial^{S^1})$ as a $\bm{k}$ linear map can be recovered from $\Delta^k$ and $\Phi^k$.
\end{remark}

\begin{proposition}\label{prop:func2}
	Let $\phi: (C,\delta)\to (D,\partial), \widetilde{\phi}:(\widetilde{C},\widetilde{\phi})\to (\widetilde{D},\widetilde{\partial})$ be two $S^1$-morphisms such that $\partial^r,\widetilde{\partial}^r=0$ for $r\ge 1$. Assume there are $S^1$-morphisms $\psi:C\to \widetilde{C}$ and $f:D\to \widetilde{D}$ such that $f^r=0$ for $r\ge 1$ and $f\circ \phi=\widetilde{\phi}\circ \psi$ up to $S^1$-homotopy, then  we have the following commutative diagram,
	$$\xymatrix{
		\ker \Delta^k(C) \ar[r]^{\Phi^k} \ar[d]^{\psi^0} & \coker \Phi^{k-1}\ar[d]^{f^0} \\
		\ker \Delta^k(\widetilde{C}) \ar[r]^{\widetilde{\Phi}^k} & \coker \widetilde{\Phi}^{k-1}} 
	$$
\end{proposition}	
\begin{proof}
	Let $\alpha_0\in \ker Z_{k}(C)$, i.e.\ there exist $\alpha_1,\ldots,\alpha_k\in C$ such that $\delta^{S^1}(\sum_{i=0}^k u^{-i}\alpha_{k-i})=0$, then 
	\begin{eqnarray*}
		f^0\circ \Phi^k(\alpha_0) & = & f^{S^1}\circ \Pi_0\circ \phi^{S^1}(\sum_{i=0}^k u^{-i}\alpha_{k-i})\\
		& = &\Pi_0\circ f^{S^1} \circ \phi^{S^1}(\sum_{i=0}^k u^{-i}\alpha_{k-i})\\
		& = & (\Pi_0\circ \widetilde{\phi}^{S^1}\circ \psi^{S^1}+\Pi_0\circ h^{S^1}\circ \delta^{S^1}+\Pi_0\circ \widetilde{\partial}^{S^1}\circ h^{S^1})(\sum_{i=0}^k u^{-i}\alpha_{k-i})\\
		& = & \text{$\widetilde{\Phi}^k(\psi^0(\alpha_0))$} \mod \Ima \widetilde{\partial}^0,
	\end{eqnarray*}
where $h$ is the $S^1$-homotopy between $f\circ \phi$ and $\widetilde{\phi}\circ \psi$.
\end{proof}

Now assume $(C,\delta)$ is an $S^1$-structure compatible with a splitting $0\to C_0\to C \to C_+\to 0$. Then $\delta_{+,0}=(\delta^0_{+,0},\delta^1_{+,0},\ldots)$ defines an $S^1$-morphism from $C_+$ to $C_0[1]$, where the $S^1$ structure on $C_0[1]$ is given by $(-\delta_0^0,0,\ldots)$. Hence we have $\Delta^k_{+,0}:\ker \Delta^k_+ \to \coker \Delta^{k-1}_{+,0}$ with $\Delta^0_{+,0}$ is the connecting map $[\delta^0_{+,0}]:H^*(C_+,\delta_+^0)\to H^{*+1}(C_0,\delta_0^0)$. 

\begin{proposition}\label{prop:equi}
	Let $(C,\delta)$ be an $S^1$-structure compatible with a splitting and $H^*(C_0,\delta^0_0)$ has a distinguished element $[e]$ and a preferred projection  $\pi_0:H^*(C_0,\delta^0_0)\to \la [e]\ra$. Then $C$ carries a $k$-dilation iff $e\subset B_k$ or equivalently $[e]\in \Ima \Delta^k$. $C$ carries a $k$-semi-dilation iff there is $[x]\in H^*(C_0,\delta^0_0)$ that is projected to $[e]$ and $[x]\in \Ima \Delta^k_{+,0}$.  
\end{proposition}
\begin{proof}
	$e\in B_k$ is equivalent to $\iota^{S^1}[e]=0$ in $H^*(F^kC^+,\delta^{S^1})$ by definition, hence equivalent to $C$ carries a $k$-dilation. Assume $C$ carries a $k$-semi-dilation, then there exists  a closed class $A=\sum_{i=0}^k u^{-i}\alpha_{k-i}$, such that $\delta^{S^1}_{+,0}( A)=x+\sum_{i=1}^k u^{-i}y_i$ with $[x]\in H^*(C_0,\delta_0^0)$ that is projected to $[e]$. Hence $\Delta^k_{+,0}(\alpha_0)=[x]$ by definition. On the other hand, we may assume $k$ is the smallest number such that there exists $[x]\in \Ima\Delta^k_{+,0}$ and is projected to $[e]$. In other words, if we view $\Ima \Delta^{k-1}_{+,0}$ as a subspace $H^*(C_0,\delta_0^0)$, then $\pi_0\circ \Ima \Delta^{k-1}_{+,0}=0$. Then we have a closed class $A=\sum_{i=0}^k u^{-i}\alpha_{k-i}$, such that $\Delta^k_{+,0}(\alpha_0)=\Pi_0 \circ \delta^{S^1}_{+,0}(A)=x$ up to ambiguity from $\Ima \Delta^{k-1}_{+,0}$. This implies that $A$ is sent to $[e]$ in the composition $H^*(F^kC^+_+,\delta^{S^1}_+)\to H^{*+1}(F^kC^+_0,\delta_0^{S^1}) \stackrel{\pi_0}{\to} \la e\ra$ as the ambiguity is projected to zero by $\pi_0$, i.e.\ $C$ carries a $k$-semi-dilation. 
\end{proof}

\begin{remark}
	In addition to the reconstruction property in Remark \ref{rmk:SS}, \ref{rmk:SSphi}, the purpose of introducing $\Delta^k,\Delta^k_+,\Delta^k_{+,0}$ and reinterpreting $k$-(semi)-dilation with them is that, in the context of symplectic cohomology, they are defined on subspace/quotient space of the regular (positive) symplectic cohomology and the regular symplectic cohomology is typically sparser than the $S^1$-equivariant symplectic cohomology. In particular, the grading information of symplectic cohomology will be handy in determining the order of dilation, see \S \ref{s52}.
\end{remark}

\section{$k$-dilations of Liouville domains}\label{s3}
$S^1$-structure appears naturally in the construction of $S^1$-equivariant symplectic cohomology, which was sketched in \cite{seidel2006biased}.  The idea is to count Floer cylinders parametrized by the moduli spaces of gradient flow lines  a perfect Morse function on $\CP^\infty$. 
Such formalism of $S^1$-equivariant symplectic cohomology was carried out in detail in \cite{bourgeois2016s} and can also be found in \cite{gutt2017positive,zhao2014periodic}. We will first review the construction following \cite{ zhao2014periodic} and define $k$-(semi)-dilation for Liouville domains. Then we will make some modifications on the construction for the proof of Theorem \ref{thm:B}.
\subsection{Symplectic cohomology}\label{s31}
Let $(W,\lambda)$ be an exact domain such that the Reeb flow of $R_{\lambda}$ on $\partial W$ is non-degenerate. We use $\cS$ to denote the length spectrum of the Reeb orbits of $R_{\lambda}$, i.e. 
$$\cS:=\left\{\int \gamma^*\lambda| \gamma \text{ is a $R_{\lambda}$ orbit}\right\}.$$ 
We assume $W,\partial W$ are both connected through out this paper. Let $(\widehat{W},\widehat{\lambda})$ be the completion of $(W,\lambda)$, i.e.\ $\widehat{W} = W \cup \partial W \times (1,\infty)$, $\widehat{\lambda}=\lambda$ on $W$ and is $r\lambda|_{\partial W}$ on $\partial W \times (1,\infty)$. We first recall the construction of symplectic cohomology and its positive counterpart. A Hamiltonian $H_t:S^1_t\times \widehat{W}\to \R$ is called admissible of slope $a\notin \cS$ if the following holds.
\begin{enumerate}
	\item $H_t$ is $C^2$-small and independent of $t$ on $W$.
	\item $H_t=ar+b$ for $r\gg 0$.
	\item All periodic orbits of $X_{H_t}$ are non-degenerate.
\end{enumerate}
Let $H_t$ be an admissible Hamiltonian, the symplectic action of a Hamiltonian orbit $x$ is defined by
\begin{equation}\label{eqn:action}
\cA_{H_t}(x):=-\int x^*\widehat{\lambda}+\int_{S^1}H_t\circ x \rd t.
\end{equation}
Let $J_t$ be an $S^1$-dependent almost complex structure, that is $J_t$ is compatible with $\rd \widehat{\lambda}$ and cylindrical convex for $r\gg 0$, i.e.\  $\widehat{\lambda} \circ J_t = \rd r$ for $r\gg 0$. Using such $J_t$, for two orbits $x,y$, we have a compactified moduli space
$$\cM_0(x,y):=\overline{\left\{u:\R_s\times S^1_t\to \widehat{W}\left|\partial_su+J_t(\partial_tu-X_{H_t})=0,\lim_{s\to-\infty} u = x, \lim_{s\to +\infty} u = y\right. \right\}/\R}.$$
For generic choices of $J$ and a choice of coherent orientations, we can count the moduli spaces $\cM_0(x,y)$ that have expected dimension $0$ and define a cochain complex $C(H_t)$ generated by orbits of $H_t$ whose differential is defined by $\delta^0(y)=\sum \# \cM_0(x,y)x$. If $H_t\le K_t$ for $r\gg 0$, then there is a continuation cochain map $C(H_t)\to C(K_t)$ and symplectic cohomology $SH^*(W)$ is defined as $\varinjlim_{H_t}H^*(C(H_t))$. Moreover, any colimit over a sequence of Hamiltonians with slope going to $\infty$ computes the symplectic cohomology. We will grade symplectic cohomology by $|x|=n-\mu_{CZ}(x)$ (which is in the $\Z_2$ sense if $c_1(W)\ne 0$) so that the product structure on the symplectic cohomology has no grading shift. 

When we choose $H_t$ such that $H_t$ is $C^2$-close to a function $h(r)$ with $h''(r)\ge 0$ on $\partial W \times (1,\infty)$, then the orbits of $X_{H_t}$ can be separated into constant orbits on $W$ whose symplectic action \eqref{eqn:action} is close to $0$ and non-constant orbits, which in pairs correspond to Reeb orbits of $\partial W$ with period up to the slope of $H_t$, whose symplectic action is smaller than $-\min \cS$. Then they form two complexes $C_0(H_t),C_+(H_t)$ and we have a short exact sequence
$$0\to C_0(H_t)\to C(H_t)\to C_+(H_t)\to 0, $$
since the differential of $C(H_t)$ increases the symplectic action. The continuation map $C(H_t)\to C(K_t)$ can be chosen to be compatible with the short exact sequence. $C_0(H_t)$ is simply the Morse cochain complex of $W$ if we choose $J_t$ independent of $t$ on $W$ and the positive symplectic cohomology $SH^*_+(W)$ is defined as $\varinjlim_{H_t} H^*(C_+(H_t))$. Then we have a tautological long exact sequence,
$$\ldots \to H^*(W)\to SH^*(W)\to SH^*_+(W)\to H^{*+1}(W)\to\ldots $$

To construct the $S^1$-equivariant symplectic cohomology, we consider the following perfect Morse function on $\CP^N$,
$$f_N([x_0:\ldots:x_N])=\frac{\sum_{j=0}^Nj|x_j|^2}{\sum_{j=0}^N|x_j|^2}.$$
Let $g_N$ be Fubini-Study metric on $\CP^N$ and $\widetilde{g}_N$ the standard metric on $S^{2N+1}$. Then $(f_N,g_N)$ is a Morse-Smale pair. We have $N+1$ critical points $\{z_i\}_{0\le i \le N}$ with the property that the gradient flow-lines from $z_i$ to $z_j$ is contained in the $\CP^{j-i}$ using the $z_i,\ldots,z_j$ coordinates and can be identified with the gradient flow lines from $z_{i+k}$ to $z_{j+k}$ naturally. Let $\widetilde{f}_N$ denote the lift of $f_N$ to the $S^{2N+1}$. Each critical point $z_i$ of $f_N$ lifts to a critical orbit $S_i$ of $\widetilde{f}_N$. For every critical point $z_i$, we pick a neighbourhood $O_i$ and a lift of $O_i$ to a local slice $U_i$ of the $S^1$ action, such that $\nabla \widetilde{f}_N$ is tangent to $U_i$\footnote{For example, we can pick $U_i\subset S^{2N+1}$ around the point determined by $x_i=1$ to be those with real $x_i$ coordinate.}. Let $\widetilde{z}_i$ denote the lift of $z_i$ in $U_i$,  then $V_i:=S^1\cdot U_i$ is an invariant neighbourhood of the critical orbit $S_i=S^1\cdot \widetilde{z}_i$. 
\begin{definition}\label{def:adm}
	A Hamiltonian $H_N:S^1\times S^{2N+1} \times \widehat{W}\to \R$ is said to be admissible if it satisfies the following conditions.
	\begin{enumerate}
		\item\label{a1} $H_N(t+\theta,\theta\cdot z, x)=H_N(t,z,x)$ for $\theta \in S^1$ and $H_N$ has slope $a\notin \cS$ for $r\gg 0$
		\item On $S^1\times V_i\times \widehat{W}$, the Hamiltonian is given by $H_N(t,x,z)=H_{t-\theta_z}(x)$, where $H_t:S^1\times \widehat{W}\to \R$ is a fixed admissible Hamiltonian with slope $a$ and $\theta_z$ is the unique element such that $\theta_z^{-1}z\in U_i$. In other words, for every $i$, on the slice $S^1\times U_i\times \widehat{W}$, we have $H_N(t,z,x)=H_t(x)$, and the property above follows from \eqref{a1}\footnote{Since we would like to construct an $S^1$-structure on $C(H_t)$, we require $H_t$ is independent of the critical point.}.
		\item\label{a3} We impose the symmetry that $H_N$ is the same on $\CP^{j-r}=\{(0:\ldots:0:z_r:\ldots:z_j:0:\ldots:0)\}\simeq \{(0:\ldots:0:z_{r+k}:\ldots:z_{j+k}:0:\ldots:0)\}$	via translation.
	\end{enumerate}
    Similarly a family of almost complex structures $J_N:S^1\times S^{2N+1}\to \End(T\widehat{W})$  is admissible if the following holds.
    \begin{enumerate}
    	\item $J_N$ is compatible with $\rd\widehat{\lambda}$ and is cylindrical convex for $r\gg 0$.
	 	\item $J_N(t+\theta,\theta\cdot z)=J_N(t,z)$, and over slice $U_j\ni \widetilde{z}_j$, $J_N$ is a fixed $S^1$-dependent admissible almost complex structure $J_t$. 
	 	\item We impose the symmetry that $J_N$ is the same on $\CP^{j-i}=\{(0:\ldots:0:z_i:\ldots:z_j:0:\ldots:0)\}\simeq \{(0:\ldots:0:z_{i+k}:\ldots:z_{j+k}:0:\ldots:0)\}$	via translation.
    \end{enumerate}
\end{definition}
Then for $r\le N$ and two orbits $x,y$ of $X_{H_t}$, we consider the following moduli space for $0\le k \le N-r$
\begin{equation*}
M_r(x,y):=\left\{\begin{array}{l} u:\R \times S^1 \to \widehat{W}, \\ z:\R \to S^{2N+1}
\end{array}\left|\begin{array}{l} \partial_su+J_N(t,z(s),u)(\partial_tu-X_{H_N(t,z(s))}) = 0\\
z'+\nabla_{\widetilde{g}_N}\widetilde{f}=0\\
\displaystyle\lim_{s\to -\infty} (u,z) \in S^1\cdot (x,\widetilde{z}_{k+r}), \lim_{s\to \infty}(u,z)\in S^1\cdot (y,\widetilde{z}_{k}). 
\end{array} \right. \right\}/\R\times S^1.
\end{equation*}
Here the $S^1$ action is given by $\theta\cdot (u,z) = (u(\cdot,\cdot+\theta), \theta\cdot z)$ and the $\R$ action is given by the simultaneous translation in $s$ of $(u,z)$. The symmetric property imposed on $H_N,J_N$ implies that $M_r(x,y)$ does not depend on $k$. The expected dimension of $M_r(x,y)$ is $|x|-|y|-1+2r$ and $M_r(x,y)$ admits a compactification 
$$\cM_r(x,y):=\bigsqcup_{\substack{r_1+\ldots+r_s=r \\ x_1,\ldots,x_r}} M_{r_1}(x,x_1)\times \ldots \times M_{r_s}(x_s,y),$$
which is a compact manifold with boundary for generic $J_N$ and if the expected dimension is $\le 1$. In particular, we can define 
$$\delta^r(y)=\sum \# \cM_r(x,y)x.$$
Then the boundary combinatorics of $\cM_r(x,y)$ implies that $(\delta^0,\ldots,\delta^N)$ is a $N$-truncated $S^1$-structure on $C(H_t)$ with $\delta^0$ is the Floer differential defined before. Next we can choose admissible $H_{N+1}$ with the property that $H_{N+1}(z_0:\ldots:z_N:0)=H_{N+1}(0:z_0:\ldots:z_N)=H_N(z_0:\ldots:z_N)$ and similarly for $J_{N+1}$, as a consequence, $\delta^r$ defined by $H_N$ and $H_{N+1}$ coincides for $r\le N$. Hence we can define $\delta^r$ for all $r$ and give $C(H_t)$ an $S^1$-structure. 

Now if $K_t$ is a single admissible Hamiltonian with slope larger than the slope of $H_t$, then we can construct admissible family $K_N,H_N$ such that $K_N\ge K_N$ for $r\gg 0$. Then can build a homotopy $F_N:\R_s\times S^1_t\times S^{2N+1}\times \widehat{W}$ satisfies the similar equivariant and symmetric properties \eqref{a1}, \eqref{a3} in Definition \ref{def:adm} and $F_N = K_N$ for $s\ll 0$, $F_N=H_N$ for $s\gg 0$.  

For $0\le r \le N, 0\le k \le N-r$, by considering the compactification $\cN_r(x,u)$ of the following moduli space using a compatible homotopy of almost complex structures $J_N:\R\times S^1 \times S^{2N+1}\to \End(T\widehat{W})$,
\begin{equation*}
N_r(x,y):=\left\{\begin{array}{l} u:\R \times S^1 \to \widehat{W},\\
z:\R \to S^{2N+1}\end{array}\left|\begin{array}{l} \partial_su+J_N(s,t,z(s),u)(\partial_tu-X_{F_N(s,t,z(s))}) = 0\\
z'+\nabla_{\widetilde{g}_N}\widetilde{f}=0\\
\displaystyle\lim_{s\to -\infty} (u,z) \in S^1\cdot (x,\widetilde{z}_{k+r}), \lim_{s\to \infty}(u,z)\in S^1\cdot (y,\widetilde{z}_{k}). 
\end{array} \right. \right\}/ S^1.
\end{equation*}
we can define a $N$-truncated morphism $\phi^r(y) = \sum \# \cN_{r}(x,y) x$  from $C(H_t)$ to $C(K_t)$. Using the same extension argument as before, we can enhance it to an $S^1$-morphism. Then the equivariant symplectic cohomology $SH^*_{S^1}(W)$ is defined as
$$\varinjlim_{H_t} H^*(C^+(H_t)),$$
which is a $\bm{k}[u]$ module and every element is torsion.

When we choose $H_t$ in the special form as in the definition of $C_+(H_t)$. We can then choose $H_N$ so that the action argument works for the $S^1$-structure, such that $\delta^r(x)$ has no $C_+(H_t)$ component for $x\in C_0(H_t)$. On the other hand, since we can pick $J_N,H_N$ on $W$  both independent of $S^1\times S^{2N+1}$ and still get transversality when $H_N|_W$ is $C^2$-small and $x,y$ are both constant orbits, then $\cM_r(x,y)$ carries an additional $S^1$-action by $\theta \cdot (u,z)= (u,\theta \cdot z)$ which is necessarily free when $r>0$. Therefore  $\la \delta^r(x),y\ra=0$ for $r>0,x,y\in C_0(H_t)$, i.e.\ the $N$-truncated $S^1$-structure $C(H_t)$ is compatible with splitting and we have the following tautological long exact sequence,
\begin{equation}\label{eqn:exact}
\ldots \to H^*_{S^1}(W)\to SH^*_{S^1}(W) \to SH^*_{S^1,+}(W)\to H^{*+1}_{S^1}(W)\to \ldots,
\end{equation}
where $SH^*_{S^1,+}(W):=\varinjlim_{H_t} H^*(C^+_+(H_t))$ and $H^*_{S^1}(W):=H^*(C^+_0(H_t))=H^*(W)\otimes (\bm{k}[u,u^{-1}]/u)$. We refer readers to \cite{bourgeois2016s,gutt2017positive,zhao2014periodic} for details on well-definedness of such invariants, e.g.\ independence of auxiliary data etc.

\begin{remark}
	Strictly speaking, it may not be possible to arrange the extension property $H_N(z_0:\ldots:z_N)=H_{N+1}((z_0:\ldots:z_N:0)=H_{N+1}(0:z_0:\ldots:z_N)$ and maintaining the action separation for $C_0,C_+$ at the same time, as $\delta^r$ counts parameterized Floer cylinders whose energy is the difference of symplectic actions with an error term $\int_{\R\times S^1} \partial_sH(t,z(s),u)\rd s \rd t$, see the proof of Proposition \ref{prop:reg}. A solution to this was explained in  \cite[p.3867 before \S 2.3]{bourgeois2016s}. Another solution is by setting up the $S^1$ complex using an autonomous Hamiltonian $H$ with $S^1$ families of orbits that are Morse-Bott non-degenerate and setting $H_N=H$ independent of $S^1\times S^{2N+1}$ for any $N$. The cascades construction in \cite{bourgeois2009symplectic} can be easily adapted to this parameterized version and the error term from $\partial_s H$ in energy vanishes, see \S \ref{s52} for a brief description of such cascades construction. For the purpose of defining $SH^*_{S^1,+}(W)$ and obtaining \eqref{eqn:exact}, we can find $H^N_N$ with slope goes to $\infty$ as $N\to \infty$, such that the action argument works for both $H^N_N$ and the continuation map induced by the homotopy between $H^N_N$ and $H^{N+1}_{N+1}|_{S^{2N+1}\subset S^{SN+3}}$. Then $SH^*_{S^1,+}(W)$ is defined as $\varinjlim_{N\to\infty} H^*(F^NC^+_+(H^N))$, where $H^N$ is $H^N_N$ over critical points. This is the approach taken in \cite{gutt2017positive}.
\end{remark}

\begin{remark}
	Instead of defining equivariant symplectic cohomology as a colimit of cohomology, it is also possible to define an $S^1$-cochain complex for the full equivariant symplectic cohomology. Given a family of admissible Hamiltonians $H_t^i$ with slope goes to $\infty$, the construction above gives  an $S^1$-structure on each $C(H_t^i)$ as well as an $S^1$ morphism $C(H_t^i)\to C(H_t^{i+1})$. Then by the telescope construction we get a homotopy colimit $\mathrm{hocolim} C(H^i_t)$ with an $S^1$-structure \cite[\S 7]{zhao2014periodic}. It is straightforward to check that the $(\mathrm{hocolim} C(H^i_t))^+$ is the homotopy colimit of $C^+(H^i_t)$, in particular $H^*((\mathrm{hocolim} C(H^i_t))^+)\simeq SH^*_{S^1}(W)$. However, as explained in \cite[\S 8]{zhao2014periodic}, this is no longer the case for $C^-,C^{\infty}$ as the homotopy colomit does not commute with the $u$-adic completion in the definition of $C^-,C^{\infty}$ in general.
\end{remark}

\subsection{$k$-(semi)-dilations}\label{s3s3}
We use $\pi_0$ to denote the projection $H^*(W)\to H^0(W)$ as well as the projection from $H_{S^1}^*(W)$ to $H^0(W)$ if there is no confusion. In view of the discussion above,  Definition \ref{def:dilation} and Proposition \ref{prop:order}, we introduce the following.
\begin{definition}
	We say $W$ carries a $k$-dilation, iff there exists $x\in SH^*_{S^1,+}(W)$ such that $[\delta^{S^1}_{+,0}](x)=1$ and $u^{k+1}x=0$. We say $W$ carries a $k$-semi-dilation, iff there exists $x\in SH^*_{S^1,+}(W)$ such that $\pi_0\circ [\delta^{S^1}_{+,0}](x)=1$ and $u^{k+1}x=0$.
\end{definition}

\begin{proposition}\label{prop:TFAE}
	The followings are equivalent.
	\begin{enumerate}
		\item\label{e1} $W$ carries a $k$-(semi)-dilation.
		\item\label{e2} $C(H_t)$ carries a $k$-(semi)-dilation for some admissible $H_t$.
		\item\label{e3} The $k$-truncated $S^1$-structure on $C(H_t)$ carries a $k$-(semi)-dilation.
	\end{enumerate}
\end{proposition}	
\begin{proof}
	The equivalence of \eqref{e1} and \eqref{e2} follows from that $SH^*_{S^1}(W)$ arises as a colimit of $H^*(C^+(H_t))$ and Proposition \ref{prop:order}. The equivalence of \eqref{e2} and \eqref{e3} is by definition.
\end{proof}

We define $\Delta^k=\varinjlim_{H_t} \Delta^k(C(H_t)),\Delta^k_+=\varinjlim_{H_t} \Delta^k(C_+(H_t))$ and $\Delta^k_{+,0}:=\varinjlim_{H_t} \Delta^k_{+,0}(C(H_t))$, then we have $\Delta^1$ is the BV operator $\Delta$ on $SH^*(W)$ and $\Delta^k:\ker \Delta^{k-1}\to \coker \Delta^{k-1},\Delta^k_+:\ker \Delta_+^{k-1} \to \coker \Delta_+^{k-1}$. $\Delta^0_{+,0}$ is the connecting map $SH^{*}_+(W)\to H^{*+1}(W)$ and $\Delta^k_{+,0}:\ker \Delta_+^{k}\to \coker \Delta_{+,0}^{k-1}$. Moreover we have the following.

\begin{proposition}\label{prop:transfer}
	The structural maps $\Delta^k, \Delta^{k}_+, \Delta_{+,0}^k$ are compatible with the Viterbo transfer map in the following sense. Assume $V\subset W$ is an exact subdomain, then the following diagram commutes.
	$$\xymatrix{\ker \Delta^k_W \ar[r]^{\Delta^{k+1}_{W}}\ar[d] &  \coker \Delta^k_W\ar[d] &\ker \Delta^k_{+,W} \ar[r]^{\Delta^{k+1}_{+,W}}\ar[d] &  \coker \Delta^k_{+,W}\ar[d] &	\ker \Delta^k_{+,W} \ar[r]^{\Delta^k_{+,0,W}}\ar[d] &  \coker \Delta^{k-1}_{+,0,W}\ar[d]\\
		\ker \Delta^k_V \ar[r]^{\Delta^{k+1}_{V}} &  \coker \Delta^k_V & 	\ker \Delta^k_{+,V} \ar[r]^{\Delta^{k+1}_{+,V}} &  \coker \Delta^k_{+,V} 	& \ker \Delta^k_{+,V} \ar[r]^{\Delta^k_{+,0,V}} &  \coker \Delta^{k-1}_{+,0,V}},	
	$$
	where the vertical maps are Viterbo transfer maps.
\end{proposition}
\begin{proof}
	It follows from the construction of the Viterbo transfer map for $S^1$-equivariant symplectic cohomology in \cite{gutt2017positive} as a continuation map, hence an $S^1$-morphism, as well as Proposition \ref{prop:func1} and \ref{prop:func2}.
\end{proof}

We define $\Delta^k_{\partial}$ to be the composition of $\Delta^k_{+,0}$ with $H^*(W)\to H^*(\partial W)$, then by definition, $\Delta^k_{\partial}$ is well-defined from $\ker \Delta^{k}_+$ to $\coker \Delta^{k-1}_{\partial}$. In order words, $\Delta_{\partial}^{k}$ is induced from the composed $S^1$-morphism $C_+(H_t)\stackrel{\delta_{+,0}}{\to} C(W)\to C(\partial W)$, where $C(W),C(\partial W)$ are Morse cochain complexes of $W,\partial W$ with the trivial $S^1$-structure (i.e.\ $\delta^r=0$ for $r\ge 1$). An instant corollary of the Viterbo transfer map in Proposition \ref{prop:transfer} is the following.
\begin{proposition}\label{prop:inv}
	The structure maps $\Delta^k, \Delta^{k}_+, \Delta^k_{+,0}, \Delta^k_{\partial}$ are invariants of Liouville domains up to exact symplectomorphisms.
\end{proposition}

Then by Proposition \ref{prop:equi}, we can define $k$-(semi)-dilation equivalently by the following.

\begin{proposition}
	$W$ carries a $k$-dilation iff $1\in \Ima \Delta^k$ and $W$ carries a $k$-semi-dilation iff there exists an element $x\in H^*(\partial W)$ such that $\pi_0(x)=1$ and $x\in \Ima \Delta^k_{\partial}$. 
\end{proposition}
Then $0$-semi-dilation is equivalent to $SH^*(W)=0$, as the condition is equivalent to the vanishing of the unit $\iota(x)\in SH^*(W)$ for $x\in H^*(W)$ with $\pi_0(x)=1$. Moreover, $1$-dilation, i.e.\ $\Delta(x)=1$, is the symplectic dilation defined in \cite{seidel2012symplectic}. By Proposition \ref{prop:hier}, $k$-dilation implies $k$-semi-dilation, but not necessarily the other way around. It is not clear to us whether there exist examples with $k$-semi-dilation but not $k$-dilations. By Proposition \ref{prop:hier} and \ref{prop:transfer}, we have the following.
\begin{proposition}\label{prop:inc} Let $W$ be a Liouville domain. If $W$ admits a $k$-(semi)-dilation, then $W$ admits a $j$-(semi)-dilation for every $j>k$. 	If in addition, $V$ is an exact subdomain of $W$, then $V$ admits a $k$-(semi)-dilation.
\end{proposition}

\begin{example}
	The cotangent bundle of $K(\pi,1)$ (e.g.\ $T^n$, hyperbolic manifolds) does not admit any $k$-semi-dilation, because the positive symplectic cohomology in the trivial homotopy class is trivial and operators $\Delta^k_+, \Delta^k_{\partial}$ preserve the homotopy class of the generator. Then if $V$ admits a $k$-semi-dilation, $V$ does not contain an exact $K(\pi,1)$ Lagrangian by Proposition \ref{prop:inc}.	
\end{example}
In view of Proposition \ref{prop:inc}, we define the following invariants for Liouville domain, which serve as measurements of the complexity of Liouville domains.
\begin{definition}
	Given a Liouville domain $W$, we define $\rD(W)$ to be the minimal number $k$ such that $W$ admits a $k$-dilation over $\Q$ and $\rSD(W)$ to be the minimal number $k$ such that $W$ admits a $k$-semi-dilation over $\Q$. $\rD(W),\rSD(W)$ are $\infty$ if the structure does not exist. We can also define $\rD(W;\bm{k})$ and $\rSD(W;\bm{k})$ using any other coefficient $\bm{k}$.
\end{definition}

\subsection{An alternative construction}\label{s32}
In the following, we will introduce a slightly different construction of the $S^1$-equivariant theory, that is used in \S \ref{s4} for the independence result for ADC contact manifolds. The construction has the following features (1) Hamiltonian $H_t$ is $0$ on $W$, which makes the neck-stretching argument cleaner; (2) The construction essentially uses a Hamiltonian that captures all Reeb orbits and absorb the continuation map $C(H_t)\to C(K_t)$ as an inclusion of subcomplexes, which makes the comparison of the theory between two hypothetical fillings simpler, since there is no need to compare this continuation map; (3) The drawback is that we can only construct a truncated $S^1$-structure on the cochain complex.  

We first modify the definition of admissible Hamiltonian, which will be called nice Hamiltonians. In the following, we fix a small positive number $\epsilon<\min \cS$.  Let $H_0:\widehat{W}\to\R$ satisfy the following.
\begin{enumerate}
	\item $H_0=0$ on $W$.
	\item $H_0=h(r)$ on $\partial W\times (1,\infty)$ such that $h',h''\ge 0$ and $\lim_{r\to \infty}h'(r)=\infty$.
	\item There exists open intervals $(a_i,b_i)$ with $a_0>1$ going to $\infty$, such that $h'(r)$ is constant and not in $\cS$ on $[a_i,b_i]$ and $h''(r)=0$ only on those intervals. Moreover, $h'(r)<\min \cS$ on $[1,a_0]$.
	\item\label{sep} $\epsilon < \min\{-r_1h'(r_1)+h(r_1)+r_2h'(r_2)-h(r_2)| r_1<a_i, r_2>b_i, h'(r_1), h'(r_2)\in \cS     \}.$
\end{enumerate}
We use $c_i$ to denote the middle point of $[a_i,b_i]$ and $W^i$ to denote the subspace $\{r\le c_i\}$. Then the non-constant orbits of $X_{H_0}$ correspond to Reeb orbits of $R_{\lambda}$ on level $r$ such that $h'(r)\in \cS$. The corresponding symplectic action \eqref{eqn:action} is given by $-rh'(r)+h(r)$. The purpose of \eqref{sep} is to make sure that the action gap of the orbits inside $W^i$ and outside $W^i$ is at least $\epsilon$, which is crutial in proving orbits inside $W^i$ form a subcomplex. To see that we can find such uniform $\epsilon$ for all $i$, note that $\frac{\rd}{\rd r}(-rh'(r)+h(r))=-rh''(r) \le 0$. Let $r_1$ be the maximum $r_1$ such that $r_1<a_i$ and $h'(r_1)\in \cS$ and $r_2$ the minimum $r_2$ such that $r_2>b_i$ and $h'(r_2)\in \cS$, hence we have
\begin{eqnarray*}
	\min\left\{-r_1h'(r_1)+h(r_1)+r_2h'(r_2)-h(r_2)| r_1<a_i, r_2>b_i, h'(r_1), h'(r_2)\in \cS \right\} & \ge &  \int_{r_1}^{r_2} rh''(r) \\
	& \ge  & r_1(h'(r_2)-h'(r_1)). 
\end{eqnarray*}
$h'(r_2)-h'(r_1)$ is a gap of two consecutive numbers in $\cS$, which might converge to $0$ as $h'(r_1)$ goes to $\infty$. However, we can arrange $h'(r)$ to grow slow enough, such that for every consecutive numbers $a,b\in \cS$, we have $r(b-a)$ is uniformly bounded below by a positive number where $r$ satisfies $h'(r)=a$. A Hamiltonian $H:S^1\times \widehat{W}\to \R$ is called nice iff it is a $C^2$-small time dependent non-degenerate perturbation to $H_0$ on $\partial W\times (1,\infty)$ away from $\cup_{i}\partial W \times (a_i,b_i)$, such that pairs of non-constant orbits correspond to $S^1$-family of orbits of $X_{H_0}$ with arbitrarily close symplectic action and the action gap between orbits in $W^i$ and outside $W^i$ is at least $\epsilon$ for all $i$.

\begin{definition}
	An almost complex structure $J$ is nice iff the following holds.
	\begin{enumerate}
		\item $J$ is compatible with $\rd \widehat{\lambda}$.
		\item $J$ is cylindrical convex on $\partial W \times (a_i,b_i)$ and near $\partial W$, i.e. $ \widehat{\lambda} \circ J = \rd r$. 
	\end{enumerate}
	We use $\cJ(W)$ to denote the set of nice almost complex structures.
\end{definition}

\begin{definition}
	A Morse function $f$ on the domain $W$ is admissible iff $\partial_r f>0$ on $\partial W$\footnote{Since $W$ is viewed as the space of constant orbits of $H$, we only need a Morse function on $W$ that computes that cohomology of $W$.}  and has a unique local minimum.
\end{definition}

The nice $H$ is degenerate, but is almost Morse-Bott non-degenerate. In particular, following \cite[\S 2.1]{filling}, we can construct the symplectic cochain complex $C(H,f)$ by non-degenerate orbits of $H$ and Morse critical points of $f$, where the differential consists of the gradient trajectories of $f$, Floer cylinders of $H$ with non-constant asymptotic orbits and fiber products of the stable manifolds of $\nabla f$ and moduli spaces of perturbed Cauchy-Riemann equations on $\C$ which is asymptotic to a non-constant Hamiltonian orbits near $\infty$. This formalism is better suited for the neck-stretching argument.

To ensure the integrated maximum principle applies, we introduce the following.
\begin{definition}\label{def:ham}
	A consistent Hamiltonian data $\bH_N$ up to level $N$ consists of the following.
	\begin{enumerate}
		\item\label{h1} For every $i\in \N$, we have a Hamiltonian $H^i_N:S^1\times S^{2N+1}\times \widehat{W}\to \R$ satisfying \eqref{a1} and \eqref{a3} of Definition \ref{def:adm}, which is a $C^2$-small perturbation to $H_0$ on $W^i$ near non-constant periodic orbits of $X_{H_0}$, such that $H^i_N = h'(c_i)r+(h(c_i)-h'(c_i)c_i)$ for $r\ge c_i$. 
		\item\label{h2} On $S^1\times U_j\times \widehat{W}$, we have $H^i_N(t,z,x)=H^i(t,x)$, where $H^i$ is the linear extension of a nice Hamiltonian $H$ truncated at $W^i$. In particular, $H^i$ has the property that the periodic orbits of $X_{H^i}$ are points in $W$ and pairs of orbits in $W^i$ corresponding to Reeb orbits of period up to $h'(c_i)$. Moreover, the symplectic action $\cA_{H^i}(x)<-\epsilon$ for every non-constant orbits $x$ of $X_{H^i}$. 
		\item For every gradient flow line $z(s)$ of $(\widetilde{f}_N,\widetilde{g}_N)$, we have $\partial_s H^i_N(t,z(s),x) < \frac{\epsilon}{2L_N}$, for every $t\in S^1,x\in W^i$, where $L_N$ is defined as
		$$L_N:=\max\left\{\int_{z(s)\notin \cup O_i} \rd s \st  z(s) \text{ is a gradient flow line (possibly broken) of } (f_N,g_N)\right\}.$$
		\item\label{h4} We have $H_N^{i+1}\ge H_N^i$ on $\widehat{W}$ and $H_N^i=H_N^{i+1}$ on $W^i\cup \partial W \times (c_i,b_i)$, we assume $\epsilon < \cA_{H^i}(\gamma)-\cA_{H^{i+1}}(\gamma')$ for all non-constant orbit $\gamma$ of $X_{H^i}$ and $\gamma'$ of $X_{H^{i+1}}$ that lies outside $W^i$.
		\item\label{h3} For every $i\in \N$, we have a $\R$-family of Hamiltonians $H_N^{i,i+1}$ defined by $\rho(s)H_{N}^{i+1}+(1-\rho(s))H_N^i$, where $\rho(s)$ is a non-increasing function such that $\rho(s)=1$ for $s\ll 0$ and $\rho(s)=0$ for $s\gg 0$. 
	\end{enumerate}	
\end{definition}

\begin{figure}[H]
	\begin{tikzpicture}[xscale=0.8,yscale=0.4]
	\draw [<->] (-1,11) -- (-1,0) -- (13,0);
	\draw (0,0) -- (1,0) to [out=0,in=200] (3,0.4);
	\draw[domain=3:5] plot (\x, {0.36397023426*(\x-3)+0.4});
	\draw[domain=3:11] plot (\x, {0.36397023426*(\x-3)+0.4});
	\node at (11.5,3.312) {$H^1$};
	\draw (5,1.12794046853) to [out=20, in=230] (7,2.3);
	\draw[domain=7:9] plot (\x, {1.19175359259*(\x-7)+2.3});
	\draw[domain=7:11] plot (\x, {1.19175359259*(\x-7)+2.3});
	\node at (11.5, 7.08) {$H^2$};
	\node at (11,10.5) {$H_0$}; 
	\draw (9, 4.68350718519) to [out=50, in=260] (11,10);
	\draw[dotted] (3,0) -- (3,0.4);
	\draw[dotted] (5,0) -- (5,1.12794046853);
	\draw[dotted] (7,0) -- (7,2.3);
	\draw[dotted] (9,0) -- (9,4.68350718519);
	\draw [fill] (3,0) circle [radius=0.05];
	\node at (3,-0.5) {$a_0$};
	\draw [fill] (4,0) circle [radius=0.05];
	\node at (4,-0.5) {$c_0$};
	\draw [fill] (5,0) circle [radius=0.05];
	\node at (5,-0.5) {$b_0$};
	\draw [fill] (7,0) circle [radius=0.05];
	\node at (7,-0.5) {$a_1$};
	\draw [fill] (8,0) circle [radius=0.05];
	\node at (8,-0.5) {$c_1$};
	\draw [fill] (9,0) circle [radius=0.05];
	\node at (9,-0.5) {$b_1$};
	\node at (0.5,-0.5) {$W$};
	\end{tikzpicture}
	\caption{Hamiltonians $H^i$ and $H_0$}\label{fig:ham}
\end{figure}
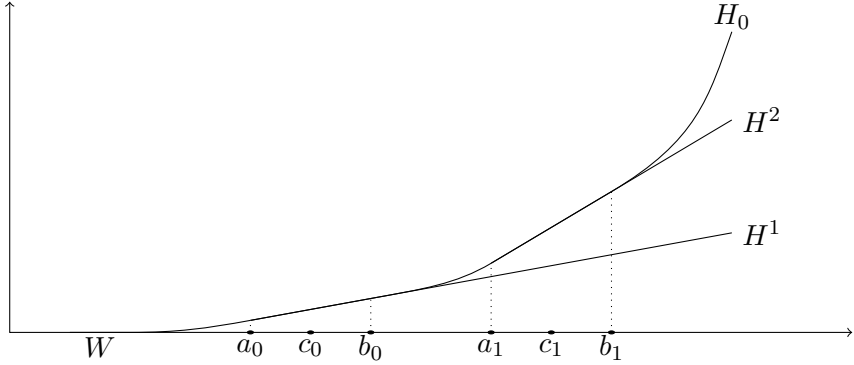
\begin{remark}
	A few remarks of the definition are in order.
	\begin{enumerate}
		\item $\partial_s H^i_N(t,z(s),x) < \frac{\epsilon}{2L_N}$  is used to make sure that the cochain complex of $H^i_N$ is a subcomplex of $H^{i+1}_N$ by action reasons, see Proposition \ref{prop:reg}. Such property can be arranged if $H^i$ is close to the autonomous $H_0$ on $W^i$.
		\item By \eqref{h4}, over critical point $\widetilde{z}_k$, we have $H^{i+1}|_{W^i}=H^i$.
		\item $H^{i,i+1}_{N}$ is used to build the $S^1$-morphism from the $N$-truncated $S^1$-complex of $H^i_N$ to $H^{i+1}_N$. The $s$-invariance on $W^i$ ensures that the induced $S^1$-morphism is the inclusion.
		\item We will not discuss any compatibility relation between $\bH_N,\bH_{N+1}$. We will only use one $\bH_N$ for some $N\gg 0$ depending on how many $\Delta^k_+,\Delta^k_{\partial}$ we want to identify in Theorem \ref{thm:ind}.
	\end{enumerate}
\end{remark}
For any $N$, one can inductively build a Hamiltonian data $\bH_N$ up to level $N$ by a $C^2$-small perturbation to the autonomous $H_0$ on each $W^i$. The motivation of using $H^i$ and continuation maps instead of using a global perturbation to $H_0$ is for the purpose of the integrated maximum principle. Since we are considering $s$-dependent Hamiltonians $H^i_N(t,z(s),x)$ for a gradient flow line $z(s)$, the integrated maximum principle requires more restrictive conditions. From now on, for every $N$, we fix a consistent Hamiltonian data $\bH_N$. Then a consistent almost complex structure data $\bJ_N$ is the following.
\begin{definition}\label{def:J}
	 A consistent almost complex structure data $\bJ_N$ up to level $N$ consists of the following data.
	 \begin{enumerate}
	 	\item For every $i\in \N$, a parameterized nice almost complex structure $J^i_N:S^1\times S^{2N+1}\to \cJ(W)$, such that $J^i_N(t+\theta,\theta\cdot z)=J^{i}_N(t,z)$, $S^1\times S^{2N+1}$ independent on $W$, and over slice $U_j\ni \widetilde{z}_j$, $J_N$ is a fixed $S^1$-dependent nice almost complex structure $J^i$. 
	 	\item For every $i\in \N$, a parameterized admissible almost complex structure $J^{i,i+1}_N:\R\times S^1\times S^{2N+1}\to \cJ(W)$, such that $J^{i,i+1}_N(s,t+\theta,\theta\cdot z)=J^{i,i+1}_N(s,t,z)$,  $J^{i,i+1}_N=J^i_N$ on $W^i$, $J_N^{i,i+1}=J_N^i$ for $s\gg 0$ and $J_N^{i,i+1}=J_N^{i+1}$ for $s\ll 0$.
	 	\item\label{J3} We impose the symmetry that $J^i_N$ is the same on $\CP^{j-i}=\{(0:\ldots:0:z_i:\ldots:z_j:0:\ldots:0)\}\simeq \{(0:\ldots:0:z_{i+k}:\ldots:z_{j+k}:0:\ldots:0)\}$	via translation, same property holds for $J^{i,i+1}_N$.
	 \end{enumerate}
	  The space of consistent almost complex structure data up to level $N$ is denoted by $\cJ_N(W)$.
\end{definition}

We fix an admissible Morse function $f$ on $W$. Let $\bH_N$ be a Hamiltonian data up to level $N$. Let $\cP^*(H^i)$ denote the set of non-constant orbits of $H^i$. Since by construction, we have $\cP^*(H^i)\subset \cP^*(H^{i+1})$. We define $\cP^*(\bH_N)=\cup \cP^*(H^i)$. Then we can assign grading to each element of $\cC(f)\cup \cP^*(\bH_N)$ by 
$$|p|=\ind p, p\in \cC(f), \quad |x| = n-\mu_{CZ}(x),x\in \cP^*(\bH_N).$$
We fix a metric $g$ on $W$, such that $(f,g)$ is a Morse-Smale pair. We also fix a $\bJ_N\in \cJ_N(W)$. Then we have the following moduli spaces.
\begin{enumerate}
	\item For $r\le N$, $i \in \N$ and $x,y \in \cP^*(H^i)$, we have
	\begin{equation}\label{eqn:M1}
	M^i_r(x,y):=\left\{\begin{array}{l} u:\R \times S^1 \to \widehat{W}, \\ z:\R \to S^{2N+1}\end{array}\left|\begin{array}{l} \partial_su+J^i_N(t,z(s),u)(\partial_tu-X_{H_N^{i}(t,z(s))}) = 0\\
	z'+\nabla_{\widetilde{g}_N}\widetilde{f}=0\\
	\displaystyle\lim_{s\to -\infty} (u,z) \in S^1\cdot (x,\widetilde{z}_{k+r}), \lim_{s\to \infty}(u,z)\in S^1\cdot (y,\widetilde{z}_{k}). 
	\end{array} \right. \right\}/\R\times S^1.
	\end{equation}
     By symmetry of $\bH_N,\bJ_N$ on $\CP^N$, the moduli space does not depend on $k$. The same independence property holds for the following moduli spaces.
	\item For $r\le N$, $i\in \N$ and $x\in \cP^*(H^i),p\in \cC(f)$, we have
	\begin{equation}\label{eqn:M2}
	M^i_r(p,x):=\left\{\begin{array}{l}	u:\R \times S^1 \to \widehat{W}, \\ z:\R \to S^{2N+1},\\ \gamma:(-\infty,0] \to W \end{array}\left|\begin{array}{l} \partial_su+J^i_N(t,z(s),u)(\partial_tu-X_{H_N^{i}(t,z(s))}) = 0\\ \gamma'+\nabla_g f= 0, z'+\nabla_{\widetilde{g}_N} \widetilde{f}_N =0,\\
	\displaystyle\lim_{s\to \infty} (u,z)\in S^1\cdot (x,\widetilde{z}_{k}), \lim_{s\to -\infty} u = \gamma(0),\\ \displaystyle \lim_{s\to -\infty} z \in S^1\cdot \widetilde{z}_{k+r}\lim_{s\to -\infty} \gamma = p 
	\end{array} \right. \right\}/\R\times S^1.
	\end{equation}
	Since $\partial_r f >0$ on $\partial W$, $\gamma(0)$ must be in the interior of $W$, therefore the Floer equation will become the Cauchy-Riemann equation near $s=-\infty$ as $H=0$ and $J$ is independent of $S^1\times S^{2N+1}$ there. In particular, we can view $u$ as a map from $\C$ to $\widehat{W}$ by removal of singularity.
	\item For $r\le N$, $i \in \N$ and $x \in \cP^*(H^{i+1}), y\in \cP^*(H^{i})$, we have
	\begin{equation}\label{eqn:N1}
	N_r^i(x,y):=\left\{\begin{array}{l} u:\R \times S^1 \to \widehat{W}, \\ z:\R \to S^{2N+1}
	\end{array}\left|\begin{array}{l} \partial_su+J^{i,i+1}_N(s,t,z(s),u)(\partial_tu-X_{H_N^{i,i+1}(s,t,z(s))}) = 0\\
	z'+\nabla_{\widetilde{g}_N}\widetilde{f}_N=0\\
	\displaystyle\lim_{s\to -\infty} (u,z) \in S^1\cdot (x,\widetilde{z}_{k+r}),\lim_{s\to \infty}(u,z)\in S^1\cdot (y,\widetilde{z}_{k}). 
	\end{array} \right. \right\}/ S^1.
	\end{equation}
	\item For $r\le N$, $i\in \N$, $x\in \cP^*(H^i)$ and $p \in \cC(f)$, we have 
	\begin{equation}\label{eqn:N2}
	N^i_r(p,x):=\left\{\begin{array}{l}u:\R \times S^1 \to \widehat{W},\\ z:\R \to S^{2N+1},\\
	\gamma:(-\infty,0] \to W \end{array}\left|\begin{array}{l} \partial_su+J^{i,i+1}_N(s,t,z(s),u)(\partial_tu-X_{H_N^{i,i+1}(s,t,z(s))}) = 0\\ \gamma'+\nabla_g f= 0, z'+\nabla_{\widetilde{g}_N} \widetilde{f}_N =0,\\
	\displaystyle\lim_{s\to \infty} (u,z)\in S^1\cdot (x,\widetilde{z}_{k}), \lim_{s\to -\infty} u = \gamma(0),\\ \displaystyle \lim_{s\to -\infty} z \in S^1\cdot \widetilde{z}_{k+r}\lim_{s\to -\infty} \gamma = p 
	\end{array} \right. \right\}/ S^1.
	\end{equation}
	\item For $p,q\in \cC(f)$, $M_0(p,q)$ is the moduli space of $-\nabla_g f$ flow lines from $p$ to $q$. $N_0(p,q)$ is the moduli space of parameterized $-\nabla_gf$ flow lines from $p$ to $q$. 
\end{enumerate}

\begin{figure}[H]
	\begin{center}
	\begin{tikzpicture}[scale=0.5]
	\draw (0,0) to [out=90, in=180] (2,1) to [out=0, in=90]  (4,0) to [out=270, in=0] (2,-1) to [out=180, in=270] (0,0); 
	\draw (0,0) to [out=270, in =180] (2,-4) to [out=0, in=270] (4,0);
	\fill (2,-4) circle[radius=2pt];
	\draw[->] (2,-4) to [out=270,in=160] (3.5,-5);
	\draw (3.5,-5) to [out=340,in=100] (5,-6);
	\node at (4,-4.7) {$\nabla_g f$};
	\fill (5,-6) circle[radius=2pt];
	\node at (5,-6.4) {$p$};
	\node at (2, 0) {$x$};
	\draw[dashed] (-2,-2) to [out=20, in=180] (2,-1.5) to [out=0, in=160] (6,-2);
	\node at (6,-2.5) {$\partial W$};
	\draw[dotted](2,-4) circle (2);
	\node at (2,-3) {$\overline{\partial}_J u=0$};
	\end{tikzpicture}
	\end{center}
	\caption{The $u$ part of $M^i_r(p,x),N^i_r(p,x)$ for $p\in \cC(f)$.}
\end{figure}
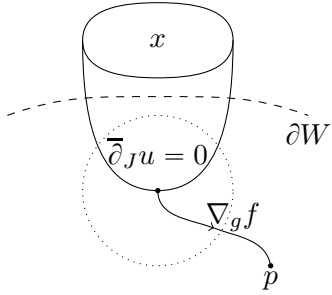

\begin{remark}\label{rmk:negative}
	One could also define $\cM^i_r(x,p),\cN^i_r(y,p)$ for $x\in \cP^*(H^i),y\in \cP^*(H^{i+1}),p\in \cC(f)$, but they are necessarily empty due to our choice of Hamiltonian (Definition \ref{def:ham}), see the proof of Proposition \ref{prop:reg}. This is the reason why have $S^1$-cochain complexes and morphisms compatible with splitting.
\end{remark}

Using the following integrated maximum principle \cite{abouzaid2010open}, we will be able to get compactness and exclude certain Floer cylinders.
\begin{lemma}[{\cite[Lemma 2.2]{cieliebak2018symplectic}}, {\cite[Lemma 2.5]{filling}}]\label{lemma:max}
	Let $(W,\omega)$ be a Liouville filling of $(Y,\alpha)$ with completion $(\widehat{W},\widehat{\omega})$. Let $H:\widehat{W}\to \R$ be a Hamiltonian such that $H=h(r)$ near $r=r_0$. Let $J$ be a $\widehat{\omega}$-compatible almost complex structure that is cylindrical convex on $Y\times(r_0,r_0+\delta)$ for some $\delta>0$. If both ends of a Floer cylinder $u$ are contained inside $Y\times \{r_0\}$, then $u$ is contained inside $Y\times \{r_0\}$. This also holds for $H_s$ depending on $s\in \R$ on $r\ge r_0$,  if $\partial_s H_s \le 0$ on $r\ge r_0$ and $H_s=h_s(r)$ on $(r_0,r_0+\delta)$ such that $\partial_s(r h'_s(r)-h_s(r))\le 0$ and $J_s$ is cylindrical convex on $Y\times (r_0,r_0+\delta)$.
\end{lemma}

\begin{proposition}\label{prop:compact}
	$M^i_{r}(x,y)$, $M^i_{r}(p,x)$, $N^i_r(p,x)$, $N^i_r(x,y)$ are contained in $W^i$ for $x\in \cP^*(H^i)$.  In general, $\cN^i_r(x,y)$ is contained in $W^{i+1}$.
\end{proposition}
\begin{proof}
	The first claim follows from  applying Lemma \ref{lemma:max} to $r=c_i$ as both $H_N^i,H_N^{i,i+1}$ are linear near $r=c_i$ and are independent of $\R\times S^1 \times S^{2N+1}$. To prove $\cN^i_r(x,y)$ is contained in $W^{i+1}$ for $x\in \cP^*(H^{i+1})$, note that on $\partial W \times (c_{i+1},c_{i+1}+\delta)$, we have $H_N^{i,i+1}=\rho(s)( h'(c_{i+1})r+(h(c_{i+1})-h'(c_{i+1})c_{i+1})) + (1-\rho(s))(h'(c_i)r+(h(c_i)-h'(c_i)c_i))$, hence
	$$\partial_s (r \frac{\rd}{\rd r} H_N^{i,i+1}- H_N^{i,i+1}) = \rho'(s) \left((h'(c_{i+1})c_{i+1}-h(c_{i+1}))-(h'(c_{i})c_{i}-h(c_{i}))\right)=\rho'(s)\int_{c_i}^{c_{i+1}} h''(r)r\rd r \le 0.$$
	Hence Lemma \ref{lemma:max} can be applied to $r=c_{i+1}$.
\end{proof}

We have the following standard regularization result. Whenever the indices do not appear in \eqref{eqn:M1}-\eqref{eqn:N2}, the moduli space is defined to be empty.
\begin{proposition}\label{prop:reg}
	There exists a second category subset $\cJ_{N,reg}^{\le 1}(W)\subset \cJ_N(W)$, such that for every $\bJ\in \cJ^{\le 1}_{N,reg}(W)$, the following holds.
	\begin{enumerate}
		\item\label{r1} For $x,y \in \cC(f)\cup \cP^*(H^i)$, 
		$$\cM^i_r(x,y):=M^i_r(x,y)\bigcup_{\substack{0\le k \le r,\\z\in \cC(f)\cup \cP^*(H^i)}} M^i_{r-k}(x,z)\times M^i_{k}(z,y)$$ is a compact manifold of boundary of dimension $|x|-|y|+(2r-1)$ whenever that is smaller than $2$, and $$\partial \cM^i_r(x,y)=\bigcup_{\substack{0\le k \le r,\\z\in \cC(f)\cup \cP^*(H^i)}} \cM^i_{r-k}(x,z)\times \cM^i_{k}(z,y).$$
		\item\label{r2} For $x\in \cC(f)\cup \cP^*(H^i), y\in \cC(f)\cup \cP^*(H^{i+1})$, $$\cN^i_r(x,y):=N^i_r(x,y)\bigcup_{\substack{0\le k \le r,\\z\in \cC(f)\cup \cP^*(H^i)}} N^i_{r-k}(x,z)\times M^i_{k}(z,y)\bigcup_{\substack{0\le k \le r,\\z\in \cC(f)\cup \cP^*(H^{i+1})}}M^{i+1}_{r-k}(x,z)\times N^i_k(z,y)$$ is a manifold with boundary of dimension $|x|-|y|+2r$ whenever that is smaller than $2$, and 
		$$\partial \cN^i_r(x,y)=\bigcup_{\substack{0\le k \le r,\\z\in \cC(f)\cup \cP^*(H^i)}}\cN^i_{r-k}(x,z)\times \cM^i_{k}(z,y)\bigcup_{\substack{0\le k \le r,\\z\in \cC(f)\cup \cP^*(H^{i+1})}}\cM^{i+1}_{r-k}(x,z)\times \cN^i_k(z,y).$$ Moreover, the only non-empty zero dimensional moduli spaces is when $x=y$ and $\cN^i_0(x,y)$ is the trivial cylinder over $x$.
	\end{enumerate}
\end{proposition}
\begin{proof}
	If $M_r^i(x,p)$ and $N^i_r(x,p)$ in Remark \ref{rmk:negative} are empty for $p\in \cC(f)$, then compactness and transversality for $\cM^i_r(x,y),\cN^i_r(x,y)$ follow from the same argument as \cite[Proposition 2.6, 2.8]{filling},  since a prior, $M_r^i(x,p)$ and $N^i_r(x,p)$ could appear in the compactification of $M_r^i(x,y)$ and $N_r^i(x,y)$. Note that if $u:\R \times S^1 \to \widehat{W}$ solves Floer equation $\partial_su+J(\partial_tu-X_{H_s})=0$ for a $s$-dependent Hamiltonian, then we have the energy of $u$ is 
	$$E(u):=\int_{\R\times S^1} |\partial_su|^2\rd s \rd t = \cA_{H_{-\infty}}(u(-\infty))-\cA_{H_{\infty}}(u(\infty)) + \int_{\R\times S^1}\partial_s H_s \rd s \rd t.$$ 
	Let $z$ be a gradient trajectories of $S^{2N+1}$, then $\partial_s H^i_N(t,z(s))$ is $0$ when $z(s)\in \cup V_i$ since $\nabla \widetilde{f}_N$ is tangent to $U_i$ and $H_N$ does not depend on $U_i$. Therefore we have
	$$\int_{\R\times S^1}\partial_s H^i_N(t,z(s),p)\rd s \rd t \le L_N \cdot \frac{\epsilon}{2L_N}<\epsilon, \quad \forall p\in \widehat{W}. $$  
	Therefore $M_r^i(x,p)$ and $N_r^i(x,p)$ must be empty for $p\in \cC(f)$ and $x\in \cP^*(\bH_N)$  by \eqref{h2} of Definition \ref{def:ham}. 
	
	To prove the last claim, for $x,y\in \cP^*(H^i)$,  since $\cN^i_r(x,y)$ stays inside $W_i$ by Proposition \ref{prop:compact} and $H_N^{i,i+1}$ is $s$-independent on $W_i$, the only nonempty zero dimensional $\cN^i_r(x,y)$ is when $r=0$ and $x=y$, and the moduli space contains only the trivial cylinder over $x$. If $x\in \cP^*(H^{i+1})\backslash \cP^*(H^i)$, note that  we have
	\begin{eqnarray*}
		\partial_s H^{i,i+1}_N(s,t,z(s)) & =  & \rho'(s)(H^{i+1}_N(t,z(s))-H^{i}_N(t,z(s)))+\rho(s)\partial_s H^{i+1}_N(t,z(s))+(1-\rho(s))\partial_s H^{i}(t,z(s)) \\
		& \le &  \rho(s)\partial_s H^{i+1}_N(t,z(s))+(1-\rho(s))\partial_s H^{i}(t,z(s))
	\end{eqnarray*}
	Then by the same argument as before, we have
	$$\int_{\R\times S^1} \partial_s H^{i,i+1}_N(s,t,z(s), p) \rd s \rd t\le 2 L_N \cdot \frac{\epsilon}{2L_N} = \epsilon, \quad \forall p \in \widehat{W}.$$
    Then $N^i_r(x,y)$ is empty for $x\in \cP^*(H^{i+1})\backslash \cP^*(H^i)$ by \eqref{h4} of Definition \ref{def:ham}. 
\end{proof}
Let $C(H^i,\bJ, f)$, $C_+(H^i,\bJ),C_0(f)$ be the free $\bm{k}$-modules generated by $\cP^*(H^i)\cup \cC(f)$, $\cP^*(H^i)$ and $\cC(f)$ respectively. Moduli spaces $\cM^i_r$ carry consistent orientations following \cite{zhao2014periodic,filling}. By signed count the $0$-dimensional moduli space $\cM^i_r$, we have a degree $1-2r$ operator
$$\delta^r(y)=\sum_x \# \cM_r^i(x,y)x.$$
Then by Proposition \ref{prop:reg}, $(\delta^r)_{0\le r \le N}$ defines a $N$-truncated $S^1$-structure on $C(H^i,\bJ,f)$ compatible with the splitting $0\to C_0(f)\to C(H^i,\bJ,f)\to C_+(H^i,\bJ)\to 0$. Counting of moduli spaces $\cN^i_r$ defines $S^1$-morphism from $C(H^i,\bJ,f)\to C(H^{i+1},\bJ,f)$ compatible with the splitting, which is inclusion by \eqref{r2} of Proposition \ref{prop:reg}. Then we defined $C(\bH,\bJ,f)$ to be the direct limit for $i\to \infty$. Then $C(\bH,\bJ,f)$ is a $N$-truncated $S^1$-complex with a compatible splitting. 
\begin{remark}\label{rmk:auto}
Conceptually, we are defining the structure on the limit of $H^i$, which can be thought as a $C^2$-small perturbation to $H_0$ capturing all Reeb orbits. We need to involve the complexity of Definition \ref{def:ham} instead of working with a global perturbation $H_N:S^1\times S^{2N+1}\times \widehat{W}\to \R$ to $H_0$, because there is no $r_0$ such that $\partial_s H_N(t,z(s))\le 0$ on $r\ge r_0$. In particular, Lemma \ref{lemma:max} can not be applied.  The integrated maximum principle can not be applied to $\cM^{i+1}_r(x,y)$ and $r=c_i$ for $x,y\in \cP^*(H^i)$ for the same reason. However, by Proposition \ref{prop:reg},   $C(H^i,\bJ,f)$ is a subcomplex in $C(H^{i+1},\bJ,f)$. In particular, it means the signed counting of curves in $\cM^{i+1}_r(x,y)$ leaving $W^i$ is zero for $x,y\in \cP^*(H^i)$.

One way to avoid using Definition \ref{def:ham} is using autonomous Hamiltonian $H_N\equiv H$ with a cascades construction of the $S^1$-complex. In this case, Lemma \ref{lemma:max} can be applied and we can get a full $S^1$-structure with a single Hamiltonian.
\end{remark}

Applying constructions in \S \ref{s2}, for $2k\le N$,  we have $\Delta^{k,i},\Delta^{k,i}_{+},\Delta^{k,i}_{+,0}$ defined by the $N$-truncated $S^1$-structure on $C(H^i,\bJ,f)$. 
\begin{proposition}
	We have $\Delta^k=\varinjlim_i \Delta^{k,i}, \Delta^k_+=\varinjlim_i \Delta^{k,i}_{+}, \Delta^k_{+,0}=\varinjlim_i \Delta^{k,i}_{+,0}$. 
\end{proposition}
\begin{proof}
 	Let $\widetilde{H}^i_N$ be a $C^2$-small perturbation to $H^i_N$ on $W$, such that $\widetilde{H}^i_N$ can be used to define the $N$-truncated $S^1$ structure on $C(\widetilde{H}^i)$ for non-degenerate $\widetilde{H}^i$ as in \S \ref{s31}. By the same argument as in \cite[Proposition 2.10]{filling}, there is a $N$-truanted $S^1$ morphism (a continuation map) from $C(\widetilde{H}^i) \to C(H^i)$, which induces isomorphism on the regular cohomology. Then by Proposition \ref{prop:func1}, this isomorphism induces isomorphism between $\Delta^k$. By the same argument as in Proposition \ref{prop:reg}, $C(\widetilde{H}^i)$ is a $N$-truncated $S^1$-subcomplex in $C(\widetilde{H}^{i+1})$ and the $S^1$-morphism $C(\widetilde{H}^i) \to C(H^i)$ can be made compatible with the inclusion. Then we have $\Delta^k:=\varinjlim \Delta^k(\widetilde{H}^i)=\varinjlim \Delta^{k,i}$. The proof for $\Delta^k_+,\Delta^k_{+,0}$ is the same as the construction above is compatible with splitting.
\end{proof}

\subsection{Naturality}
In this part, we discuss the naturality of the construction which is very important in the comparison argument when we do neck-stretching. 
\begin{definition}
Let $\bJ_1,\bJ_2$ be two consistent almost complex structures, then a homotopy $\bJ_s$ from $\bJ_1$ to $\bJ_2$ consists of $J^i_s:\R \times S^1 \times S^{2N+1}\to \cJ(W)$, such that $J^i_s(t+\theta,\theta\cdot z)=J^{i}_s(t,z)$, $J^i_s=J^i_{1,N}$ for $s\gg 0$ and $J^i_s=J^i_{2,N}$ for $s\ll 0$. The symmetric property in \eqref{J3} of Definition \ref{def:J} also holds for $\bJ_s$.
\end{definition} 

Then by choosing a generic homotopy and counting the moduli space similar to $\cN^i_r$ in \eqref{eqn:N1} and \eqref{eqn:N2}, we get an $S^1$-morphism $\phi_{J_s^i}$  compatible with the splitting from $C(H^i,\bJ_1,f)$ to $C(H^i,\bJ_2,f)$ for every $i$. Moreover, by a standard homotopy of homotopies argument, we have the following commutative diagram of $N$-truncated $S^1$-complex up to $S^1$-homotopy, which is also compatible with the splitting,

$$
\xymatrix{
	C(H^0,\bJ_1,f) \ar[r]\ar[d]^{\phi_{J^1_s}} & C(H^1,\bJ_1,f) \ar[r]\ar[d]^{\phi_{J^2_s}} & \ldots  \\
	C(H^0,\bJ_2,f) \ar[r] & C(H^1,\bJ_2,f) \ar[r] & \ldots	
	}
$$ 
Therefore by Proposition \ref{prop:func1} \ref{prop:func2}, they induce maps commuting with $\Delta^k,\Delta^k_+,\Delta^k_{+,0}$. Moreover, we have the following standard functorial property by a homotopy of homotopies of almost complex structures.
\begin{proposition}
	$\phi_{J^i_s}$ is independent of $\bJ_s$ up to $N$-truncated $S^1$-homotopy on each $C(H^i,\bJ,f)$ and is functorial with respect to concatenation of homotopies.
\end{proposition}

\begin{definition}
	We introduce the following sets of consistent almost complex structures, which satisfy various regularity conditions.
	\begin{enumerate}
		\item $\cJ_{N,reg}(W)$ denotes the set of regular consistent almost complex structures for moduli spaces up to dimension $0$, and $\cJ_{N,reg,+}(W)$ denotes the set of regular consistent almost complex structures for moduli spaces with ends asymptotic to non-constant orbits up to dimension $0$.
		\item $\cJ^i_{N,reg}(W)$ denotes the set of regular almost complex structures for $\cM^j_r,\cN^{j}_r$ for $j\le i$ up to dimension $0$. Similarly for $\cJ^i_{N,reg,+}(W)$. 
	\end{enumerate}
\end{definition}
Then by the compactness of $\cM^i_r,\cN^i_r$, we have that $\cJ^i_{N,reg},\cJ^i_{N,reg,+}$ are an open dense subsets. 

\begin{lemma}\label{lemma:natural}
	Assume we have a smooth family $\bJ_s\in \cJ^i_{N,reg,+}$, then the continuation map from $C_+^*(H^i,\bJ_1,f)$ to $C_+^*(H^i,\bJ_2,f)$ is identity up to $S^1$-homotopy. $\Delta^{k,i}_+, \Delta^{k,i}_{+,0}$ are well-defined for $\bJ \in \cJ^i_{N,reg,+}(W)$.
\end{lemma}
\begin{proof}
	The first claim follows from the same argument of \cite[Lemma 2.15]{filling}. A prior, $\Delta^{k,i}_+, \Delta^{k,i}_{+,0}$ are defined for $\bJ \in \cJ^{\le 1}_{N,reg}(W)$, which is dense in $\cJ_N(W)$. For $\bJ$ in the open dense set $\cJ^i_{N,reg,+}(W)$, $\Delta^{k,i}_+, \Delta^{k,i}_{+,0}$ are defined by a nearby $\bJ'\in \cJ^{\le 1}_{N,reg}(W)$ such that $\bJ,\bJ'$ are in an connected component of $\cJ^i_{N,reg,+}(W)$. The well-definedness of  $\Delta^{k,i}_+, \Delta^{k,i}_{+,0}$ follows from the first claim and Proposition \ref{prop:func2}.
\end{proof}
Working with almost complex structures such that only moduli spaces up to dimension zero are cut out transversely in the above lemma and Proposition \ref{prop:natural} below is important for the neck-stretching in \S \ref{s4} and the reason was explained in \cite[Remark 2.13, 3.13]{filling}

\subsection{Uniruledness}
In \cite{filling}, we proved that vanishing of symplectic cohomology and existence of symplectic dilation implies that $W$ is $(1,\Lambda)$-uniruled for $\Lambda \in \R_+$ in the sense of \cite{mclean2014symplectic}. Roughly speaking, $(1,\Lambda)$-uniruled means that for any almost complex structure $J$ that is convex near $\partial W$ and any point $e$ in the interior $W^{\circ}$, there exists a proper holomorphic curve $u:S\to W^\circ$ such that $\int u^*\omega \le \Lambda$ and $H_1(S;\Q)=0$. Since the concept of $k$-semi-dilation is a generalization of vanishing and dilation, by the exactly same argument as \cite[Theorem J]{filling}, we have the following.
\begin{theorem}\label{thm:unirule}
	Let $W$ be an exact domain admitting a $k$-semi-dilation, then $W$ is $(1,\Lambda)$-uniruled in the sense \cite{mclean2014symplectic} for some $\Lambda > 0$.
\end{theorem}
\begin{proof}
	Assume $W$ admits a $k$-semi-dilation but not a $k-1$-semi-dilation, i.e.\ $1+x \in \Ima \Delta^k_{+,0}$ on $C(H^j,\bJ,f)$ for some $j$ and $x$ is in $\oplus_{i>1}H^i(W)$. Let $e$ denote the unique minimum of $f$, which represents $1$ in $H^*(W)$. Therefore by the definition of $\Delta^k_{+,0}$, we can find $k+1$ periodic orbits $\{x_i\}_{0\le i \le k}$ of $X_{H^j}$, such that at least one the following equations has a solution for $u:\C \to \widehat{W}$ and $z:\R \to S^{2k+1}$ is a solution to $ z'+\nabla_{\widetilde{g}_N} \widetilde{f}_N=0$,
	$$\left\{\begin{array}{rl}
		\partial_su+J_N(t,z(s),u)(\partial_t-X_{H_N^j(t,z(s))}) = 0, & \displaystyle \lim_{s\to \infty}(u,z)\in S^1\cdot (x_0,\widetilde{z}_0), \lim_{s\to -\infty}z\in S^1\cdot \widetilde{z}_0, u(0)=e;  \\
		& \ldots  \\
	    \partial_su+J_N(t,z(s),u)(\partial_t-X_{H_N^j(t,z(s))}) = 0, & \displaystyle \lim_{s\to \infty}(u,z)\in S^1\cdot (x_k,\widetilde{z}_0), \lim_{s\to -\infty}z\in S^1\cdot \widetilde{z}_k, u(0)=e. 
	\end{array}\right.$$
	Then by the same maximum principle argument in \cite[Theorem 5.4]{filling}, $u^{-1}(\widetilde{W})$ contains a component $S$ (an open Riemann surface) such that $u:S\to \widetilde{W}$ is a proper holomorphic curve passing through $e$ with a uniform energy bound (the maximum of $|\cA_{H^j}|$ among all Hamiltonian orbits of $H^j$) and $H_1(S;\Q)=0$ for $\widetilde{W}:=W\backslash \partial W \times [1-\delta,1]$. Since $C(H^j,\bJ,f)$ carries a $k$-semi-dilation is a property independent of $f$ (this can be seen from the proof of \cite[Theorem 5.4]{filling}, or use a continuation map for a homotopy $f_s$ between admissible Morse functions),  in particular $e$ could be any point on $\widetilde{W}$. Therefore for every point of $\widetilde{W}$, we have found a proper holomorphic curve $u:S\to \widetilde{W}$ with a uniform energy bound passing through that point and $H_1(S;\Q)=0$. Hence $W$ is $(1,\Lambda)$-uniruled for some $\Lambda$ by \cite[Theorem 2.3]{mclean2014symplectic}.
\end{proof}
\begin{remark}
	It is unlikely that the existence of $k$-semi-dilation is equivalent to $(1,\Lambda)$-uniruled. Roughly speaking, the existence of $k$-semi-dilation implies a rigid rational curve passing through a fixed point and with another constraint at infinity. For the case of affine variety, often times, it can be explained by the non-vanishing of certain two point Gromov-Witten invariant, see \S \ref{s6}. In general, it is possible that the rational curve responsible for uniruledness is rigid only when we put more constraints. 
\end{remark}
Since the $(1,\Lambda)$-uniruledness in \cite{mclean2014symplectic} is equivalent to the algebraic $\A^1$-uniruledness, an instant corollary of Theorem \ref{thm:unirule} is the following.
\begin{corollary}\label{cor:unirule}
	Let $V$ be an affine variety. If one of the following conditions hold, then $V$ does not admits a $k$-semi-dilation for any $k$ and any coefficient $\bm{k}$.
	\begin{enumerate}
		\item The log Kodaira dimension of $V$ is not $-\infty$.
		\item $V$ admits a projective compactification $W$, such that $W$ is not uniruled.
	\end{enumerate}
\end{corollary}
Therefore, combining with the proof of \cite[Theorem 1.4]{li2019exact}, Corollary \ref{cor:unirule} proves the third part of \cite[Conjecture 5.1]{li2019exact}, i.e.\ log general variety never admits a cyclic dilation.

\subsection{Constructions preserving $k$-(semi)-dilations} 
We list three constructions that preserve $k$-(semi)-dilations. The first two of the following propositions are generalizations of the $0$ and $1$ dilations cases considered in \cite{oancea2006kunneth,seidel2012symplectic}.
\begin{proposition}\label{prop:prod}
	Let $V,W$ be two Liouville domains, then $\rD(V\times W) = \min \{\rD(V),\rD(W)\}$ and $\rSD(V\times W)\le \min \{\rSD(V),\rSD(W)\}$.
\end{proposition}
\begin{proof}
	Following \cite{oancea2006kunneth}, the symplectic cohomology $SH^*(V\times W)$ can be computed as $\varinjlim H^*(C(H^i+K^i))$, where $H^i,K^i$ are Hamiltonians on $V,W$ respectively with slope goes to $\infty$. Moreover, $C(H^i+K^i)=C(H^i)\otimes C(K^i)$ if we use product almost complex structure. The same holds for $S^1$-structures by the same argument, in particular, $\delta^r_{H^i+K^i}$ on $C(H^i)\otimes C(K^i)$ is given by $\delta^r_{H^i}\otimes \Id + \Id \otimes \delta^r_{K^i}$\footnote{The application of $\delta^r_{H^i}\otimes \Id + \Id \otimes \delta^r_{K^i}$ to $a\otimes b$ follows the Koszul rule, namely $(\delta^r_{H^i}\otimes \Id + \Id \otimes \delta^r_{K^i})(a\otimes b)=\delta^r_{H^i}(a)\otimes b + (-1)^{|a|}a\otimes \delta^r_{K^i}(b)$}. Let $e_W$ denote the unique minimum of $K^i$ on $W$, then $C(H^i)\to C(H^i)\otimes C(K^i),x\mapsto x \otimes e_W$ is an $S^1$-morphism as $\delta^r_{K^i}(e_W)=0$ for $r\ge 0$. The morphism $C(H^i)\to C(H^i)\otimes C(K^i)$ is compatible with splitting and sends $1\in H^*(C(H^i))$ to $1\in H^*(C(H^i))\otimes H^*(C(K^i))$. Therefore $\rD(V\times W)\le \rD(V),\rSD(V\times W)\le \rSD(V)$. Switching $V,W$, we have $\rD(V\times W)\le \min\{\rD(V),\rD(W)\},\rSD(V\times W)\le\min\{\rSD(V),\rSD(W)\}$. 
	
	We consider the spectral sequence $\{E_i=(Z_{i-1}/B_{i-1},\overline{\Delta}^i)\}_{i\ge 1}$ in \eqref{ss1} of Remark \ref{rmk:SS}, the spectral sequence associated to $C(H_t+K_t)$ is exactly the tensor product of the spectral sequences associated to $C(H_t)$ and $C(K_t)$ in the case of field coefficient. Since $1\in \Delta^k$ iff $1\in \overline{\Delta}^k$, $C(H_t+K_t)$ carries a $k$-dilation iff $C(H_t)$ or $C(K_t)$ carries a $k$-dilation. Therefore $\rD(V\times W)=\min \{\rD(V),\rD(W)\}$.
\end{proof}

\begin{proposition}\label{prop:lefchetz}
	Let $\pi:V^{2n}\to \C$ be a Lefschetz fibration with fiber $F$, such that $c_1(V)=0$.
	\begin{enumerate}
		\item If $F$ admits a $k$-semi-dilation, then $V$ also has a $k$-semi-dilation.
		\item If $F$ admits a $k$-dilation, and $n$ is odd or $n-2k>0$, then $V$ also has a $k$-dilation.
	\end{enumerate}
\end{proposition}
\begin{proof}
	We choose an admissible Hamiltonian $H_F:S^1\times S^{2N+1}\times \widehat{F}\to \R$ on the fiber $F$ as in Definition \ref{def:adm}, but first with the modification that $H_F=0$ on $F$. In particular, we can view $H_F$ as a function on $\widehat{V}$, since the conical ends of fibers form a trivial fibration over $\C$.  Pick a regular value $b$ of $\pi$, and define $H_B=\epsilon|z-b|^2$ for some sufficiently small $\epsilon>0$ on $\C$. Then $H_V=\pi^*H+H_F$ is well-defined Hamiltonian on $\widehat{V}$. Moreover, the periodic orbits of $H_V$ are points in the fiber $F_b$ over $b$, non-constant orbits in the conical ends of $\widehat{F_b}$, and critical points of $\pi$, which are Morse critical points of $\pi^*H_B$ of index $n$. By choosing an admissible Morse function on $F_b$\footnote{Or one can choose a perturbation supported near $F_b$ to make every orbits non-degenerate as in the proof of \cite[Lemma 7.2]{seidel2012symplectic}.}, the construction in \S \ref{s3} yields a $N$-truncated $S^1$-structure.  Then following \cite[Lemma 7.2]{seidel2012symplectic} and \cite[Lemma 5.3.8]{zhao2016periodic}, when asymptotics of a Floer solution are in $\widehat{F_b}$, we have the curve is contained in $\widehat{F_b}$. Then there is a short exact sequence of $N$-truncated $S^1$-cochain complexes and the following commutative exact sequences for $k\le N$ 
	$$\xymatrix{
		0 \ar[r] & \bm{k}^{Crit(\pi)}[-n]\otimes \la 1,\ldots,u^{-k}\ra \ar[r]^{\qquad \qquad  j} &  F^kC^+(H_V)\ar[r] & F^kC^+(H_F)\ar[r] & 0\\
		0 \ar[r] & \bm{k}^{Crit(\pi)}[-n]\otimes \la 1,\ldots,u^{-k}\ra \ar[r]\ar[u] & F^kC^+_0(H_V)\ar[r]\ar[u]^{i_V} & F^kC^+_0(H_F)\ar[r]\ar[u]^{i_F} & 0}$$
	where the bottom row is the part from constant orbits and is the short exact sequence from attaching $|Crit(\pi)|$ $n$-handles to $F\times \D$. 
	
	We assume $H_F$ already observes the $k$-semi-dilation. Since $c_1(F)=0$, there is a class $1+\sum_{i=1}^k u^{-i}a_i$ for $a_i\in H^{2i}(F)$, such that $\i_F(1+\sum_{i=1}^ku^{-i}a_i)=0$ in $H^*(F^kC^+(H_F))$. Then by the induced long exact sequence and that the left vertical arrow is an isomorphism, we have $1+\sum_{i=1}^k u^{-i}a_i$  mapped to zero in $\bm{k}^{Crit(\pi)}[-n]\otimes \la 1,\ldots,u^{-k}\ra$. In particular, $1+\sum_{i=1}^k u^{-i}a_i$ is closed in $F^kC_0^*(H_V)$ and $i_V(1+\sum_{i=1}^k u^{-i}a_i) \in \Ima j$, which implies there exists $B\in \bm{k}^{Crit(\pi)}[n]$, such that $i_V(1+\sum_{i=1}^k u^{-i}a_i+u^{-\frac{n}{2}}B)=0$, i.e.\ $H_V$ carries a $k$-semi-dilation. It is clear that we can build an $S^1$-morphism (continuation map) from the $S^1$-cochain complex of $H_V$ to the $S^1$-complex of the usual admissible Hamiltonians on $V$ with very large slope\footnote{The Hamiltonian $H_V$ does not detect many Reeb orbits on $\partial V$. In particular, since $\epsilon$ is sufficiently small, the Reeb orbits from the fixed points of the monodromy are not seen. In particular, $V$ could have smaller order of semi-dilation compared to $F$ when we consider all Reeb orbits of $V$, e.g. the trivial fibration $F\times \D$, where the $0$-dilation is supplied by $\D$ instead of $F$. }. In particular, $V$ also carries a $k$-semi-dilation.
	
	For the dilation case, it is sufficient to prove $u^{-\frac{n}{2}}B$ is zero in $\bm{k}^{Crit(\pi)}[-n]\otimes \la 1, \ldots, u^{-k}\ra$, which is automatic whenever $n$ is odd or $n>2k$. 
\end{proof}

\begin{proposition}\label{prop:flex}
	Assume Liouville domain $W$ is obtained by attaching a flexible Weinstein cobordism to $V$.
	\begin{enumerate}
		\item If $V$ admits a $k$-dilation, so does $W$ and $\rD(W)=\rD(V)$
		\item If $V$ admits a $k$-semi-dilation, and $H^{2i}(W)\to H^{2i}(V)$ are surjective for $i>0$, then $W$ admits a $k$-semi-dilation and $\rSD(W)=\rSD(V)$. In particular, this holds when $W$ is obtained by even-index subcritical/flexible handle attachments. If in addition we have $c_1(V)=0$, then the condition can be relaxed to that $H^{2i}(W)\to H^{2i}(V)$ are surjective for $0<i\le k$.
	\end{enumerate}
\end{proposition}
\begin{proof}
	We use $SH^*_{S^1,k}(W)$ to denote $\varinjlim_{H_t} H^*(F^kC^+(H_t))$. By \cite{bourgeois2012effect,cieliebak2002handle}, we have the Viterbo transfer map $SH^*(W)\to SH^*(V)$ is an isomorphism. Then $SH^*_{S^1,k}(W)\to SH^*_{S^1,k}(V)$ is also an isomorphism, since the first page of the spectral sequence from $u$-adic filtration is the regular symplectic cohomology. Therefore we have the following commutative diagram,
	$$
	\xymatrix{
		H^*(W)\otimes \la 1,\ldots,u^{-k}\ra \ar[r]\ar[d] & SH^*_{S^1,k}(W)\ar[d]^{\simeq}\\
		H^*(V)\otimes \la 1,\ldots, u^{-k}\ra \ar[r] & SH^*_{S^1,k}(V)
	}	
	$$
	$V$ admits a $k$-dilation iff $1$ is mapped to zero in $SH^*_{S^1,k}(V)$, hence $1\in H^*(W)$ is also mapped to $0$ in $SH^*_{S^1,k}(W)$. Therefore $\rD(W) \le \rD (V)$ and by Proposition \ref{prop:inc}, we have $\rD(W)=\rD(V)$. If $V$ admits a $k$-semi-dilation, then $1+\sum_{i=1}^k a_i u^{-i}$ is mapped to zero in $SH^*_{S^1,k}(V)$ for some $a_i\in \oplus_{j>0} H^{2j}(V)$\footnote{If $c_1(V)=0$, then $a_i\in H^{2i}(V)$ by the $\Z$-grading.}. By assumption $H^{2i}(W)\to H^{2i}(V)$ is surjective, there is some lift of $1+\sum_{i=1}^k a_i u^{-i}$ to $1+\sum_{i=1}^k \widetilde{a}_iu^{-i}$, such that it is mapped to $0$ in $SH^*_{S^1}(W,k)$. Hence $\rSD(W) = \rSD(V)$ as before.
\end{proof}

\section{Independence}\label{s4}
In this section, we will prove $\Delta^k_{\partial},\Delta^k_+$ are independent of fillings for ADC manifolds. The method has been carried out in \cite{filling} for $\Delta^0_{\partial}$ and $\Delta^1_{\partial}$. We will prove the invariance of $\Delta^k_{\partial}$ by rewriting it without using the Morse function $f$ on $W$. Then a neck-stretching argument with the ADC property implies that  $\Delta^k_{\partial},\Delta^k_+$ are independent of fillings. Since we will be using index properties, we assume $c_1(W)=0$ throughout this section.
\subsection{Shrinking the gradient flow}
In the following, we assume $\partial_r f > 0$ on $\partial W \times [1-\epsilon, 1]$. We pick a Morse-Smale pair $(h,g_{\partial})$ on $\partial W \times \{1-\epsilon\}$ such that \cite[Proposition 3.1]{filling} holds, i.e.\ the moduli space $\cR_{p,q}$ of concatenations of flow lines of $\nabla_g f$ and $\nabla_{g_\partial}h$ is cut out transversely. More precisely, $\cR(p,q)$ is the compactification of the following,
$$\left\{(\gamma_1,\gamma_2) \left| \begin{array}{l}
\frac{\rd}{\rd s}\gamma_1 + \nabla_{g_{\partial}} h = 0, \frac{\rd}{\rd s}\gamma_2 + \nabla_{g} f = 0, \\
\displaystyle\lim_{s\to -\infty} \gamma_1 = p,  \gamma_1(0)=\gamma_2(0), \lim_{s\to \infty} \gamma_2 = q \end{array}\right. \right\}, \quad p \in \cC(h),q\in \cC(f).$$
In particular, the restriction map $H^*(W) \to H^*(\partial W \times \{1-\epsilon\})$ can be defined by counting $\cR_{p,q}$. We define $\cP^i_r(x,p)$ and $\cH^i_{r}(x,p)$ to be the compactified moduli spaces of the following.
\begin{enumerate}
	\item For $x\in \cP^*(H^i),p\in \cC(h)$,
	\begin{equation}	
	P^i_r(p,x):=\left\{ \begin{array}{l} 
	u:\C \to \widehat{W},\\
	\gamma:\R_- \to \partial W\times \{1-\epsilon\},\\
	z:\R \to S^{2N+1} 
	\end{array} \left|\begin{array}{l}
	\partial_su+J_N^i(t,z(s),u)(\partial_tu-X_{H_N^i(t,z(s))})=0,\\
	z'+\nabla_{\widetilde{g}_N} \widetilde{f}_N=0, \gamma'+\nabla_{g_{\partial}} h = 0,\\
	\displaystyle\lim_{s\to \infty} (u,z)\in S^1\cdot (x,\widetilde{z}_k), \lim_{s\to-\infty} z \in S_1\cdot \widetilde{z}_{k+r},\\ \displaystyle u(0)=\gamma(0),\lim_{s\to -\infty} \gamma = p.
	\end{array}             
	\right.\right\}/\R\times S^1
	\end{equation}
	i.e.\ the fiber product over $W$ of the stable manifold of $\nabla_{g_{\partial}}h$ in $\partial W \times \{1-\epsilon \}$ associated to $p$ and the moduli space of $\{(u,z)|u:\C \to \widehat{W},z:\R \to S^{2N+1}\}$ solving the perturbed Cauchy-Riemann equation and the gradient flow equation with the listed asymptotics and $u(0)\in W^{\circ}$.
	\item For $x\in \cP^*(H^i),p\in \cC(h)$,
	\begin{equation}
	H^i_r(p,x):=\left\{ \begin{array}{l} 
	u:\C \to \widehat{W},\\
	\gamma_1:[0,l] \to W,l>0,\\\gamma_2:\R_-\to \partial W \times \{1-\epsilon\},\\
	z:\R \to S^{2N+1} 
	\end{array} \left|\begin{array}{l}
	\partial_su+J_N^i(t,z(s),u)(\partial_tu-X_{H_N^i(t,z(s))})=0, \\ \gamma_1'+\nabla_{g}f=0, \gamma_2'+\nabla_{g_\partial}h=0,z'+\nabla_{\widetilde{g}_N} \widetilde{f}_N=0,\\
	\displaystyle\lim_{s\to \infty} (u,z) \in S^1 \cdot (x,\widetilde{z}_k), \lim_{s\to -\infty} z \in S^1\cdot \widetilde{z}_{k+r},\\ u(0)=\gamma_1(l),\gamma_1(0)=\gamma_2(0),\displaystyle\lim_{s\to -\infty} \gamma_2 = p.
	\end{array}             
	\right.\right\}/\R\times S^1
	\end{equation}
\end{enumerate}
Then similar to \cite[Proposition 3.2]{filling}, besides the usual breaking, the boundary of $\cH^i_r$ also contains $\cP^i_r$ and $\cR\times \cM^i_r$ corresponding to $l=0$ and $l=\infty$. In particular, we have the following.
\begin{proposition}\label{prop:regP}
For generic choice of $\bJ$, we have $\cP^i_r(p,x)$ and $\cH^i_r(p,x)$ are manifolds with boundary of dimension $|p|-|x|+2r-1$ and $|p|-|x|+2r$ respectively, whenever they are smaller than $2$. And we have 
\begin{enumerate}
	\item $$\partial \cP^i_r(p,x)=\bigcup_{q\in \cC(h)} \cM_0(p,q)\times\cP^i_r(q,x)\bigcup_{\substack{0\le k \le r, \\ y\in \cP^*(H^i)}} \cP^i_k(p,y)\times\cM^i_{r-k}(y,x),$$
	\item $$\partial \cH^i_r(p,x)=\cP^i_{r}(p,x)\bigcup_{q\in \cC(h)} \cR(p,q)\times \cM^i_r(q,x)\bigcup_{q\in \cC(h)}\cM_0(p,q)\times\cH^i_r(q,x)\bigcup_{\substack{0\le k \le r,\\y\in \cP^*(H^i)}} \cH^i_k(p,y)\times \cM^i_{r-k}(y,x),$$
\end{enumerate}
where $\cM_0(p,q)$ is the compactified moduli space of $\nabla_{g_{\partial}}h$ flow lines from $q$ to $p$. 
\end{proposition}

Moreover, $\cP_r,\cH_r$ are oriented by a combination of \cite[Appendix A]{filling} and \cite{zhao2014periodic}. As a consequence of Proposition \ref{prop:regP}, by counting $\cP_r$ we have an $S^1$ morphism $P^i$ from $C_+(H^i)$ to $C_0(h)[1]$. The counting $\cup_{q\in \cC(f)} \cR(p,q)\times \cM_r(q,x)$ for $x\in \cP^*(H^i)$ is the composition $R\circ \delta^r_{+,0}$, where $R$ is the restriction map from the Morse cochain complex of $f$ to the Morse cochain complex of $h$. Moreover,  $\{R\circ \delta^r_{+,0}\}_{r\ge 0}$ form another $S^1$ morphism from $C_+(H^i)$ to $C_0(h)[1]$, which is $S^1$-homotopic to $P^i$ by counting $\cH_r$.  As a consequence, the following is also commutative up to homotopy
$$
\xymatrix{
C(H^i,\bJ,f) \ar[r]^{\subset} \ar[d]^{P^i}& C(H^{i+1},\bJ,f)\ar[d]^{P^{i+1}}\\
C_0(h)[1] \ar[r]^{\Id} & C_0(h)[1]
}
$$
Therefore by Proposition \ref{prop:func2}, we have $P^{k,i}=P^k(C(H^i,\bJ,f)) = \Delta^{k,i}_{\partial}$ and $\varinjlim_{i} P^{k,i}=\Delta^k_{\partial}$.

\begin{figure}[H]
	\begin{tikzpicture}
	\draw (0,0) to [out=90,in=180] (0.25,0.5);
	\draw (0.25,0.5) to [out=0,in=90] (0.5,0);
	\draw (0.5,0) to [out=270,in=0] (0.25,-0.5);
	\draw (0.25,-0.5) to [out=180,in=270] (0,0);
	\draw (0.25,-0.5) -- (-1,-0.5);
	\draw (0.25,0.5) -- (-1,0.5);
	\draw (-1.5,0) to [out=90,in=180] (-1,0.5);
	\draw (-1,-0.5) to [out=180,in=270] (-1.5,0);
	\draw[->] (-1.5,0) to [out=180, in=250](-2,1);
	\draw (-2,1) to [out=70, in=220](-1.5,1.5);
	\fill (-1.5,1.5) circle[radius=1pt];
	\node at (-1.5, 1.7) {$q$};
	\node at (-2.5,1) {$\nabla_g f$};
	\draw[dotted] (-0.5,5) -- (-0.5,-1);
	\draw[->] (-1.5,1.5) -- (-1, 1.5);
	\draw (-1,1.5) -- (-0.5, 1.5); 
	\draw[->] (-0.5,1.5) -- (-0.5, 2); 
	\draw (-0.5,2) -- (-0.5, 2.5);
	\fill (-0.5,2.5) circle[radius=1pt];
	\node at (-0.3, 2.5) {$p$};
	\node at (0, 2) {$\nabla_{g_{\partial}} h$};
	\node at (-0.5, 4) {$\partial W$};
	\node at (-0.5, -1.5) {$l=\infty$, $\Delta^k_{\partial}$};
	
	\draw (4,0) to [out=90,in=180] (4.25,0.5);
	\draw (4.25,0.5) to [out=0,in=90] (4.5,0);
	\draw (4.5,0) to [out=270,in=0] (4.25,-0.5);
	\draw (4.25,-0.5) to [out=180,in=270] (4,0);
	\draw (4.25,-0.5) -- (3,-0.5);
	\draw (4.25,0.5) -- (3,0.5);
	\draw (2.5,0) to [out=90,in=180] (3,0.5);
	\draw (3,-0.2) to [out=180,in=270] (2.5,0);
	\draw (3,-0.2) to [out=0, in=180]  (3,-0.5);
	\draw[->] (2.5,0) to [out=180, in=250](2,1);
	\draw (2,1) to [out=70, in=220](3.5,1.5);
	\draw[dotted] (3.5,5) -- (3.5,-1);
	\node at (3.5, 4) {$\partial W$};
	\draw[->] (3.5,1.5) -- (3.5, 2); 
	\draw (3.5,2) -- (3.5, 2.5);
	\fill (3.5,2.5) circle[radius=1pt];
	\node at (3.7, 2.5) {$p$};
	\node at (3.5,-1.5) {$0<l<\infty$};

	\draw (8,0) to [out=90,in=180] (8.25,0.5);
	\draw (8.25,0.5) to [out=0,in=90] (8.5,0);
	\draw (8.5,0) to [out=270,in=0] (8.25,-0.5);
	\draw (8.25,-0.5) to [out=180,in=270] (8,0);
	\draw (8.25,-0.5) -- (7,-0.5);
	\draw (8.25,0.5) -- (7,0.5);
	\draw (6.5,0.25) to [out=90,in=180] (7,0.5);
	\draw (7,0) to [out=180,in=270] (6.5,0.25);
	\draw (7,0) to [out = 0, in = 0] (6.2, -0.25);
	\draw (6.2, -0.25) to [out=180, in=90] (6, -0.35);
	\draw (6,-0.35) to [out=270, in=180](7,-0.5);
	\draw[dotted] (6.5,5) -- (6.5,-1);
	\node at (6.5, 4) {$\partial W$};
	\draw[->] (6.5,0.25) -- (6.5, 1.5); 
	\draw (6.5,1.5) -- (6.5, 2.5);
	\fill (6.5,2.5) circle[radius=1pt];
	\node at (6.7, 2.5) {$p$};
	\node at (6.5,-1.5) {$l=0$, $P^k$};
	\draw [decorate,decoration={brace,amplitude=10pt,mirror,raise=4pt},yshift=0pt]
	(-1,-1.5) -- (7,-1.5) node [black,midway,yshift=-0.8cm] {The $u$ part of the homotopy between $P^k$ and $\Delta^k_{\partial}$};
	
	\draw (14,0) to [out=90,in=180] (14.25,0.5);
	\draw (14.25,0.5) to [out=0,in=90] (14.5,0);
	\draw (14.5,0) to [out=270,in=0] (14.25,-0.5);
	\draw (14.25,-0.5) to [out=180,in=270] (14,0);
	\draw (14.25,-0.5) -- (13,-0.5);
	\draw (14.25,0.5) -- (13,0.5);
	\draw (12.5,0) to [out=90,in=180] (13,0.5);
	\draw (13,-0.5) to [out=180,in=270] (12.5,0);
	\draw[dotted] (12.5,5) -- (12.5,-1);
	\node at (12.5, 4) {$\partial W$};
	\draw[->] (12.5,0) -- (12.5, 1.5); 
	\draw (12.5,1.5) -- (12.5, 2.5);
	\fill (12.5,2.5) circle[radius=1pt];
	\node at (12.7, 2.5) {$p$};
	\node at (12.5,-1.5) {After neck-streching along $Y$};
	\draw[dotted] (10.5,5) -- (10.5,-1);
	\node at (10.5, 4) {$Y$};
	\end{tikzpicture}
	\caption{The $u$ part of a pictorial proof of Theorem \ref{thm:B}.}
\end{figure}
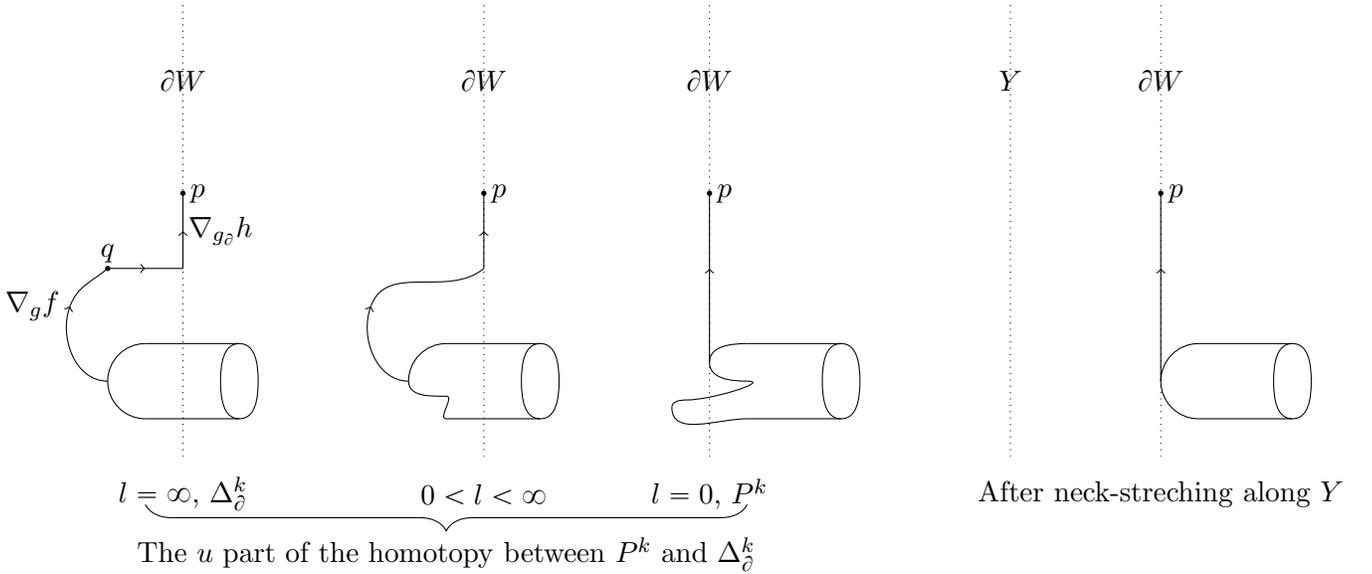

Next we will show that the equality holds with $\bJ$ of low regularity (i.e.\ only moduli spaces with low expected dimension are cut out transversely). We use $\cJ^{j}_{N,reg,P}(W)$ to denote the set of regular almost complex structures for $\cP_r^i$ with $i\le j$ and $r\le N$ of dimension up to $0$. Since the neck-stretching in Proposition \ref{prop:neck} only produces almost complex structures in  $ \cJ^{j}_{N,reg,+}(W)\cap\cJ^{j}_{N,reg,P}(W)$, we need to know that using such almost complex structures is sufficient to compute $\Delta^k_+,\Delta^k_{\partial}$.
\begin{proposition}\label{prop:natural}
	Let $\bJ^i\in \cJ^{i}_{N,reg,+}(W)\cap \cJ^i_{N,reg,P}(W)$, $\Delta^{k,i}_+,P^{k,i}$ denote the corresponding operators on $C_+(H^i,\bJ^i)$. Then $\Delta^k_+=\varinjlim \Delta^{k,i}_+$, $\Delta^k_{\partial}=\varinjlim P^{k,i}$, where the colimit is taken over the continuation maps for $C_+(H^1,\bJ^1)\to C_+(H^2,\bJ^2)\to \ldots$.
\end{proposition}
\begin{proof}
	The well-definedness of  $\Delta^{k,i}_+,P^{k,i}$ follows from the same argument of Lemma \ref{lemma:natural} and \cite[Proposition 3.3]{filling}, i.e.\ they are defined by a nearby regular enough $\bJ$ and different choices give rise to homotopic maps. The remaining of the claim follows from comparing with a regular enough $\bJ$ such that the associated  direct limit computes  $\Delta^{k}_{\partial},\Delta^k_+$ and the functoriality of continuation maps similar to \cite[Proposition 3.4]{filling}.
\end{proof}

\subsection{Neck-stretching and independence}
Asymptotically dynamically convex (ADC) contact manifolds were introduced in \cite{lazarev2016contact}. Let $(Y,\xi)$ be a $2n-1$-dimensional contact manifold with a contact form $\alpha$, then we use $\cP^{<D}(\alpha)$ to denote the set of contractible Reeb orbits of period smaller than $D$. If $c_1(\xi) = 0$, then for any contractible non-degenerate Reeb orbit $x$, there is an associated Conley-Zehnder index $\mu_{CZ}(x) \in \Z$. The (SFT) degree of $x$ is defined to be $\deg(x):=\mu_{CZ}(x)+n-3$.
\begin{definition}\label{def:ADC}
	Let $(Y,\xi)$ be a contact manifold. $Y$ is called ADC iff there exists a sequence of contact forms $\alpha_1>\ldots > \alpha_i >\ldots$ and real numbers $D_1 < \ldots < D_i < \ldots \to \infty$, such that all Reeb orbits in $\cP^{<D_i}(\alpha_i)$ are non-degenerate and have degree greater than $0$.  
\end{definition}
We refer readers to \cite{lazarev2016contact,filling} for examples of ADC manifolds and constructions preserving ADC properties. The major classes of ADC manifolds admitting $k$-dilations are cotangent bundles $T^*M$ for $\dim M >3$ and links of terminal singularity \cite{mclean2016reeb}. In order to make sure the Conley-Zehnder index property at the boundary is preserved in the filling, we need to consider the following filling.
\begin{definition}
	A filling of $W$ of $Y$ is called topologically simple iff $c_1(W)=0$ and $\pi_1(Y)\to \pi_1(W)$ is injective.
\end{definition}

Let $(Y,\alpha)$ be an ADC contact manifold with two topologically simple exact fillings $W_1,W_2$ with fixed consistent Hamiltonian data $\bH_1=\bH_2=\bH$ outside $W_1,W_2$ as in \S \ref{s3}. Note that $\widehat{W}_1, \widehat{W}_2$ both contain the symplectization $(Y\times (0,\infty), \rd(r\alpha))$. Since $Y$ is ADC, there exist contact type surfaces $Y_i\subset Y\times (0,1-\epsilon)$, such that $Y_i$ lies outside of $Y_{i+1}$ and contractible Reeb orbits of contact form $r\alpha|_{Y_i}$ has the property that the degree is greater than $0$ if period is smaller than $D_i$.

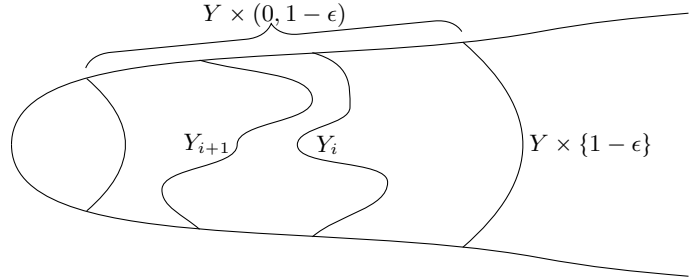
\begin{figure}[H]
	\begin{center}
		\begin{tikzpicture}[yscale=0.5]
		\draw (0,0) to [out=90, in =200] (7,3);
		\draw (7,3) to [out=20, in=190] (9,3.5);
		\draw (0,0) to [out=270, in =160] (7,-3);
		\draw (7,-3) to [out=340, in=170] (9,-3.5);
		\draw (6,2.73) to [out=300, in=60](6,-2.73);
		\draw (1,1.76) to [out=300, in=60](1,-1.76);
		\draw (4,2.45) to [out=330, in=90] (4.5,1) to [out=270, in=90] (3.8,0) to  [out=270, in=90](5,-1) to [out=270, in=60](4,-2.45);
		\draw (2.5,2.25) to [out=330, in=90](4,1.2) to [out=270, in=90](3,0) to [out=270, in=90](2,-1.2) to [out=270, in=120] (2.5,-2.25);
		\node at (4.2,0) {\footnotesize$Y_i$};
		\node at (2.6,0) {\footnotesize$Y_{i+1}$};
		\node at (7.7,0) {\footnotesize$Y\times \{1-\epsilon\}$};
		\draw [decorate,decoration={brace,amplitude=10pt,raise=4pt},yshift=0pt]
		(1,1.76) -- (6,2.73) node [black,midway,yshift=0.6cm] {\footnotesize$Y\times (0,1-\epsilon)$};
		\end{tikzpicture}
	\end{center}
	\caption{$Y_i\subset \widehat{W}_*$}
\end{figure}

Neck-stretching near $Y_i$ is given by the following. Assume $Y_i\times [1-\epsilon_i,1+\epsilon_i]_{r_i}$ does not intersect each other for some small $\epsilon_i$. Assume $J|_{Y_i\times [1-\epsilon_i,1+\epsilon_i]_{r_i}}=J_0$, where $J_0$ is independent of $S^1$ and $r_i$ and $J_0(r_i\partial_{r_i})=R_i,J_0\xi_i=\xi_i$. Then we pick a family of diffeomorphism $\phi_R:[(1-\epsilon_i)e^{1-\frac{1}{R}}, (1+\epsilon_i)e^{\frac{1}{R}-1}]\to [1-\epsilon_i,1+\epsilon_i]$ for $R\in (0,1]$ such that $\phi_1=\Id$ and $\phi_R$ near the boundary is linear with slope $1$. Then the stretched almost complex structure $NS_{i,R}(J)$ is defined to be $J$ outside $Y_i\times [1-\epsilon_i,1+\epsilon_i]$ and is $(\phi_R\times \Id)_*J_0$ on $Y_i\times [1-\epsilon_i,1+\epsilon_i]$. Then $NS_{i,1}(J)=J$ and $NS_{i,0}(J)$ gives almost complex structures on the completions of $W$ outside $Y_i$, inside $Y_i$ and the symplectization $Y_i\times \R_+$. Since we need to stretch along different contact surfaces, we assume the $NS_{i,R}(J)$ have the property that $NS_{i,R}(J)$ will modify the almost complex structure near $Y_{i+1}$ to a cylindrical almost complex structure for $R$ from $1$ to $\frac{1}{2}$, and for $R\le \frac{1}{2}$ we only keep stretching along $Y_i$. This explains the neck-stretching on a single almost complex structure, similarly we can apply neck-stretching on almost complex structure data $\bJ$ to get $NS_{i,R}(\bJ)$. We use $\cJ^{i}_{N,reg,SFT}(\bH,h,g_{\partial})$ to denote the set of regular $\bJ$, i.e.\ almost complex structures data satisfying Definition \ref{def:J} on the completion of $W$ outside $Y_i$ and asymptotic (in a prescribed way as in stretching process) to $J_0$ on the negative cylindrical end, such that the compactification of the following two moduli spaces up to virtual dimension $0$ is cut out transversely. We use $\widehat{V}_i$ to denote the completion of the cobordism from $Y_i$ to $\partial W$.

$$\left\{\begin{array}{l} u:\R \times S^1\backslash\{p_1,\ldots,p_k\} \to \widehat{V}_i, \\
z:\R \to S^{2N+1}
\end{array} \left| \begin{array}{l} 
\partial_s u + J(t,z(s),u)(\partial_tu-X_{H^i_{N}(t,z(s),u)})=0, \\
z'+\nabla_{\widetilde{g}_N} \widetilde{f}_N = 0,\\
\displaystyle \lim_{s \to \infty}(u,z)\in S^1\cdot(x,\widetilde{z}_i), \lim_{s\to -\infty}(u,z)\in S^1\cdot(y,\widetilde{z}_j),\\
\displaystyle\lim_{s\to -\infty} \epsilon_i^*u = (-\infty, \gamma_i), \gamma_i\in \cP(Y_i)
\end{array}\right. \right\}/\R\times S^1$$
$$\left\{\begin{array}{l} u:\R \times S^1\backslash\{p_1,\ldots,p_k\} \to \widehat{V}_i, \\
z:\R \to S^{2N+1},\\
\gamma:\R_- \to \partial W \times \{1-\epsilon\}
\end{array} \left| \begin{array}{l} 
\partial_s u + J(t,z(s),u)(\partial_tu-X_{H^i_{N}(t,z(s),u)})=0, \\
z'+\nabla_{\widetilde{g}_N} \widetilde{f}_N = 0,\\
\gamma'+\nabla_{g_{\partial}}h=0,\\
\displaystyle \lim_{s \to \infty}(u,z)\in S^1\cdot(x,\widetilde{z}_i), \lim_{s\to -\infty}u = \gamma(0) \\
\displaystyle \lim_{s\to -\infty} z \in S^1\cdot \widetilde{z}_j, \lim_{s\to \infty} \gamma = p\\
\displaystyle\lim_{s\to -\infty} \epsilon_i^*u = (-\infty, \gamma_i), \gamma_i\in \cP(Y_i)
\end{array}\right. \right\}/\R \times S^1$$
where $\epsilon_i:\R_-\times S^1 \to \R\times S^1\backslash\{p_i,\ldots,p_k\}$ is a cylindrical end around $p_i$, $x,y\in \cP^*(H^i)$ and $p\in \cC(h)$. 
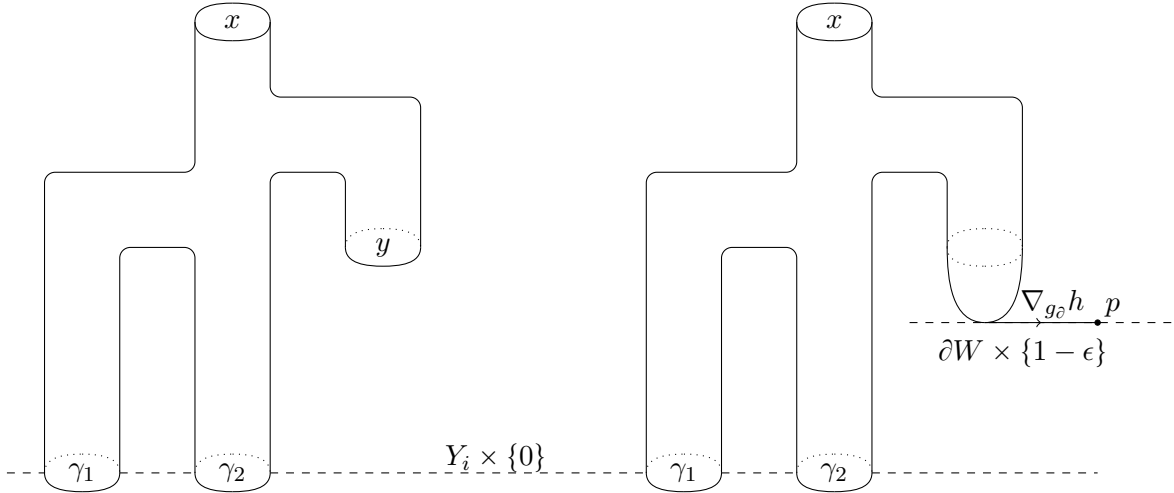
\begin{figure}[H]
	\begin{center}
		\begin{tikzpicture}
		\draw[rounded corners](0,0) -- (0,4) -- (2,4) -- (2,6);
		\draw[rounded corners](3,6)--(3,5)--(5,5)--(5,3);
		\draw[rounded corners](4,3)--(4,4)--(3,4)--(3,0);
		\draw[rounded corners](2,0)--(2,3)--(1,3)--(1,0);
		
		\draw (0,0) to [out=270,in=180] (0.5,-0.25);
		\draw (0.5,-0.25) to [out=0,in=270] (1,0);
		\draw[dotted] (0,0) to [out=90,in=180] (0.5,0.25);
		\draw[dotted] (0.5,0.25) to [out=0,in=90] (1,0);
		
		\draw (2,0) to [out=270,in=180] (2.5,-0.25);
		\draw (2.5,-0.25) to [out=0,in=270] (3,0);
		\draw[dotted] (2,0) to [out=90,in=180] (2.5,0.25);
		\draw[dotted] (2.5,0.25) to [out=0,in=90] (3,0);
		
		\draw (4,3) to [out=270,in=180] (4.5,2.75);
		\draw (4.5,2.75) to [out=0,in=270] (5,3);
		\draw[dotted] (4,3) to [out=90,in=180] (4.5,3.25);
		\draw[dotted] (4.5,3.25) to [out=0,in=90] (5,3);
		
		\draw (2,6) to [out=270,in=180] (2.5,5.75);
		\draw (2.5,5.75) to [out=0,in=270] (3,6);
		\draw (2,6) to [out=90,in=180] (2.5,6.25);
		\draw (2.5,6.25) to [out=0,in=90] (3,6);
		
		\node at (2.5,6) {$x$};
		\node at (4.5,3) {$y$};
		\node at (0.5,0) {$\gamma_1$};
		\node at (2.5,0) {$\gamma_2$};
		
		\draw[rounded corners](8,0) -- (8,4) -- (10,4) -- (10,6);
		\draw[rounded corners](11,6)--(11,5)--(13,5)--(13,3);
		\draw[rounded corners](12,3)--(12,4)--(11,4)--(11,0);
		\draw[rounded corners](10,0)--(10,3)--(9,3)--(9,0);

		\draw (8,0) to [out=270,in=180] (8.5,-0.25);
		\draw (8.5,-0.25) to [out=0,in=270] (9,0);
		\draw[dotted] (8,0) to [out=90,in=180] (8.5,0.25);
		\draw[dotted] (8.5,0.25) to [out=0,in=90] (9,0);
		
		\draw (10,0) to [out=270,in=180] (10.5,-0.25);
		\draw (10.5,-0.25) to [out=0,in=270] (11,0);
		\draw[dotted] (10,0) to [out=90,in=180] (10.5,0.25);
		\draw[dotted] (10.5,0.25) to [out=0,in=90] (11,0);
		
		\draw (10,6) to [out=270,in=180] (10.5,5.75);
		\draw (10.5,5.75) to [out=0,in=270] (11,6);
		\draw (10,6) to [out=90,in=180] (10.5,6.25);
		\draw (10.5,6.25) to [out=0,in=90] (11,6);
		
		\draw[dotted] (12,3) to [out=270,in=180] (12.5,2.75);
		\draw[dotted] (12.5,2.75) to [out=0,in=270] (13,3);
		\draw[dotted] (12,3) to [out=90,in=180] (12.5,3.25);
		\draw[dotted] (12.5,3.25) to [out=0,in=90] (13,3);
		
		\draw (12,3) to [out=270,in=180] (12.5,2);
		\draw (12.5,2) to [out=0,in=270] (13,3);
		
		\draw[->] (12.5,2) -- (13.25,2);
		\node at (13.4,2.25) {$\nabla_{g_{\partial}}h$};
		\draw (13.25,2) -- (14,2);
		\draw (14,2) circle[radius=1pt];
		\fill (14,2) circle[radius=1pt];
		
		\node at (10.5,6) {$x$};
		\node at (14.2,2.2) {$p$};
		\node at (8.5,0) {$\gamma_1$};
		\node at (10.5,0) {$\gamma_2$};
		
		\draw[dashed] (-0.5,0)--(0,0);
		\draw[dashed] (1,0)--(2,0);
		\draw[dashed] (3,0)--(8,0);
		\draw[dashed] (9,0)--(10,0);
		\draw[dashed] (11,0)--(14,0);
		\node at (6,0.2) {$Y_i\times \{0\}$};
		
		\draw[dashed] (11.5,2) -- (15,2);
		\node at (13,1.6) {$\partial W\times \{1-\epsilon\}$};
		\end{tikzpicture}
	\end{center}
	\caption{The $u$ part of the moduli spaces for the definition of  $\cJ^{i}_{N,reg,SFT}(\bH,h,g_{\partial})$.}\label{fig:SFT}
\end{figure}

Then by the same argument as in \cite[Proposition 3.12]{filling}, we have the following.
\begin{proposition}\label{prop:neck}
	Assume $Y$ is ADC with two topological simple exact fillings $W_1,W_2$, there exist consistent almost complex structures $\bJ^1_1,\bJ^2_1,\ldots,\bJ^1_2,\bJ^2_2,\ldots$ on $\widehat{W}_1$ and $\widehat{W}_2$ respectively and positive real numbers $\epsilon_1, \epsilon_2,\ldots$ such that the following holds.
	\begin{enumerate}
		\item\label{ns1} For $R<\epsilon_i$ and any $R'$, $NS_{i,R}(\bJ^i_*), NS_{i+1,R'}(NS_{i,R}(\bJ^i_*))\in \cJ^{i}_{N,reg,+}\cap \cJ^{i}_{N,reg,P}$, such that all zero dimensional $\cM^i_r(x,y)$ and $\cP^i_r(x,p)$ are the same for both $W_1,W_2$ and contained outside $Y_i$ for $x,y\in \cP^*(H^i)$ and $q\in C(h)$.   
		\item\label{ns2} $\bJ^{i+1}_*=NS_{i,\frac{\epsilon_i}{2}}(\bJ^i_*)$ on $W^i_*$. 
	\end{enumerate}
\end{proposition}

\begin{theorem}\label{thm:ind}
	Let $Y$ be an ADC contact manifolds, and $W_1,W_2$ be two topologically simple exact fillings of $Y$. For every $N\ge 0$, we have an isomorphism $\Gamma:SH^*_+(W_1)\to SH^*_+(W_2)$ (depends on $N$) such that  $\Gamma\circ \Delta^j_{+}=\Delta^j_{+}\circ \Gamma$ and $\Delta^j_{\partial}\circ \Gamma = \Delta^j_{\partial}$ for every $j\le N$.
\end{theorem}
\begin{proof}
	By Proposition \ref{prop:natural}, $\Delta^j_+, \Delta^{j}_{\partial}$ is represented by the limit of $\Delta^{j,i}_+,P^{j,i}$ on $C_+(H^i,NS_{i,\frac{\epsilon_i}{2}}\bJ^i_*)$ respectively and the limit is taken over the continuation maps $$C_+(H^1,NS_{1,\frac{\epsilon_1}{2}}\bJ^1_*) \to C_+(H^2,NS_{2,\frac{\epsilon_2}{2}}\bJ^2_*) \to \ldots$$ 
	By Proposition \ref{prop:neck}, $C_+(H^i,NS_{i,\frac{\epsilon_1}{2}}\bJ^i_*)$ can be identified for $W_1,W_2$ and the continuation map 
	\begin{equation}\label{eqn:con}
	C_+(H^i,NS_{i,\frac{\epsilon_i}{2}}\bJ^i_*) \to C_+(H^{i+1},NS_{i+1,\frac{\epsilon_{i+1}}{2}}\bJ^{i+1}_*)
	\end{equation}
	 can be decomposed into a composition of continuation maps up to $S^1$-homotopy
	 $$ C_+(H^i,NS_{i,\frac{\epsilon_i}{2}}\bJ^i_*) \to C_+(H^i,NS_{i+1,\frac{\epsilon_{i+1}}{2}}(NS_{i,\frac{\epsilon_i}{2}}\bJ^i_*)) \to C_+(H^{i+1},NS_{i+1,\frac{\epsilon_{i+1}}{2}}\bJ^{i+1}_*).$$ The first map is $S^1$-homotopic to identity by Lemma \ref{lemma:natural} and the second map is inclusion by \eqref{ns2} of Proposition \ref{prop:neck} and the argument in Proposition \ref{prop:reg}. As a consequence, \eqref{eqn:con} can be identified up to homotopy for $W_1,W_2$. Therefore $\Delta^{j}_+, \Delta^{j}_{\partial}$ are the same for both $W_1$ and $W_2$.
\end{proof}

\begin{remark}
	Theorem \ref{thm:ind} is a generalization of \cite[Theorem A, E]{filling}, which correspond to $k=0,1$. We can only match up finitely many structures at the same time, since we only constructed truncated $S^1$-structures. For the purpose of applications in this paper, finite identification is enough. Moreover, the identification $\Gamma$ is rather ad hoc and is not canonical. To identify all the structures at the same time, in addition to building the full $S^1$ structure (e.g. Remark \ref{rmk:auto}), Proposition \ref{prop:neck} may fail since we need to compare infinitely many moduli spaces in one neck-stretching process.
\end{remark}

\begin{remark}\label{rmk:generalization}
	Let $SH^*_{S^1,+,N}(W)$ denote the $\varinjlim_{H_t} H^*(F^NC^+_+(H_t))$. Then by the same argument of Theorem \ref{thm:ind}, we have $SH^*_{S^1,+,N}(W)\to H^{*+1}_{S^1,N}(\partial W)$ is independent of the topological simple exact filling $W$ as long as $\partial W$ is ADC for any $N$. As explained in Remark \ref{rmk:SS} and \ref{rmk:SSphi},  in principle Theorem \ref{thm:ind} can be used to reconstructed  $SH^*_{S^1,+,N}(W)\to H^{*+1}_{S^1,N}(\partial W)$ as a $\bm{k}$-linear map (in a non-conical way).
\end{remark}	

\subsection{Applications}
Apparently, the existence of $k$-semi-dilation is independent of topologically simple exact fillings if the contact boundary is ADC. For $k$-dilation we have the following.
\begin{corollary}[Corollary \ref{cor:C}]\label{cor:dilation}
	If $Y$ is an $ADC$ manifold such that there exists a topologically simple exact filling $W$ admitting a $k$-dilation. Then for any other topologically simple exact filling $W'$ such that $H^{2j}(W')\to H^{2j}(Y)$ is injective for every $j\le k$, then $W'$ also have $k$-dilation. In particular,  any Weinstein filling $W'$ of dimension at least $4k+2$ or $\frac{\dim W}{2}$ is odd has a $k$-dilation.
\end{corollary}
\begin{proof}
Assume $W$ carries a $k$-dilation, i.e.\ $1\in \Ima(SH^{-1}_{S^1,+,k}(W)\to H^0_{S^1,k}(W))$. By the injective property of  $H^{2j}(W)\to H^{2j}(Y)$ for $j\le k$, we have $H^0_{S^1,k}(W)\to H^0_{S^1,k}(Y)$ is injective. Then the invariance property of  $SH^{-1}_{S^1,+,k}(W)\to H^0_{S^1,k}(Y)$ in Remark \ref{rmk:generalization} implies $W'$ carries a $k$-dilation.

\end{proof}

\begin{corollary}\label{cor:ob}
	Let $Y$ be a $2n-1$-dimensional ADC manifold with a topologically simple exact filling $W$ for $n\ge 3$, such that $\Ima \Delta^k_{\partial}$ contains a class of degree greater than $n$, then $Y$ has no Weinstein filling.
\end{corollary}
\begin{proof}
	Since $\Delta^k_{\partial}$ factors through $H^*(W)$, the independence of $\Delta^k_{\partial}$ of topologically simple fillings (in particular, any Weinstein filling of dimension $\ge 6$ is topologically simple) implies that $Y$ has no Weinstein filling.
\end{proof}
\begin{remark}\label{rmk:ob}
By Remark \ref{rmk:SSphi}, the non-vanishing of the obstruction in Corollary \ref{cor:ob} is equivalent to that the image of $SH^*_{S^1,+,k}(W)\to H^{*+1}_{S^1,k}(\partial W)$ contains an element in the form of $x+u^{-1}y+\ldots$ with $x\in H^*(\partial W)$ of degree higher than $n$.	
\end{remark}
The obstruction is robust in the following sense.
\begin{proposition}\label{prop:robust}
	Let $Y_1$ and $Y_2$ be two ADC contact manifolds with topological simple exact fillings $V_1,V_2$. Assume the obstruction in Corollary \ref{cor:ob} exists for $Y_1$, then the obstruction exists for contact connected sum $Y_1\#Y_2$.
\end{proposition}
\begin{proof}
By Remark \ref{rmk:ob}, the obstruction exists for $Y_1$ is equivalent to that $H^*(V_1)\otimes \bm{k}[u,u^{-1}]/u\to SH^*_{S^1}(V_1)$ maps $\sum a_i u^{-i}$ to zero for $a_0\in H^*(V_1)$ such that $\deg(a_0)>n$ and $a_0|_{Y_1}\ne 0$. Let $V_1\natural V_2$ denote the boundary connected sum, i.e.\ attaching a Weinstein $1$-handle to connect $V_1$ and $V_2$. Then we have the following commutative diagram,
$$
\xymatrix{
H^*(V_1\natural V_2)\otimes \bm{k}[u,u^{-1}]/u \ar[r]\ar[d] & SH^*_{S^1}(V_1\natural V_2)\ar[d]^{\simeq}\\
\left(H^*(V_1)\otimes \bm{k}[u,u^{-1}]/u\right)\oplus \left(H^*(V_2)\otimes \bm{k}[u,u^{-1}]/u\right) \ar[r] &  SH^*_{S^1}(V_1)\oplus SH^*_{S^1}(V_2)
}$$
Since $\deg a_i=\deg a_0+2i>0$, we can view $a_i\in H^*(V_1\natural V_2)$. Then the diagram implies that $\sum a_i u^{-i}$ is mapped to zero in $SH^*_{S^1}(V_1\natural V_2)$. Since $Y_1$ is connected and $V_1$ is oriented, we have $H^{2n-1}(V_1)\to H^{2n-1}(Y_1)$ is zero, hence $a_0\notin H^{2n-1}(V_1)$. Therefore $a_0|_{Y_1}$ can be viewed as a nonzero class in $H^*(Y_1\# Y_2)=H^*(Y_1)\oplus H^*(Y_2)$ for $n<*<2n-1$. Then $V_1\natural V_2$ is a topological simple filling of $Y_1\#Y_2$ and the obstruction exists for $Y_1\#Y_2$.
\end{proof}

\begin{example}\label{ex:ob}
	Let $V$ be the $2n$ dimensional Liouville but not Weinstein domain in \cite{massot2013weak}, and $W$ be a $2m$ dimensional ADC domain with a $k$-dilation for any $k\ge 0$ in Proposition \ref{prop:ADC}. If $n-1>m$, then $\partial (V\times W)$ does not admit a Weinstein filling. The reason is that $V$ is hypertight, by the argument in Proposition \ref{prop:prod}, $\Delta^k_{+,0}$ for $V\times W$ hits the non-zero element in $H^{2n-1}(V)$. Since $V$ is in the form of $M\times [0,1]$, we have $H^*(V\times W) \to H^*(\partial(V\times W))$ is injective. In particular, $\Ima \Delta^k_{\partial}$ contains an element of degree $2n-1>n+m$ and $\partial(V\times W)$ is ADC by \cite[Theorem K]{filling}. Then by Proposition \ref{prop:robust}, let $Y$ be any other ADC contact manifold with a topological simple exact filling, then $\partial(V\times W)\# Y$ is also not Weinstein fillable.
\end{example}

\begin{remark}
	When $k=0,1$, the obstructions are the ones we defined in \cite{filling}. In $k=0$ case, we show that the obstruction is symplectic in nature by constructing almost Weinstein fillable examples with non-vanishing obstruction. In principle, all such obstructions are symplectic in nature. However, it is not easy to prove or disprove examples above and modifications of them (c.f.\ \cite[Theorem G]{filling}) are almost Weinstein fillable following the criterion in \cite{bowden2014topology}.
\end{remark}

Similar to \cite[Corollary H]{filling}, we have the following obstruction to cobordisms by Proposition \ref{prop:transfer}.
\begin{corollary}\label{cor:cob}
	Let $Y_1$, $Y_2$ be two contact manifolds with topologically simple exact fillings $V_1,V_2$, such that $\rSD(V_1)>\rSD(V_2)$. Assume $Y_2$ is ADC. Then there is no exact cobordism $W$ from $Y_1$ to $Y_2$ such that $V_1\cup W$ is topologically simple. In particular, if $V_1$ is Weinstein and $\dim V_1\ge 6$, then there is no Weinstein cobordism $W$ from $Y_1$ to $Y_2$. If $Y_1,Y_2$ are both simply connected, then there is no exact cobordism with vanishing first Chern class from $Y_1$ to $Y_2$.
\end{corollary}
\begin{proof}
	The first claim follows from $\rSD$ is independent of topological simple fillings for ADC contact manifolds. The Weinstein case follows from $V_1\cup W$ is Weinstein, hence is automatically topologically simple if dimension is at least $6$. The last case follows from the Mayer–Vietoris sequence that $H^2(W)\oplus H^2(V_1)\to H^2(V_1\cup W)$ is injective by $H^1(Y_1)=0$, hence $c_1(V_1\cup W)=0$ if $c_1(W),c_1(V_1)=0$. 
\end{proof}

\section{Examples and applications}\label{s6}
In this section, we discuss examples admitting $k$-dilations. In particular, we show that there are many examples with $k$-dilation, but not $k-1$-dilation. 
\subsection{Cotangent bundles}
A manifold $Q$ is called rationally-inessential, iff $H_n(Q;\Q)\to H_n(B\pi_1Q;\Q)$ vanishes for the classifying map $Q\to B\pi_1Q$ of the universal cover. In particular, any closed manifold with finite fundamental group is rationally-inessential. 
\begin{proposition}\label{prop:cot}
	If $Q$ is oriented, rationally-inessential, spin manifold, then $T^*Q$ admits a $k$-dilation for some $k$.
\end{proposition}
By Remark \ref{rmk:cyclic}, Proposition \ref{prop:cot} was proved by Li \cite[Proposition 4.2]{li2019exact} based on Zhao's higher dilation \cite[Definition 4.2.1]{zhao2016periodic}. $W$ admits a higher dilation iff $1$ is zero in the completed periodic symplectic cohomology $\widehat{PSH}(W)$. Then Proposition \ref{prop:cot} follows from that the map $H^*(W)\to SH^*_{S^1}(W)$ factors through $\widehat{PSH}(W)$.   However, by \cite[Lemma 4.2.5]{zhao2016periodic}, higher dilation is equivalent to that $u^{-k}$ is zero in $SH^*_{S^1}(W)$ for all $k$, hence is a stronger property than $k$-dilation.

\begin{remark}
	One can define $S^1$-equivariant symplectic cohomology with local systems. Then if $Q$ is not spin, but satisfy all other conditions in Proposition \ref{prop:cot},  then $T^*Q$ carries a $k$-dilation in a local system \cite{abouzaid2013symplectic}.
\end{remark}
\begin{question}
	For rational inessential manifold $Q$, what is $\rD(T^*Q)$?
\end{question}

\subsection{Brieskorn manifolds}\label{s52}
 Another source of examples comes from links of singularities. In the following, we provide a list of basic examples with $k$-dilation but not $k-1$ dilations. Moreover, by \cite{mclean2016reeb}, the link is ADC if the singularity is terminal. In particular, we have many examples where results from \S \ref{s4} can be applied, see Proposition \ref{prop:ADC}.
\begin{theorem}[Theorem \ref{thm:A}]\label{thm:main}
	Let $W_{k,m}$ denote the Milnor fiber of the singularity $x_0^k+x_1^k+\ldots+x_m^k=0$ for $m\ge k$, then $\rD(W_{k,m})=\rSD(W_{k,m})=k-1$.
\end{theorem}
\begin{remark}
	When $k=m=1$, the Milnor fiber is $\C$, which has a $0$-dilation. When $k=m=2$, the Milnor fiber is $T^*S^2$, which admits a $1$-dilation \cite{seidel2012symplectic}. When $k=m=3$, the Milnor fiber is considered in \cite[Theorem 2.2]{li2019exact} and proven to have a cyclic dilation with $h=1$ but not a dilation by an algebraic method. Since $SH^*_+(W_{3,3})$ is supported in grading $\le 3$. It must be a $2$-dilation. 
\end{remark}	
We will prove Theorem \ref{thm:main} by a direct computation of the $S^1$-cochain complex truncated by some action threshold. Let $X_{k,m}$ denote the projective variety $x_0^k+\ldots+x_m^k=x_{m+1}^k\subset \CP^{m+1}$, then $W_{k,m}\subset X_{k,m}$ is the complement of the divisor $\Sigma_{k,m}$ at infinity. Symplectic cohomology of complements of smooth divisors was studied by Diogo-Lisi \cite{diogo2019symplectic}, we will follow their Morse-Bott construction and give a brief review of their setup for our situation, which is a much simpler case than the general situation.

\subsubsection{Diogo-Lisi's model for symplectic cohomology}
The contact boundary $Y_{k,m}$ of $W_{k,m}$ is given the Boothby-Wang contact structure on the prequantization bundle over $\Sigma_{k,m}$. Let $\pi:Y_{k.m}\to \Sigma_{k,m}$ denote the projection. Hence the Reeb flow is the $S^1$ action with periods $2\pi \cdot \N$. We fix two Morse functions $f_{W}$ and $f_{\Sigma}$ on $W$ and $\Sigma$ respectively, and assume $\partial_r f_{W}>0$ near $Y_{k,m}$. If $m>3$, $W_{k,m},\Sigma_{k,m}$ are both simply connected of dimension at least $6$ and $H_*(W_{k,m}),H_*(\Sigma_{k,m})$ have no torsion, then we can assume both $f_{W}$ and $f_{\Sigma}$ are perfect Morse functions, i.e.\ every critical point represents a homology class. Since  $\Sigma_{3.3}$ is a smooth cubic surface, which is $\CP^2$ blown up at six points, $\Sigma_{2,3}=\CP^1\times \CP^1$ and $\Sigma_{2,2}=\CP^1$, we can assume the same thing for $m\le 3$. We fix a Morse function $f_Y$, such that critical points of $f_Y$ are in the $S^1$ fibers of critical points of $f_{\Sigma}$, and over each critical point of $f_{\Sigma}$ there are two critical points of $f_Y$. This can be done by perturbing $\pi^*f_{\Sigma}$. Let $Z_{\Sigma}$ be a gradient-like vector field on $\Sigma$ for $f_{\Sigma}$, such that the Morse-Smale condition is satisfied, i.e.\ the stable and unstable manifolds of $-Z_{\Sigma}$\footnote{We follow the notation in \cite{diogo2019symplectic}, i.e.\ $W^s/W^u$ is the stable/unstable manifold of $\bm{-Z}$.}
$$W^s_{\Sigma}(p):=\left\{q\in \Sigma_{k,m}| \displaystyle \lim_{t\to \infty}\phi^{-t}(q)=p \right\},\quad W^u_{\Sigma}(p):=\left\{q\in \Sigma_{k,m}| \displaystyle \lim_{t\to -\infty}\phi^{-t}(q)=p \right\}$$
intersect transversely, where $\phi^t$ is the time $t$ flow of $Z_{\Sigma}$. Let $\pi^*Z_{\Sigma}$ denote the lift of $Z_{\Sigma}$ to $Y_{k,m}$ in the contact distribution. We pick a gradient-like vector field $Z_Y$ of $f_Y$, such that $(f_Y,Z_Y)$ is a Morse-Smale pair and $Z_Y-\pi^*Z_{\Sigma}$ is vertical, i.e.\ flow lines of $Z_Y$ projects to flow lines of $Z_{\Sigma}$. We also pick a Morse-Smale pair $(f_W,Z_W)$ on $W_{k,m}$.

We fix a Hamiltonian $H_1$, which is $0$ on $W_{k,m}$ and is $h(r)$ on $Y_{k,m}\times (1,\infty)$ with $h'(r)>0$ and $h'(r)=2\pi+0.1$ for $r\gg 0$. Moreover, $h''(r)>0$ when $h'(r)=2\pi$. Therefore the non-constant periodic orbits of $X_{H_1}$ are $Y_{k,m}$ family of orbits contained in the level $r_0$, where $h'(r_0)=2\pi$, as it is parameterized by the initial point of the orbit on the hypersurface $r=r_0$. Then the cochain complex for the Hamiltonian-Floer cochain complex $C(H_1,f_Y, f_W)$ of $H_1$ is generated by $\cC(f_W)\cup \cC(f_{Y})$. Let $p\in \cC(f_{\Sigma})$, we use $\hat{p}\in \cC(f_Y)$ to denote the maximum in the fiber and $\check{p}\in \cC(f_Y)$ to denote the minimum in the fiber. Then the gradings are defined as follows.
\begin{equation}\label{eqn:ind}
|\check{p}|=2k-\ind(p)-3, \quad |\hat{p}|=2k-\ind(p)-4, \quad \forall p \in \cC(f_\Sigma),
\end{equation}
$$|p|=\ind(p), \quad \forall p \in \cC(f_W).$$
\begin{remark}\label{rmk:mono}
	Since we will be using cohomological convention, our grading is $n$ minus the grading in \cite[(3.4)]{diogo2019symplectic}, where their $\frac{\tau_X-K}{K}$ is $m+1-k\ge 1$ in our case. In particular, the monotonicity assumption in \cite{diogo2019symplectic} holds. Other notations in this section will be compatible in \cite{diogo2019symplectic}, in particular, the differential $\delta^0(p)=\sum \#\cM(p,q)q$ for moduli spaces $\cM(p,q)$ defined below. \textbf{We would like to warn that $\bm{\cM(p,q)}$ in this section was previously denoted by $\bm{\cM(q,p)}$ in \S \ref{s3}, \ref{s4} to be consistent with \cite{filling}.}
\end{remark}
Let $J$ be an $S^1$-independent admissible almost complex structure (which is sufficient for transversality as explained in Remark \ref{rmk:S^1}). The differentials for $C(H_1)$ is defined by counting the compactified moduli space $\cM(p,q)$ of the following two types of moduli spaces for $|q|-|p|=1$ for generic choice of $J,f$ and $f_Y$.
\begin{enumerate}
	\item For $p,q\in \cC(f_Y)$,
	\begin{equation}\label{eqn:cas3}M(p,q):=\left\{\gamma:\R \to Y_{k,m}|\gamma'-Z_{Y}=0,\displaystyle \lim_{s\to -\infty}\gamma=q,\lim_{s\to \infty}\gamma=p \right\}/\R.\end{equation}
	\item For $p\in \cC(f_Y),q\in \cC(f_W)$.
	\begin{equation}\label{eqn:cas4}
	M(p,q)=\left\{\begin{array}{l}
	u:\C \to \widehat{W}_{k,m},\\
	\gamma_1:\R_-\to W_{k,m},\\
	\gamma_2: \R_+ \to Y_{k,m} \times \{r_0\}.
	\end{array} \left| \begin{array}{l}
	\partial_su+J(\partial_tu-X_{H_1})=0,\\
	\gamma'_1+Z_W=0,\gamma'_2-Z_Y=0,\\
	\displaystyle\lim_{s\to -\infty } \gamma_1=q, \gamma_1(0)=u(0),\\
	\displaystyle \lim_{s\to \infty} \gamma_2=p, \lim_{s\to \infty}u(s,0)=\gamma_2(0).
	\end{array}
	\right. \right\}/\R.
	\end{equation}
\end{enumerate}
More precisely, in \eqref{eqn:cas4}, we need to consider those $u$ with finite energy, then $\lim_{s\to\infty} u$ converges to a Hamiltonian orbit. In particular, $\lim_{s\to \infty} u(s,0)=\gamma_2(0)$ is equivalent to that $\lim_{s\to \infty} u$ is the Hamiltonian orbit $\gamma$ such that $\gamma(0)=\gamma_2(0)$. We can also use the latter as the constraint, then it follows that $u$ has finite energy. The compactification $\cM_{p,q}$ involves adding broken Morse flow lines of $f_W,f_Y$ to $M_{p,q}$ at both ends.

In general,  \eqref{eqn:cas3} and \eqref{eqn:cas4} are the part of the differential of the presplit Floer complex in \cite{diogo2019symplectic}, which can be defined for any exact domain with a Morse-Bott contact form if we allow $J$ depending on $S^1$ following the setup and proofs in \cite{bourgeois2009symplectic}. Construction of such type first appeared in \cite{bourgeois2002morse} are called the cascades construction whose analytical foundation was discussed in detail in \cite{bourgeois2009symplectic}. The reason why we only need to consider \eqref{eqn:cas3} and \eqref{eqn:cas4} is that $C(H_1)$ only sees the first period, in particular, there is no multi-level cascades. 

\begin{remark}\label{rmk:S^1}
There are two ways to see why we choose $J$ to be independent of $S^1$: (1) Since we only consider simple orbits here, we have all the necessary some-injectivity, in particular \cite[Proposition 3.5 (i)]{bourgeois2009symplectic} can be applied; (2) Since the monotonicity assumption holds in our case by Remark \ref{rmk:mono}, as explained in \cite{diogo2019symplectic}, after neck-stretching, the moduli spaces \eqref{eqn:cas3}, \eqref{eqn:cas4} can be identified with some SFT moduli spaces that are cut out transversely with $S^1$-independent $J$. Hence by the openness of transversality, we can pick a sufficiently stretched $S^1$-independent $J$ such that \eqref{eqn:cas3}, \eqref{eqn:cas4} are cut out transversely when the expected dimension is at most $0$. 
\end{remark}

By neck-stretching along $Y_{k,m}$, under the monotonicity assumption, which in our case is equivalent to $m+1-k>0$, Diogo-Lisi \cite{diogo2019symplectic} showed that there exists an $S^1$-independent almost complex structure $J$, such that transversality for moduli spaces  \eqref{eqn:cas3}, \eqref{eqn:cas4} up to dimension zero holds and they can be identified with another SFT moduli space after fully stretching. In particular, by \cite[Theorem 9.1, Lemma 9.4]{diogo2019symplectic}, the only possible nontrivial differential on $C(H_1,f_Y,f_W)$ using such $J$ is given by
\begin{eqnarray}
\la \delta^0 \check{p} , \hat{q}\ra & = & \la c_1(Y_{k,m}), (W^s_{\Sigma}(q)\cap W^u_\Sigma(p))\ra=[H]\cap  [W^s_{\Sigma}(q)]\cap [W^u_\Sigma(p)] , \quad p,q \in \cC(f_{\Sigma})\label{eqn:d1}\\
\la \delta^0 \check{p}, x \ra & = &  \text{GW}_{0,2,A}([W^u_W(x)],[W^u_\Sigma(p)]), \quad p\in \cC(f_\Sigma), x \in \cC(f_W), \label{eqn:d2}
\end{eqnarray}
where $A\in H_2(X_{k,m};\Z)$ is the element that is mapped to the generator in $H^2(\CP^{m+1})$\footnote{It is possible that there are several $A$ that is mapped to the generator of $H_2(\CP^{m+1})$, e.g.\ $X_{2,2}$, then $GW_A$ should be interpreted as the sum of Gromov-Witten invariants for each such $A$.} and $c_1(Y_{k,m})\in H^2(\Sigma_{k,m},\Z)$ is the first Chern class of the $S^1$-bundle $Y_{k,m}$, $[H]$ is the codimension-$2$ homology class represented by the hyperplane class, i.e.\ the pushforward of the fundamental class of $H\cap \Sigma_{k,m}$ in $H_*(\Sigma_{k,m})$ for a generic hyperplane $H\subset \CP^{m+1}$, which is also the Poincar\'e dual of the first Chern class of the $S^1$-bundle $Y_{k,m}\to \Sigma_{k,m}$, and $\text{GW}_{0,2,A}([W^u_W(x)],[W^u_\Sigma(p)])$ is the 2-pointed Gromov-Witten invariants computed in $X_{k,m}$. 

In our case, there are only two Gromov-Witten invariants with expected dimension $0$, hence could be non-trivial, namely $\text{GW}_{0,2,A}([pt],[H^{m+1-k}])$, and $\text{GW}_{0,2,A}([S^m],[H^{\frac{3m}{2}+1-k}])$ when $m$ even and $m\le 2(k-1)$, where $[H^k]$ is the codimension-$2k$ homology class represented by intersection of $k$ generic hyperplanes with $X_{k,m}$ and $[S^m]$ is one of the sphere generators of $H_m(W_{k,m})$ as $W_{k,m}$ is homotopy equivalent the wedge product of $(k-1)^{m+1}$ spheres of dimension $m$. By \cite[(2.2)]{MR1484335}  and \cite[Proposition 1]{collino1999structure}, we have
\begin{eqnarray}
\text{GW}_{0,2,A}([pt],[H^{m+1-k}]) & = & k! \label{eqn:GW1}\\
\text{GW}_{0,2,A}([S^m],[H^{\frac{3m}{2}+1-k}]) & = & 0 \label{eqn:GW2}
\end{eqnarray}

\begin{remark}\label{rmk:primitive}
	To relate $[H^i]$ back to the unstable manifolds of critical points $p\in \cC(f_{\Sigma})$, first note that $[\Sigma_{k,m}]=[H]$ in $X_{k,m}$. Therefore $[H^{i+1}]$ in $X_{k,m}$ equals to the pushforward of $[H^i]$ in $\Sigma_{k,m}$, i.e.\ the intersection of $i$ generic hyperplanes in $\CP^m$ and $\Sigma_{k,m}\subset \CP^{m}$. Next we will determine the relation between $[W^u_\Sigma(p)]$ and $[H^i]$ in $\Sigma_{k,m}$.  Note that $[W^u_\Sigma(p)]$ is primitive\footnote{Not primitive classes in the hard Lefschetz theorem.} in $H_*(\Sigma_{k,m})$, i.e.\ $[W^u_\Sigma(p)]$ can not be written as a non-trivial multiple of another integral class. We claim that $[H^{n}]$ is a primitive class in homology if $n<\frac{m-1}{2}$ and is $k$ multiple of a primitive class if $n>\frac{m-1}{2}$. First of all, $H^*(\CP^m)\to H^*(\Sigma_{k,m})$ is an isomorphism for $*< m-1$ by the Lefschetz hyperplane theorem. Since the Poincar\'e dual of $[H^n]$ in $H^{2n}(\Sigma_{k,m})$ is the restriction of the positive generator of $H^{2n}(\CP^m)$ and Poincar\'e  duality preserves primitive classes, we have $[H^n]$ is primitive in $H_*(\Sigma)$ if $2n<m-1$. Now since $[H^n]\cap [H^{m-1-n}] = k[pt]$, we have $[H^{n}]$ is $k$-multiple of a primitive class if $n>\frac{m-1}{2}$. In particular, $[H^n]$ is represented by $[W^u_{\Sigma}(p)]$ for the unique $p$ with index $2m-2-2n$ for $2n<m-1$ and is represented by $k[W^u_{\Sigma}(p)]$ for the unique $p$ with index $2m-2-2n$ when $2n>m-1$. In the middle case $2n=m-1$, if $k$ is not divided by a square. then $[H^{n}]$ is primitive, however there are multiple critical points with index $m-1$.
\end{remark}

\subsubsection{Cascades model for $S^1$-cochain complexes}
In order to finish the computation, we need to build the $S^1$-structure compatible with $\delta^0$, i.e.\ using the cascades version of \eqref{eqn:M1} and \eqref{eqn:M2}. To motivate the moduli spaces that we need to consider for $\delta^i$ in our special case, which is oversimplified, we will review a more general setup as in \cite{bourgeois2009symplectic}. 

We assume the contact boundary of $W$ is Morse-Bott non-degenerate and  $H:\widehat{W}\to \R$ is an autonomous Hamiltonian such that all Hamiltonian orbits come in Morse-Bott non-degenerate families $\Sigma_1,\Sigma_2,\ldots$ and $\cA_H(\Sigma_1)>\cA_H(\Sigma_2)>\ldots$. Let $J$ be an $S^1$-dependent compatible almost complex structure and $f_i$ Morse function on $\Sigma_i$ with gradient like vector field $Z_i$ such that the Morse-Smale condition holds. Then the cascades Floer cochain complex is generated by critical points of $f_i$ with differential defined by counting rigid cascades as follows.
\begin{enumerate}
	\item A $0$-cascades from $p$ to $q$ is just $\gamma:\R\to \Sigma_i$ modulo the $\R$ translation such that $\gamma'-Z_i=0$, $\displaystyle \lim_{s\to -\infty}=q, \lim_{s\to +\infty}=p$.
	\item A $1$-cascades from $p\in \cC(f_i)$ to $q\in \cC(f_j)$ consists of 
	\begin{enumerate}
		\item $u:\R\times S^1\to \widehat{W}$  modulo the $\R$ translation such that $\partial_su+J_t(\partial_tu-X_H)=0$,
		\item $\gamma_1:\R_+\to \Sigma_i, \gamma_2: \R_- \to \Sigma_j$ such that $\gamma'_1-Z_i=0,\gamma'_2-Z_j=0$,
		\item $\displaystyle \lim_{s\to \infty} u=\gamma_1(0),\lim_{s\to -\infty } u=\gamma_2(0),\lim_{s\to \infty}  \gamma_1=p,\lim_{s\to -\infty } \gamma_2=q$.
	\end{enumerate}
    Then by action reasons, we must have $j<i$.
	\item A $2$-cascades from $p\in \cC(f_i)$ to $q\in \cC(f_j)$ consists of 
	\begin{enumerate}
		\item $u_1,u_2:\R\times S^1\to \widehat{W}$ modulo $\R$ translations such that $\partial_su+J_t(\partial_tu-X_H)=0$,
		\item an index $k$ and a positive number $l$,
		\item $\gamma_1:\R_+\to \Sigma_i,\gamma_2:[-l,0]\to \Sigma_k,\gamma_3:\R_-\to \Sigma_j$ such that $\gamma'_1-Z_i=0,\gamma'_2-Z_k=0,\gamma'_3-Z_j=0$,
		\item  $\displaystyle \lim_{s\to \infty} u_1=\gamma_1(0),\lim_{s\to -\infty } u_1=\gamma_2(0),\lim_{s\to \infty} u_2=\gamma_2(-l), \lim_{s\to -\infty} u_2=\gamma_3(0), \lim_{s\to \infty} \gamma_1=p,\lim_{s\to -\infty } \gamma_3=q$.
	\end{enumerate}
    Then by action reasons, we must have $j<k<i$, see the left of Figure \ref{fig:cascades}.
    \item In general, a $k$-cascades consists of $k$ Floer cylinders modulo translations and $k-1$ finite gradient flow lines connecting the asymptotics of Floer cylinders in between and two half infinite gradient flows at two ends asymptotic to the critical points.
\end{enumerate}

\begin{figure}[H]
	\begin{tikzpicture}[scale=1.50]
	\draw (0,0) to [out=90, in = 180] (0.5, 0.25) to [out=0, in=90] (1,0) to [out=270, in=0] (0.5,-0.25)
	to [out = 180, in=270] (0,0) to (0,-1);
	\draw[dotted] (0,-1) to  [out=90, in = 180] (0.5, -0.75) to [out=0, in=90] (1,-1);
	\draw (1,-1) to [out=270, in=0] (0.5,-1.25) to [out = 180, in=270] (0,-1);
	\draw (1,0) to (1,-1);
	\draw[-<] (1,-1) to (1.25,-1);
	\node at (1.25,-0.8) {$Z_k$};
	\draw (1.25,-1) to (1.5,-1);
	\draw (1.5,-1) to [out=90, in = 180] (2, -0.75) to [out=0, in=90] (2.5,-1) to [out=270, in=0] (2,-1.25)
	to [out = 180, in=270] (1.5,-1) to (1.5,-2);
	\draw[dotted] (1.5,-2) to  [out=90, in = 180] (2, -1.75) to [out=0, in=90] (2.5,-2);
	\draw (2.5,-2) to [out=270, in=0] (2,-2.25) to [out = 180, in=270] (1.5,-2);
	\draw (2.5,-1) to (2.5,-2);
	\draw[-<] (2.5,-2) to (2.75,-2);
	\node at (2.75,-1.8) {$Z_j$};
	\draw (2.75,-2) to (3,-2);
	\draw[-<] (-0.5,0) to (-0.25,0);
	\node at (-0.25,0.2) {$Z_i$};
	\draw (-0.25,0) to (0,0);
	\node at (0.5,-0.5) {$u_1$};
	\node at (2, -1.5) {$u_2$};
	\end{tikzpicture}
	\quad 
	\begin{tikzpicture}[scale=1.25]
	\draw (0,0) to [out=90, in = 180] (0.5, 0.25) to [out=0, in=90] (1,0) to [out=270, in=0] (0.5,-0.25)
	to [out = 180, in=270] (0,0) to (0,-1);
	\draw[dotted] (0,-1) to  [out=90, in = 180] (0.5, -0.75) to [out=0, in=90] (1,-1);
	\draw (1,-1) to [out=270, in=0] (0.5,-1.25) to [out = 180, in=270] (0,-1);
	\draw (1,0) to (1,-1);
	\draw[-<] (1,-1) to (1.25,-1);
	\node at (1.25, -0.8) {$Z_k$};
	\draw (1.25,-1) to (1.5,-1);
	\draw (1.5,-1) to [out=90, in = 180] (2, -0.75) to [out=0, in=90] (2.5,-1) to [out=270, in=0] (2,-1.25)
	to [out = 180, in=270] (1.5,-1) to (1.5,-2);
	\draw (2.5,-2) to [out=270, in=0] (2,-2.5) to [out = 180, in=270] (1.5,-2);
	\draw (2.5,-1) to (2.5,-2);
	\draw[-<] (-0.5,0) to (-0.25,0);
	\node at (-0.25,0.2) {$Z_i$};
	\draw (-0.25,0) to (0,0);
	\draw[->] (2,-2.5) to (2.5,-2.5);
	\draw (2.5,-2.5) to (3,-2.5);
	\node at (2.5, -2.8) {$Z_W$};  
	\node at (0.5,-0.5) {$u_1$};
	\node at (2, -1.5) {$u_2$};
	\end{tikzpicture}
	\caption{$2$ level cascades}\label{fig:cascades}
\end{figure}
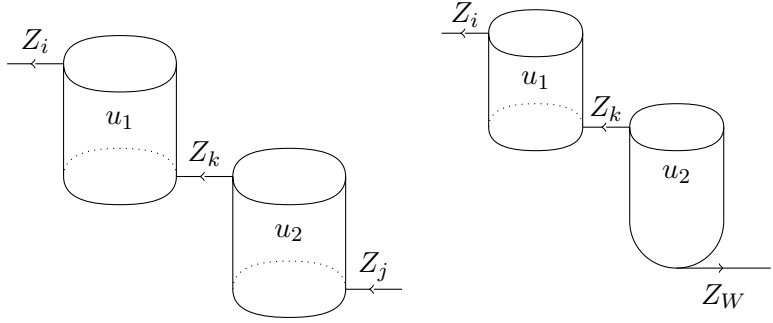

Let $M_{p,q}$ denote the moduli space of cascades (with any number of levels) from $p$ to $q$, then $M_{p,q}$ admits a compactification $\cM_{p,q}$ by adding broken cascades, where a cascades can break when one of the $u$-components breaks or one of the $\gamma$-components breaks or one of the finite lengths $l$ converges to $0$. The breaking of a $u$-component will be identified with a length converges to $0$. Using generic choices of $J,Z_i$, $\cM_{p,q}$ is cut out transversely and the zero dimensional moduli spaces $\cM_{p,q}= M_{p,q}$, whose count defines a differential.

When $H=0$ on $W$ as in \S \ref{s32}, the constant orbits in $W$ do not form a Morse-Bott family (the Morse-Bott non-degeneracy conditions fails along $\partial W$). However, since such orbits have maximal symplectic action, the failure along $\partial W$ can not been seen by Floer cylinders \cite[Proposition 2.6]{filling}. Then by picking an admissible Morse function $f_W$ on $W$ and treat $W$ as a Morse-Bott family of orbits, we can define a cascades model\footnote{But use $-Z_W$, as we are following the convention in \cite[Definition 4.7]{diogo2019symplectic}.} for the Floer cochain complex using such $H$, see right of Figure \ref{fig:cascades} for the cascades whose bottom cascades land in $W$. In particular, \eqref{eqn:cas3} is $0$-cascades and \eqref{eqn:cas4} is $1$-cascades and they are the only cascades we need to consider for $H_1$, since we only have two Morse-Bott families of orbits for $H_1$. The equivalence between the cascades model and the usual construction can be seen from the degeneration and gluing analysis in \cite{bourgeois2009symplectic} or a cascade continuation map, see e.g.\ \cite[Proposition 5.1]{diogo2019symplectic}.

Next we will adapt the cascades construction to the $S^1$-structure. We define $H_N:S^1\times S^{2N+1}\times \widehat{W}\to \R$ independent of $S^1\times S^{2N+1}$ and equal to $H$, which satisfies all the conditions in Definition \ref{def:ham} except that $H_N$ over critical points is not non-degenerate but only Morse-Bott non-degenerate. Let $J$ be a family of compatible almost complex structure as in Definition \ref{def:ham}. We also need to consider Morse-Smale pairs $(f^z_i,Z^z_i)$ equivariantly parametrized critical point $z$ of $\widetilde{f_N}$. For our purpose, we can simply consider $f^{\widetilde{z}_k}_i=f_i,Z^{\widetilde{z}_k}_i=Z_i$ and define  $f^{\theta \cdot \widetilde{z}_k}_i=\theta_*f_i, Z^{\theta \cdot \widetilde{z}_k}_i=\theta_*Z_i$.

Then we introduce the cascades for $\delta^r$ for $r\ge 1$ as follows.
\begin{enumerate}
	\item The $1$-cascades space for $\delta^r$ from $p\in \cC(f_i)$ to $q\in \cC(f_j)$ is the following
	\begin{equation}\label{cas:r1}
	  M_{r,1}(p,q):= \left\{
	  \begin{array}{l}
	  u:\R\times S^1\to \widehat{W},\\
	  z:\R \to S^{2N+1},\\
	  \gamma_1:\R_+\to \Sigma_i,\\
	  \gamma_2:\R_- \to \Sigma_j
	  \end{array}\left|
	  \begin{array}{l}
	  \partial_s u+J_N(t,z(s))(\partial_t u-X_H) = 0,\\
	  z'+\nabla_{\widetilde{g}_N}\widetilde{f}_N = 0,\\
	  \displaystyle z_-:=\lim_{s\to -\infty} z \in S^1\cdot \widetilde{z}_r,  z_+:=\lim_{s\to \infty} z \in S^1\cdot \widetilde{z}_0,\\ 
	  \theta_+\cdot \widetilde{z}_0=z_+,  \theta_-\cdot \widetilde{z}_r=z_- \\
	  \gamma'_1-(\theta_+)_*Z_i=0,\gamma'_2-(\theta_-)_*Z_j = 0,\\
	  \displaystyle \lim_{s\to \infty} u = \gamma_1(0), \lim_{s\to -\infty} u = \gamma_2(0),\\
	  \displaystyle \lim_{s\to \infty} \gamma_1 = \theta_+\cdot p, 
	 \displaystyle \lim_{s\to \infty} \gamma_2 = \theta_-\cdot q.
	  \end{array}
	  \right.   \right\}/\R\times S^1
	\end{equation}
	where $(s,t)\in \R\times S^1$ acts on $(u,z)$ by $(u(\cdot + s, \cdot + t), t\cdot z(\cdot + s))$ and $\theta$ acts on $\Sigma$ by the reperamretrization in $S^1$. Again by action reasons we must have $j\le i$. Note that now we can have $i=j$, then $u$ is necessarily a trivial cylinder by action reasons, but the $\R$ the action is still free due to the extra flow line $z$. 
	\item The $2$-cascades space $M_{r,2}(p,q)$ for $\delta^r$ from $p\in \cC(f_i)$ to $q\in \cC(f_j)$ is the following
	\begin{equation}\label{cas:r23}
	\left\{
	\begin{array}{l}
	u_1,u_2:\R\times S^1\to \widehat{W},\\
	z_1,z_2:\R \to S^{2N+1},\\
	\gamma_1:\R_+\to \Sigma_i,\\
	\gamma_2:[-l,0] \to \Sigma_k,\\
	\gamma_3:\R_-\to \Sigma_j
	\end{array}\left|
	\begin{array}{l}
	\partial_s u_1+J_N(t,z_1(s))(\partial_t u_1-X_H) = 0,\\
	\partial_s u_2+J_N(t,z_2(s))(\partial_t u_2-X_H)=0,\\
	z'_1+\nabla_{\widetilde{g}_N}\widetilde{f}_N = 0, z'_2+\nabla_{\widetilde{g}_N}\widetilde{f}_N=0\\
	\displaystyle z_0:=\lim_{s\to -\infty} z_1 = \lim_{s\to -\infty} z_2 \in S^1\cdot \widetilde{z}_v, 0\le v \le r,\\
	\displaystyle z_+:=\lim_{s\to \infty} z_1 \in S^1\cdot \widetilde{z}_0,z_-:=\lim_{s\to -\infty} z_2 \in S^1\cdot \widetilde{z}_0\\ 
	\theta_+\cdot \widetilde{z}_0=z_+, \theta_-\cdot \widetilde{z}_r=z_-, \theta_0\cdot \widetilde{z}_v=z_0,\\
	\gamma'_1-(\theta_+)_*Z_i=0, \gamma'_2-(\theta_0)_*Z_k = 0, \gamma'_3-(\theta_-)_*Z_j=0\\
	\displaystyle \lim_{s\to \infty} u_1 = \gamma_1(0), \lim_{s\to -\infty} u_1 = \gamma_2(0),\\
	\displaystyle \lim_{s\to \infty} u_2 = \gamma_1(-l), \lim_{s\to -\infty} u_2 = \gamma_3(0), \\
	\displaystyle \lim_{s\to \infty} \gamma_1 = \theta_+\cdot p, 
	\displaystyle \lim_{s\to \infty} \gamma_3 = \theta_-\cdot q.
	\end{array}
	\right.   \right\}/\R^2\times S^1
	\end{equation}
	where $(s_1,s_2,t)\in \R^2 \times S^1$ acts on $(u_1,z_1, u_2,z_2)$ by $(u_1(\cdot + s_1, \cdot + t), t\cdot z_1(\cdot + s_1), u_2(\cdot + s_2, \cdot + t), t\cdot z_2(\cdot + s_2))$. In particular, when $v=0$, $u_1$ solves the parameterized Floer equation and $u_2$ solves the Floer equation over $z_-$ and the situation is reversed when $v=r$. 
		
	\item The $k$-cascades moduli space for $\delta^r$ consists of (1) $k$ $(u,z)$-components and if a $u$-component is trivial cylinder then the corresponding $z$-component must be non-trivial so that the $\R$ action is non-trivial, (2) $k-1$ finite flow lines of $Z^*_*$ in the middle, (3) two half infinite flow lines of $Z^{z_+}_i=(\theta_+)_*Z_i,Z^{z_-}_j=(\theta_-)_*Z_j$, and they match by the obvious asymptotic conditions. 
\end{enumerate}

Let $M_r(p,q):=\cup_{i\ge 1} M_{r,i}(p,q)$ be the cascades space for $\delta^r, r\ge 1$. Then $M_r(p,q)$ admits a compactification $\cM_r(p,q)$ by adding broken cascades, where a cascades can break when one of the $(u,z)$-components breaks,  or one of the $\gamma$-components breaks (including one of the finite lengths $l$ goes to $\infty$) or one of the finite lengths $l$ converges to $0$. And we will identify the $l=0$ case with a breaking from a $(u,z)$-component as before. And for generic choice of $J_N,f_i,Z_i$, $\cM_r(p,q)$ is cut out transversely up to dimension zero and the count of zero dimensional moduli spaces $\cM_r(p,q)=M_r(p,q)$  defines an $S^1$-structure. One can similarly adapt the cascades construction to the case where $H=0$ on $W$ as before. The equivalence between the cascades construction and the usual construction again follows from the gluing argument in \cite{bourgeois2009symplectic} or a cascade continuation map/isomorphism in \cite{diogo2019symplectic}.

\begin{proposition}\label{prop:S1} In addition to the $\delta^0$ given in \eqref{eqn:d1},\eqref{eqn:d2}, the $S^1$-structure on $C(H_1)$ is given by $\delta^1(\check{p})=\hat{p}$ and $\delta^i=0$ for $i>1$.
\end{proposition}
\begin{proof}
	We will apply the above cascades construction of $S^1$-cochain complex to $H_1,f_W,f_Y$ and $J_N=J$ independent of $S^1\times S^{2N+1}$. A prior, such $J$ may not guarantee transversality of the cascades moduli space. But we will show that in our special case, such $J$ is sufficient and will imply claim. 
	
	Since we $H_1$ has two families of orbits, namely constant orbits parameterized by $W_{k,m}$ and simple Hamiltonian orbits parametrized by the circle bundle $Y_{k,m}$, where the $S^1$-action of reparametrization is identified with the natural $S^1$ action on $Y_{k,m}$. 
	
	\textbf{Step 1: cascades space for $\delta^r$ from $p\in \cC(f_Y)$ to $q\in \cC(f_Y)$ with $r\ge 1$.} We first consider the $1$-cascades case, namely 
	\begin{equation}\label{eqn:cas1}
	M_{r,1}(p,q):=
	\left\{\begin{array}{l}
	u:\R\times S^1 \to \widehat{W}_{k,m},\\
	z:\R \to S^{2N+1},\\
	\gamma_1:\R_+\to  Y_{k,m}\times \{r_0\},\\
	\gamma_2:\R_-\to Y_{k,m}\times \{r_0\},\\
	\end{array}	\left| \begin{array}{l}
	\partial_su+J(\partial_tu-X_{H_1}) = 0, \\
	z'+\nabla_{\widetilde{g}_N}\widetilde{f}_N=0,\\
	  \displaystyle z_-:=\lim_{s\to -\infty} z \in S^1\cdot \widetilde{z}_r,  z_+:=\lim_{s\to \infty} z \in S^1\cdot \widetilde{z}_0,\\
	  \theta_+\cdot \widetilde{z}_0=z_+, \theta_-\cdot \widetilde{z}_r=z_-, \\
	\gamma'_1-(\theta_+)_*Z_Y=0, \gamma'_2-(\theta_-)_*Z_Y=0,\\
   \displaystyle \lim_{s\to \infty} u = \gamma_1(0), \lim_{s\to -\infty} u = \gamma_2(0),\\
    \displaystyle \lim_{s\to \infty} \gamma_1 = \theta_+\cdot p,
    \displaystyle \lim_{s\to \infty} \gamma_2 = \theta_-\cdot q.
	\end{array}
	\right.
	\right\}/\R^1\times S^1
	\end{equation}
	whose expected dimension is $|q|-|p|+2r-1$. We first show that $M_{r,1}(p,q)$ is cut out transversely when the expected dimension is $\le 0$. Note that by action reason, we must have $u$ is a trivial cylinder over a periodic orbits $\gamma = \lim_{s\to \pm \infty} u$. Therefore for a flow line $z$ in $S^{2N+1}$, the moduli space is simply the intersection of $\theta_+\cdot W^u_Y(p)$ and $\theta_-\cdot W^s_Y(q)$. Since projections of flow lines $Z_Y$ are flow lines of $Z_{\Sigma}$, the same holds for $\theta_*Z_Y$ for any $\theta$. In particular, $M_{r,1}(p,q)$ projects to flow lines in $\Sigma_{k,m}$ connecting $\pi(p),\pi(q)$. In particular, to have $M_{r,1}(p,q)\ne \emptyset$ we must have $\ind \pi(p) \ge \ind \pi(q)$ and when they are equal we must have $\pi(p)=\pi(q)$, where $\ind$ is the Morse index. Since by \eqref{eqn:ind}, $|q|-|p|=\ind \pi(p)-\ind \pi(q)$ iff both $p,q$ are check or hat generators, $|q|-|p|=\ind \pi(p)-\ind \pi(q)+1$ iff $p$ is a hat generator and $q$ is a check generator and  $|q|-|p|=\ind \pi(p)-\ind \pi(q)-1$ iff $p$ is a check generator and $q$ is a hat generator, then the only possibility with $|q|-|p|+2r-1\le 0$ and $M_{r,1}(p,q)$ possibly non-empty is when $r=1$, $p=\check{p}$ and $q= \hat{p}$. In this case, the flow line $z$ is between $S^1\cdot \widetilde{z}_0,S^1\cdot \widetilde{z}_1$, whose moduli space (before modulo $\R\times S^1$) is parameterized by $\R \times S^1_{\alpha}\times S^1_{\beta}$, where $S^1_{\beta}$ is the moduli space of unparametrized flow lines over in $\CP^1$ between the maximum and minimum and $S^1_{\alpha}$ comes from the reparametrization in $S^1$. With suitable parameterization of $S^1_{\beta}$,  we can normalize to $\theta_+=\alpha,\theta_-=\alpha+\beta$. Hence after modulo the $\R\times S^1$ action, $M_{r,1}(\check{p},\hat{p})\simeq W^u_Y(\check{p})\cap S^1\cdot W^s_Y(\hat{p})=\{\check{p}\}$ and is cut out transversely. 
	
	Next we claim there are no cascades with multiple ($\ge 2$) levels from $p\in \cC(f_Y)$ to $q\in \cC(f_Y)$. Assume otherwise, since the $u$-component is necessarily a constant cylinder, the $u$ part must be parameterized by a non-trivial gradient flow line $z$. Therefore we must have $r\ge 2$ in order to have multiple cascade levels. From the same argument as before,  every element in the moduli space is projected to a flow line of $Z_{\Sigma}$ between $\pi(p),\pi(q)$. In order to $|q|-|p|+2r-1\le 0$ and $\ind \pi(p)\ge \ind \pi(q)$, we must have $r\le 1$, contradiction. This concludes that $M_{r}(p,q)$ of dimension up to zero is cut out transversely with only $M_1(\check{p},\hat{p})$ is non-empty with one element. Then $\cM_r(p,q)\backslash M_r(p,q)$ is empty when the expected dimension is at most zero. Hence $\cM^r_{p,q}$ is cut out transversely up to dimension $0$.

    \textbf{Step 2: cascades space for $\delta^r$ from $p\in \cC(f_Y)$ to $x\in \cC(f_W)$ with $r\ge 1$.} We start with the $1$-cascades case as before, namely 

	\begin{equation}\label{eqn:cas2}	
	M_r(p,x):=
	\left\{\begin{array}{l}
	u:\R^1\times S \to \widehat{W}_{k,m},\\
	z:\R \to S^{2N+1},\\
	\gamma_1:\R_+\to Y_{k,m}\times \{r_0\},\\
	\gamma_2:\R_-\to  W_{k,m},\\
	\end{array}	\left| \begin{array}{l}
	\partial_su+J(\partial_tu-X_{H_1}) = 0, \\
	z'+\nabla_{\widetilde{g}_N}\widetilde{f}_N=0,\\
	\displaystyle z_-:=\lim_{s\to -\infty} z \in S^1\cdot \widetilde{z}_r,  z_+:=\lim_{s\to \infty} z \in S^1\cdot \widetilde{z}_0,\\
	\theta_+\cdot \widetilde{z}_0=z_+,\\
	\gamma'_1-(\theta_+)_*Z_Y=0, \gamma'_2+Z_W=0,\\
	\displaystyle \lim_{s\to \infty} u =  \gamma_1(0), \lim_{s\to -\infty}u=\gamma_2(0)\\
	\displaystyle\lim_{s\to -\infty}z\in S^1 \cdot \widetilde{z}_0, \lim_{s\to \infty}z\in S^1\cdot \widetilde{z}_r,\\
    \displaystyle \lim_{s\to \infty} u = \gamma_1(0), \lim_{s\to -\infty} u = \gamma_2(0),\\
    \displaystyle \lim_{s\to \infty} \gamma_1 = \theta_+\cdot p,
    \displaystyle \lim_{s\to -\infty} \gamma_2 =  x.
    \end{array}
    \right.
    \right\}/\R^1\times S^1
	\end{equation}
	whose expected dimension is $|x|-|p|+2r-1$. Note that we do not have the rotation of $\theta_-$ anymore, as the $S^1$-reperamretrization on $W_{k,m}$ is trivial. By forgetting the $z$ component, we get a solution to the $\delta^0$-moduli space \eqref{eqn:cas4} but for $(\theta_+)_*f_Y,(\theta_+)_*Z_Y$.  Since \eqref{eqn:cas4} is cut out transversely up to dimension $0$ for the $S^1$-independent $J$, \eqref{eqn:cas4} for $(\theta_+)_*f_Y,(\theta_+)_*Z_Y$ is also cut out transversely up to dimension $0$, which we will denote by $M_0(p,x)$ by a little abuse of notation. Then in order to have $M_0(p,x)\ne \emptyset$, we must have $|x|-|p|-1\ge 0$, hence $\dim M_{r,1}(p,x)\ge 2r>0$. Therefore $M_{r,1}(p,x)=\emptyset$, in particular is cut out transversely, when the expected dimension is $\le 0$. 
	
	Then we consider the case with multiple cascades levels. Assume  the space is not empty and contains a curve of levels $(u^+_1,z^+_1),\ldots (u^+_a,z^+_a)$ in $Y_{k,m}$ and $(u^-_1,z^-_1),\ldots, (u^-_b,z^-_b)$ in $W_{k,m}$ and $(u,z)$ in between. Then $z^{\pm}_*$ must be nontrivial gradient flow lines but $z$ could rest over a critical point. Since $u^{\pm}_{*}$ is necessarily constant orbits, we can forget all the $z$ components, then $\pi(\displaystyle\lim_{s\to \infty} u)\in  W^u_\Sigma(\pi(p))$. Therefore either  $M_0(\widehat{\pi(p)},x)$ or $M_0(\widecheck{\pi(p)},x)$ is not empty, as flow lines of $Z_Y$ project to flow lines of $Z_{\Sigma}$. In order to have $M_0(\widehat{\pi(p)},x)$ or $M_0(\widecheck{\pi(p)},x)$ not empty, we must have $|x|-|\widehat{\pi(p)}|-1\ge 0$. Then $|x|-|p|+2r-1\ge 2r-1>0$. Therefore $M_{r,v}(p,x)=\emptyset$ when the expected dimension is $\le 0$ for any $v\ge 2$. Then the compactification $\cM_r(p,x)$ is also empty when the expected dimension is $\le 0$ for $r\ge 1$. The situation for $\delta^r$ on $\cC(f_W)$ is similar, and $\delta^r=0$ on $\cC(f_W)$ for $r\ge 1$. 
\end{proof}

\begin{remark}
	The proof of Proposition \ref{prop:S1} relies on that we can pick $S^1$-independent $J$ to get transversality for moduli spaces in the definition of $\delta^0$ up to dimension $0$. Then we can pick the trivial $J_N$ and all higher $M_r(p,q)$ with expected dimension $0$ projects to $M_0(p,q)$ with negative expected dimension except for one case. Since we only need to consider the Hamiltonian capturing the first period instead of all periods where $S^1$-equivariant transversality is possible, the computation of $\delta^i, i\ge 1$ in Proposition \ref{prop:S1} holds even without the monotonicity assumption $k\le m$. But the relation to Gromov-Witten invariants might break if we drop the monotonicity assumption.
\end{remark}

\begin{proof}[Proof of Theorem \ref{thm:main}]
Let $p_i$ denote the linear combination of critical points such that $[W^u_{\Sigma}(p_i)]$ representing  $[H^i]$ on $\Sigma_{k,m}$. In other words, $[W^u_{\Sigma}(p_i)]$ represents $[H^{i+1}]$ in $X_{k,m}$. $p_i$ is not necessarily a single critical point depending on whether $[H^i]$ is primitive. By Proposition \ref{prop:S1}, let $e$ denote the unique minimum of $f_W$, we have
\begin{eqnarray*}
	\delta^{S^1}(\sum_{i=0}^{k-1} (-1)^i\frac{1}{k!}\check{p}_{m-k+i}u^{-i}) &  \stackrel{Prop.\ref{prop:S1}}{=} &	\sum_{i=0}^{k-1} (-1)^i\frac{1}{k!}\delta^0(\check{p}_{m-k+i})u^{-i}+ \sum_{i=1}^{k-1} (-1)^i\frac{1}{k!}\delta^1(\check{p}_{m-k+i})u^{-i+1}\\
	& \stackrel{\eqref{eqn:d1},\eqref{eqn:GW1},\eqref{eqn:GW2}}{=} & e + \sum_{i=0}^{k-2}(-i)^i \frac{1}{k!}\hat{p}_{m-k+i+1}u^{-i} + \sum_{i=1}^{k-1} (-1)^i\frac{1}{k!}\delta^1(\check{p}_{m-k+i})u^{-i+1} \\
	& \stackrel{Prop.\ref{prop:S1}}{=} e&
\end{eqnarray*}	
and $e$  is not exact in $F^lC(H_1)^+$ for $l<k-1$. Therefore we have a $k-1$-dilation on $C(H_1)$. Therefore we have $W_{k,m}$ admits a $k-1$-dilation by Proposition \ref{prop:TFAE}.  Since for $m\gg k$, the Morse-Bott spectral sequences \cite{kwon2016brieskorn} implies that $SH^*_+(W)=0$ for $0< * < 2k-3$. Hence $W_{k,m}$ can not admits a $k-2$-semi-dilation for $m\gg k$. Note that $W_{k,m}$ admits a Lefschetz fibration with fiber $W_{k,m-1}$ by projecting along one of the variables. Then by Proposition \ref{prop:lefchetz}, $W_{k,m}$ can not admits a $k-2$-semi-dilation for every $m$. Therefore $\rD(W_{k,m})=\rSD(W_{k,m})=k-1$.
\end{proof}

\begin{remark}
From the proof of Theorem \ref{thm:main} and Remark \ref{rmk:primitive}, we know that the existence of dilation depends on the coefficient we use. $W_{k,m}$ admits a $k$-dilation if $k!$ is invertible in the coefficient when $k>\frac{m+1}{2}$ and $(k-1)!$ is invertible if $k<\frac{m+1}{2}$. In particular, we obtain another proof of $T^*S^2$ admits a dilation when the characteristic of the field is not $2$ and $T^*S^n$ admits a dilation in any ring coefficient when $n\ge 3$.
\end{remark}
We suspect that examples in Theorem \ref{thm:main} are optimal in dimension, in particular, we make the following conjecture.
\begin{conjecture}
	Let $W$ be a Liouville domain, then $\rD(W),\rSD(W)$ are either $\infty$ or smaller than $\frac{\dim W}{2}$.
\end{conjecture}
In the following, we will consider general Brieskorn varieties. Let $a_0,\ldots,a_n\in \N_{\ge 2}$, the Brieskorn variety $V_{a_0,\ldots,a_n}:=\{\sum x_i^{a_i}=1 \}$. $V_{a_0,\ldots,a_n}$ can be understood as a Weinstein filling of the Brieskorn manifold $Y_{a_0,\ldots,a_n}:=\{\sum x_i^{a_i}=0 \}\cap S^{2n+1}$. Since $V_{a_0,\ldots,a_n}$ has non-negative log-Kodaira dimension when $\sum \frac{1}{a_i}\le 1$, therefore we have the following by Corollary \ref{cor:unirule}.
\begin{proposition}\label{prop:general}
	If $\sum \frac{1}{a_i}\le 1$, then $V_{a_0,\ldots,a_n}$ does not admits a $k$-dilation for any $k$.
\end{proposition}

For general Brieskorn manifold $Y_{a_0,\ldots,a_n}$, we have the following description of the Reeb dynamics of a preferred contact form \cite[\S 5.3.1]{kwon2016brieskorn}. $T$ is a called a principle period iff there exists $I_T\subset I=\{0,\ldots,n\}$, such that $T=lcm_{i\in I_T} a_i$, $a_i|T$ iff $i\in I_T$ and $|I_T|\ge 2$. Then the periods of the Reeb flow are multiples $NT$ of principle period $T$. Given a period $A$, then $A=NT$ for a principle period $T$ such that there is no principle period $T'$ divides $NT$ and is divided by $T$. Then there is family\footnote{The family is diffeomorphic to the Brieskorn manifold determined by the sequence $I_T$.} of parametrized Reeb orbits with period $NT$ of dimension $2|I_T|-3$, which after quotient the $S^1$-reparametrization is $2|I_T|-4$ dimensional. And the generalized Conley-Zehnder index of such family is given by 
\begin{equation*}
2\sum_{i\in I_T}\frac{NT}{a_i}+2\sum_{i\notin I_T}\lfloor\frac{NT}{a_i}\rfloor+|I|-|I_T|-2NT.
\end{equation*}
If we perturb the contact form into a non-degenerate one using a Morse function on the critical manifold following \cite[Lemma 2.3]{bourgeois2002morse}, the periodic orbit with minimal Conley-Zehnder index from such family is given by the generalized Conley-Zehnder index minus half of the dimension of this family, i.e.\
\begin{equation}\label{eqn:CZ}
2\sum_{i\in I_T}\frac{NT}{a_i}+2\sum_{i\notin I_T}\lfloor\frac{NT}{a_i}\rfloor+|I|-2|I_T|-2NT+2.
\end{equation}
For each $A$, there might be several such principle $T$. Each will give rise to a Morse-Bott non-degenerate family of Reeb orbits and they describe all of the Reeb orbits.

There are many natural geometric relations between Brieskorn manifolds which can be useful to determine the order of dilation. 
\begin{enumerate}
	\item The projection along the $x_{n}$ coordinate gives a Lefschetz fibration on $V_{a_0,\ldots,a_n}$ with fiber $V_{a_0,\ldots,a_{n-1}}$.
	\item If $a_i\le b_i$ for $0\le i \le n$, then $V_{a_0,\ldots,a_n}\subset V_{b_0,\ldots,b_n}$ as an exact subdomain. This follows from the adjacency of singularities \cite[Definition 9.8, Lemma 9.9]{MR3217627}.
\end{enumerate}
Then using the basic examples in Theorem \ref{thm:main}, Proposition \ref{prop:lefchetz} and the Viterbo transfer maps,  we can determine whether $V_{a_0,\ldots,a_n}$ admits a $k$-(semi)-dilation for some $k$ for most of $a_0,\ldots,a_n$. With the following proposition, we can actually determine the number $\rD(W)$ and $\rSD(W)$ in some preferable cases. 
\begin{proposition}\label{prop:D}
	Assume the affine variety $V_{a_0,\ldots,a_n}$ admits a $k$-(semi)-dilation for some $k$. If there is a minimum Conley-Zehnder index $\mu_{\min}$ among the minimum Conley-Zehnder indices of all families, and it is attained at the the minimal period, then $\rD(W)$, repetitively, $\rSD(W)$ is $\frac{n-\mu_{\min}+1}{2}$
\end{proposition}
\begin{proof}
	Assume the minimal period is $T$, we may consider $V_{a_0,\ldots, a_n, T}$, then the new minimal Conley-Zehnder index with the minimal period $T$ is given by $\mu_{\min}+1$ by \eqref{eqn:CZ}, since now we add a $T$ into $I_T$. We claim this is also the global minimum. For other period $T'>T$, if $T$ does not divide $T'$, the minimal Conley-Zehnder index of $T'$ will increase by at least $2$ in $2\sum_{i\notin I_T}\lfloor\frac{NT}{a_i}\rfloor$ of \eqref{eqn:CZ}. If $T$ divides $T'$, then the minimal Conley-Zehnder index of period $T'$ will increase at least $2\frac{T'}{T}-1\ge 3$ since we need to include the new $T$ in $I_{T'}$. Hence the claim follows. In particular, $V_{a_0,\ldots,a_m,T,\ldots,T}$ attains unique minimal Conley-Zehnder index at the minimal period, and the gap of minimal Conley-Zehnder indices between the minimal period and other periods strictly increase with respect to the number of $T$ added.
	
	Note that the space of parameterized orbits with period $T$ is diffeomorphic to the Brieskorn manifold $Y_{I_T}$ \cite[Remark 5.3]{kwon2016brieskorn}, whose cohomology is supported in degree $0$, $|I_T|-2$, $|I_T|-1$ and $2|I_T|-3$. Therefore by the Morse-Bott spectral sequence \cite[Theorem 5.4]{kwon2016brieskorn}, if we add in enough $T$,   $V_{a_0,\ldots,a_m,T,\ldots,T}$ will have the property that $SH_+^*(V_{a_0,\ldots,a_n,T,\ldots,T})$ is supported in degree $n-\mu_{min}$ and negative degrees, since the degree gap blows up between the minimal period and other periods. By Proposition \ref{prop:lefchetz} and Theorem $\ref{thm:main},V_{a_0,\ldots,a_n,T,\ldots,T}$ supports a $m$-dilation form some $m$.  Note that $\Delta^{k}_{+,0}$ has degree $1-2k$,  we must have $\rD(V_{a_0,\ldots,a_n,T,\ldots,T})=\rSD(V_{a_0,\ldots,a_n,T,\ldots,T})=k:=\frac{n-\mu_{\min}+1}{2}$. Moreover, we have $\rSD(V_{a_0,\ldots,a_n})\ge k$ by Proposition \ref{prop:lefchetz} . By the minimal index assumption on $V_{a_0,\ldots, a_n}$, we know that $SH_+^*(V_{a_0,\ldots,a_n})$ is supported in degree below $n-\mu_{\min}$, hence $ \rD(V_{a_0,\ldots,a_n}), \rSD(V_{a_0,\ldots,a_n}) \le k$ by degree reasons if $V_{a_0,\ldots,a_n}$ admits a $m$-(semi)-dilation for some $m$. This finishes the proof.
\end{proof}
Using Proposition \ref{prop:D}, we are able to prove many examples with symplectic dilation, which are not constructed from $T^*S^2$ as in \cite{etgu2017koszul,seidel2012symplectic}, hence previously unknown.
\begin{corollary}\label{cor:1dilation}
	Let $n\ge 3$,  $a_i=i+2$ for $i\le n-2$ and $a_{n-1}=a_n=n$, then $V_{a_0,a_1,\ldots,a_n}$ admits a $1$-dilation.
\end{corollary}
\begin{proof}
	Since $V_{a_0,\ldots,a_n}\subset V_{n,\ldots,n}$, by Theorem \ref{thm:main}, $V_{a_0,\ldots,a_n}$ admits a $(n-1)$-dilation. In view of Proposition \ref{prop:D}, we need to verify the minimal Conley-Zehnder index is $n-1$ and is represented by an orbit with minimal period. When $n=3$, the minimal period is $T=3$  whose minimal Conley-Zehnder index in this period is $2\cdot 3+2+1-6-1=2$. For any other period $3N$ with $N>1$, the minimal Conley-Zehnder index is (which depends on the parity of $N$)
	$$2 \cdot 3 \cdot \frac{NT}{3}+2 \lfloor \frac{3N}{2} \rfloor + 4-6-(1+(-1)^N)-2NT+2\ge 2\lfloor\frac{3N}{2}\rfloor-2\ge 4.$$
	Hence the claim holds for $n=3$. For $n\ge 4$, the minimal period is $T=4$, the minimal Conley-Zehnder index in this period is $2\cdot(2+1)+2+n-1-8=n-1$ for $n>4$ and is $2\cdot(2+3)+2+1-8-2=3$ when $n=4$. We need to show that all other periods can not have smaller minimal Conley-Zehnder indices. Let $f(T)$ be the minimal Conley-Zehnder index formula by \eqref{eqn:CZ} even if $T$ is not period, i.e.\
	$$f(T):=2\sum_{a_i|T} \frac{T}{a_i}+2\sum_{a_i\nmid T}\lfloor \frac{T}{a_i}\rfloor + n+1 - 2|I_T|-2T+2 = 2\sum_{i=0}^n\lfloor \frac{T}{a_i}\rfloor + n+3 - 2|I_T|-2T ,$$
	where $I_T$ is subset of $\{a_0,\ldots,a_n\}$ that divides $T$. Note that if $T=NT'$ for a principle period $T'$, we have the minimal Conley-Zehnder index for period $NT'$ is no less than $f(T)$, since $I_T\supset I_{T'}$. It is direct to check that $f(4)=n-1$. Then we can compute $f(T+1)-f(T)$
	\begin{eqnarray*}
	f(T+1) - f(T) & =  & 2\sum_{i=0}^n \lfloor (\frac{T+1}{a_i}\rfloor - \lfloor \frac{T}{a_i}\rfloor) - 2|I_{T+1}|+2|I_T|-2 \\
	& = &   2 |I_{T+1}|-2|I_{T+1}|+2|I_T|-2	= 2 |I_T|-2.
	\end{eqnarray*}	
    We can compute directly
	\begin{center}
		\begin{tabular}{ |c|c|c|c|c|c|c|c|c| } 
				\hline
				$T $ & 4 & 5 & 6 & 7 & 8 & 9 & 10 & 11  \\ 
				\hline
				$f(T+1)-f(T)\ge$ & 2 & -2 & 2 & -2 & 2 & 0 & 0 & -2  \\ 
				\hline
		\end{tabular}
	\end{center}
	Since we have $2,3,4$ in $\{a_i\}$, after $11$ we can compute the lower bound of $f(T+1)-f(T)$ for different $T \mod 12$,
	\begin{center}
		\begin{tabular}{ |c|c|c|c|c|c|c|c|c|c|c|c|c| } 
			\hline
			$T \mod 12$ & 0 & 1 & 2 & 3 & 4 & 5 & 6 & 7 & 8 & 9 & 10 & 11 \\ 
			\hline
			$f(T+1)-f(T)\ge$ & 4 & -2 & 0 & 0 & 2 & -2 & 2 & -2 & 2 & 0 & 0 & -2  \\ 
			\hline
		\end{tabular}
	\end{center}
   As a consequence, for any $T>4$, we have $f(T)-f(4)\ge 0$, this proves that the minimal Conley-Zehnder index is obtained at the minimal period $T=4$. 
   \end{proof}
\begin{remark}
	The example in Corollary \ref{cor:1dilation} is not optimal, for example, $V_{2,3,5,5,5,5}$ also admits a $1$-dilation by Proposition \ref{prop:D}. By Conjecture \ref{conj} below, $V_{2,3,5}$ as the $E_8$ plumbing of $T^*S^2$ is expected to admit dilation. That $V_{2,3,5,5,5,5}$ admits dilation will also follow from the existence of dilation on $V_{2,3,5}$.
\end{remark}
In view of Proposition \ref{prop:D}, it is useful to determine when the minimal Conley-Zehnder index is represented by the minimal period. We believe this holds when $\sum \frac{1}{a_i}>1$ as it can be verified for small $n$. To support this claim, note that the remaining case can not support $k$-dilation by Proposition \ref{prop:general}. On the other hand, $\sum \frac{1}{a_i}-1$ is related with the average growth of the Conley-Zehnder indices and when it is negative, the Conley-Zehnder index is not bounded from below. Moreover, $\sum \frac{1}{a_i}> 1$ is equivalent to the singularity $x_0^{a_0}+\ldots+x_n^{a_n}=0$ is canonical. By \cite[Theorem 1.1]{mclean2016reeb}, the singularity being canonical is equivalent to that for some contact form on $\partial V_{a_0,\ldots,a_n}$, the minimal SFT degree $\mu_{CZ}+n-3$ is non-negative. Moreover by \cite{ss}, the minimal Conley-Zehnder index orbit represents non-trivial class in the positive symplectic cohomology, which is very likely to be responsible for a $k$-dilation. We hope to prove this claim and fully understand the existence of $k$-dilations of Brieskorn varieties via studying the compactification in the weighted projective spaces in future work. In view of the conjectural picture in \cite[Conjecture 6.1]{li2019exact}, we can conjecture the following special case with a precise expectation on the order of dilation.
\begin{conjecture}\label{conj}
	$V_{a_0,\ldots,a_n}$ admits a $k$-dilation iff $\sum\frac{1}{a_i}>1$, and $k$ is determined by the minimal Conley-Zehnder index of the minimal principle period.
\end{conjecture}

The following proposition provides examples with ADC boundary and admits $k$-(semi)-dilation, hence they provide supplies for results in \S \ref{s4}. 
\begin{proposition}\label{prop:ADC}
	The boundary of $\sum_{i=0}^k x_i^k+\sum_{j=1}^rx_{i+j}^{a_j}=1$ is ADC for $a_j\ge k$ and $r>0$.
\end{proposition} 
\begin{proof}
	The minimal Conley-Zehnder index of the minimal period $k$ is $r-k+3$. Hence the SFT degree is at least $2r>0$. Following the same argument in Corollary \ref{cor:dilation}, we consider the function $f(T)$. Then  $f(T+1)-f(T)\ge 2k$ if $T\mod k=0$ and $f(T+1)-f(T)\ge -2$ else, hence $f(T)-f(k)\ge 0$ for any $T>k$. In particular, the minimal SFT degree is $2r>0$.
\end{proof}

The study of exotic contact structure on the almost contact manifold $(S^{2n-1},J_{std})$ for $n\ge 3$ has a long history. By the result of Eliashberg-Floer-McDuff \cite{mcduff1991symplectic}, any contact manifold representing $(S^{2n-1},J_{std})$ with a Liouville filling not diffeomorphic to the ball is an exotic one. And Floer theoretic invariants were used to show that there are infinitely many different exotic contact structures \cite{lazarev2016contact,ustilovsky1999infinitely}.  Moreover, there are inner hierarchies in the examples found by Lazarev \cite{lazarev2016contact} as explained to us by the referee as follows. Let $V_{flex}$ be one of the flexible Weinstein domains whose contact boundary is $(S^{2n-1},J_{std})$ as almost contact manifold used in \cite[Corollary 1.4]{lazarev2016contact}, then we can only have Weinstein cobordisms from $\#^i \partial V_{flex}$ to $\#^{i+1} \partial V_{flex}$ but not the other way around. Assume otherwise, if there is an Weinstein cobordism from $W$ from $\#^{i+1} \partial V_{flex}$ to $\#^i \partial V_{flex}$, then $W\cup \natural^{i+1}V_{flex}$ is an Weinstein filling of $\#^i \partial V_{flex}$. Then by \cite[Corollary 1.3]{MR4182808}, $H^*(W\cup \natural^{i+1}V_{flex})$ is isomorphic to $H^*(\natural^{i}V_{flex})$. However we must have $\rank H^n(W\cup \natural^{i+1}V_{flex})\ge \rank H^n(\natural^{i+1}V_{flex})> \rank H^n(\natural^{i}V_{flex})$, contradiction. Moreover, the same argument can be applied to show that there is no exact cobordism with vanishing first Chern class from $\#^{i+1} \partial V_{flex}$ and $\#^i \partial V_{flex}$ using that $\#^{i+1} \partial V_{flex} = S^{2n-1}$ and the Mayer-Vietoris sequence. The argument for non-existence of Weinstein cobordism can be applied to any almost Weinstein fillable contact manifold $(Y,J)$ of dimension at least $5$. However, the argument for the non-existence of exact cobordisms with vanishing first Chern class (possibly with additional constraints on $\pi_1$) relies on that $H^*(Y)$ is sparse. Using the order of semi-dilation, we can get another hierarchy of exoticity which does not depend on the topology of $Y$ with the drawback that it is only finite in a fixed dimension. 
\begin{corollary}
	For every $k\in \N$, there exists $n$ such that for any simply connected almost Weinstein fillable contact manifold $(Y^{2n-1},J)$, there exist contact manifolds $Y_1,\ldots,Y_k$ such that the following holds.
	\begin{enumerate}
		\item $Y_i = (Y,J)$ as almost contact manifolds.
		\item There is a Weinstein cobordism from $Y_i$ to $Y_j$ if $i<j$.
		\item There is no exact cobordism with vanishing first Chern class  from $Y_i$ to $Y_j$ for $i>j$.
	\end{enumerate}
\end{corollary}
\begin{proof}
	We first construct the example on $(S^{2n-1},J_{std})$. Let $n=2m+1\gg 0$ and two different prime numbers $p,q \gg 0$, then by \cite[Proposition 3.6]{kwon2016brieskorn}, we have that $\partial V_{i+1,\ldots,i+1,p,q}$ is a homotopy sphere of dimension $2n-1$ for every $1 \le i\le k$. Moreover, by Proposition \ref{prop:lefchetz} and \ref{prop:D}, we have $\rSD(V_{i+1,\ldots,i+1,p,q})=i$. Let $N:=|bP_{2n}|\cdot |\pi_{4m+1}(SO(4m+1)/U(2m))|$, where $bP_{2n}$ is the group of boundary parallelizable homotopy sphere of dimension $2n-1$. Then the contact connected sum $S_i:=\#^N\partial V_{i+1,\ldots,i+1,p,q}$ is $(S^{2n-1},J_{std})$, see \cite[Theorem 3.12]{kwon2016brieskorn} and \cite{ustilovsky1999infinitely}. By Proposition \ref{prop:flex}, we have $\rSD(\natural^N V_{i+1,\ldots,i+1,p,q})=i$. Moreover, by construction, we have $\natural^NV_{i+1,\ldots,i+1,p,q}\subset \natural^N V_{i+2,\ldots,i+2,p,q}$. Hence there is a Weinstein cobordism from $S_i$ to $S_{i+1}$, but there is no exact cobordism with vanishing first Chern class from $S_{i+1}$ to $S_i$ by Corollary \ref{cor:cob}. Now let $\widetilde{Y}$ be the contact boundary of the flexible Weinstein filling and is $(Y,J)$ as almost contact manifolds. We define $Y_i=S_i\# \widetilde{Y}$, then there is a Weinstein cobordism from $Y_i$ to $Y_{i+1}$ and $Y_i$ has a Weinstein filling $W_i$ such that $\rSD(W_i)=i$ by Proposition \ref{prop:flex}. Then claim for $Y_i$ follows from Corollary \ref{cor:cob}. 
\end{proof}

\bibliographystyle{plain} 
\bibliography{ref}

\end{document}